\documentclass[11pt]{article}

\usepackage[T1]{fontenc}
\usepackage[utf8]{inputenc}
\usepackage[margin=1in]{geometry}
\usepackage{amsthm,amsmath,amsfonts,amssymb}
\usepackage[authoryear]{natbib}
\usepackage{xifthen,bbm}
\usepackage{bm,comment}
\usepackage{esvect}
\usepackage[colorlinks,citecolor=blue,urlcolor=blue,linkcolor=blue]{hyperref}
\usepackage{mathrsfs,xcolor}
\usepackage{subfig}
\usepackage{graphicx}
\usepackage{algorithm,float}
\usepackage{csquotes}
\usepackage{algpseudocode}
\allowdisplaybreaks
\setcitestyle{round}
\newtheorem{thm}{Theorem}[section]

\newtheorem{remark}[thm]{Remark}

\newtheorem{case}[thm]{Case}
\newtheorem{lem}[thm]{Lemma}
\newtheorem{theorem}[thm]{Theorem}

\newtheorem{proposition}[thm]{Proposition}
\theoremstyle{definition}
\newtheorem{definition}[thm]{Definition}

\renewcommand{\ge}{\geqslant}
\renewcommand{\le}{\leqslant}

\newcommand{\eps}{{\varepsilon}}

\newcommand{\tm}{{\widetilde{m}}}
\newcommand{\e}{{\mathbb{E}}}

\newcommand{\onb}{{\overline{\nabla}}}

\newcommand{\R}{{\mathbb{R}}}

\newcommand{\p}{{\mathbb{P}}}

\def\A{\bm{A}}
\def\R{\mathbb{R}}

\DeclareMathOperator*{\argmax}{argmax}

\title{Joint Estimation in Potts Model}
\author{Somabha Mukherjee\thanks{Department of Statistics and Data Science, National University of Singapore, Singapore. Email: \texttt{somabha@nus.edu.sg}}\and
Sumit Mukherjee\thanks{Department of Statistics, Columbia University, USA. Email: \texttt{sm3949@columbia.edu}}\and
Sayar Karmakar\thanks{Department of Statistics, University of Florida, USA. Email: \texttt{sayarkarmakar@ufl.edu}}}
\date{}

\begin{document}
\maketitle

\begin{abstract}
In this paper, we study estimation of parameters in a two-parameter Potts model with $q$ colors and coupling matrix $\A_N$. We characterize concrete sufficient conditions for existence of the pseudo-likelihood estimator of the Potts model, in terms of the local magnetic fields, and give sufficient conditions for the validity of the above characterization. We then provide sufficient criteria for estimation of both parameters at the optimal rate $\sqrt{N}$. In particular, if $\A_N$ is the scaled adjacency matrix of a graph $G_N$, then we show that joint estimation is possible if either $G_N$ has bounded degree or is irregular. In contrast, we give an example of a graph sequence $G_N$ which is approximately regular and dense, where no consistent estimator exists. We also show that one-parameter estimation at the optimal rate $\sqrt{N}$ holds under much milder conditions when the other parameter is known. Along the way, we develop a concentration result for mean-field Potts models using the framework of nonlinear large deviations. Compared to the Ising case, our results for the Potts case require a novel analysis across multiple colors.
\end{abstract}


\section{Introduction}\label{sec:intro}
The \emph{Potts model}, whose origin can be traced back to the 1900s (see \cite{AshkinTeller1943}), 
is a statistical physics model for capturing dependence in complex stochastic systems. 
What began as a generalization of the \emph{Ising model} (see \cite{ising}) in order to accommodate spins with more than  two values (see \cite{potts, wu1982potts}) has, over the past several decades, found widespread applications in a number of diverse fields including biomedical problems
\citep{Boas2018PottsBio,Moltchanova2005PottsHaplotype},
image processing and computer vision
\citep{Celeux2002PottsSegmentation,Levada2009PottsImage},
spatial statistics \citep{Zukovic2008SpatialPotts},
social sciences \citep{Bosconti2015SocialPotts}, finance \citep{Takaishi2005PottsFinance,Bornholdt2021PottsMarkets} and automata theory \citep{Graner1992CellularPotts}, among others. 

The $q$-state Potts model, for any positive integer $q$ can be described as a discrete probability distribution supported on the set $[q]^N$. Here and henceforth, the notation $[m]$, for any $m \in \mathbb{N}$, denotes the set $\{1, 2, \ldots, m\}$. The positive integer $N \in \mathbb{N}$ indicates the size of the system (number of interacting particles in the system) under consideration. This distribution is given by the probability mass function
\begin{equation}\label{eq:pmf}
    \p_{\beta,\bm {B}}(\bm {x}) := \frac{1}{Z_N(\beta,\bm {B})} \exp\left(\frac{\beta}{2} \sum_{1\le i,j\le N} a_{ij}\mathbbm{1}_{x_i=x_j} + \sum_{i=1}^N \sum_{r=1}^{q} B_r\mathbbm{1}_{x_i=r}\right) \text{ for } \bm x \in [q]^{N},
\end{equation}
where $\beta>0$ represents the \emph{inverse temperature}, $\bm {B} := (B_1,\ldots,B_{q-1})\in \mathbb{R}^{q-1}$ represents the \emph{magnetic field} vector, and $\bm A_N := ((a_{ij}))_{1\le i,j\le N}$ is a symmetric matrix,  with zeros on the diagonal. We will refer to $\bm A_N$ as the \emph{coupling/interaction matrix}. Note here that we did not include a non-zero magnetic field parameter $B_q$ for identifiability reasons, since otherwise, the model remains unchanged if the same constant is added to all the magnetic fields. Throughout the paper, we will use by convention the notation $B_q := 0$.
Some of the commonest examples of coupling matrices are suitably scaled adjacency matrices of graphs, defined via
\begin{equation}\label{adjacency_coupling_defn}
a_{ij} := \frac{N}{2|E(G_N)|} \mathbbm{1}(i~\text{and}~j~\text{form an edge in}~G_N), \text{ for all } i, j \in [N],
\end{equation}
where $G_{N}$ is any graph on vertex set $[N]$ and edge set $E(G_N)$. 

Equation \eqref{eq:pmf} takes the form of a discrete exponential family with natural parameters $\beta>0$ and $\bm B\in \R^{q-1}$. The problem we address in our paper is the estimation of these parameters given a \emph{single} sample $\bm X := (X_1,\ldots,X_N)$ from this model. The unavailability of multiple, mutually independent random vectors sampled from the same distribution (as is typical in epidemics, elections, or criminal activity, where the underlying network is typically observed
only once, and replications are spatio-temporally dependent) is what poses the primary challenge in this problem. Throughout this paper, we assume that all entries of the matrix $\bm A_N$ are non-negative and completely known. 

The choice $q=2$ corresponds to the Ising model, and in this special case, changing the domain of ${\bm x}$ from $[2]^N=\{1,2\}^N$ to $\{\pm 1\}^N$, one can write
\[\sum_{i,j=1}^na_{ij}1\{x_i=x_j\}=\frac{1}{2}\sum_{i,j=1}^N a_{ij} (1+x_ix_j),\quad \sum_{i=1}^N1\{x_i=1\}=\frac{N}{2}(1+\bar{x}).\]
Plugging these in \eqref{eq:pmf}, the pmf of the Ising model on $\{\pm 1\}^N$ can thus be written as
\begin{equation}
\p_{\beta,B}(\bm {X} = \bm {x}) \propto \exp\left\{\frac{\beta}{4} \bm{x}'\bm A_{N}\bm{x} + \frac{B_1}{2}\sum_{i=1}^{N}x_{i}\right\}, \text{ for } \bm x = (x_{1}, \ldots, x_{N}) \in \{\pm 1\}^{N},\label{ising_case_eq}
\end{equation}
Statistical inference for general Ising and Potts models traces back to the seminal work \cite{estimation_15}, which analyzed the Ising model \eqref{ising_case_eq} in the absence of an external field ($B_1=0$). Allowing the coupling matrix $\bm A_N$ to have both positive and negative entries, under bare minimal conditions \cite{estimation_15} establishes $\sqrt{N}$-consistency of the \emph{maximum pseudo-likelihood estimator} of the natural parameter $\beta$ of this one-parameter exponential family. In particular, 
 the results of this paper apply to the well known spin glass models such as 
the celebrated 
Sherrington-Kirkpatrick model, and the Hopfield model of neural networks. One of the many open questions raised in \cite{estimation_15} speculates whether the methods developed in \cite{estimation_15} can be adapted for estimation in multi-parameter models. As an answer to this, \cite{ghosal2020joint} considers the two-parameter Ising model in \eqref{ising_case_eq} when the coupling matrix $\bm A_{N}$ has non-negative entries, and studies the joint estimation of the inverse temperature parameter and the magnetization parameter, i.e. the pair \ $(\beta, B_1)$. In this paper, we will study the analogous question of joint estimation of $(\beta,{\bf B})$ for the more general Potts model \eqref{eq:pmf}. 


\subsection{Literature Review}

The problem of statistical inference in statistical physics models has a body of growing literature, and here we cite some of the relevant literature close to our work. Some of the earliest rigorous studies for the Curie--Weiss Ising model were done in \cite{ellis1980limit, ellis1985entropy}, which established the CLT for the magnetization for the Curie-Weiss Ising model (\eqref{ising_case_eq} with $\A_N(i,j)=N^{-1}\mathbbm{1}_{i\ne j}$). Subsequently, \cite{comets} studied asymptotics of the MLE in the Curie-Weiss Ising model, and showed that one-parameter estimation is possible if the other parameter is known. Going beyond the Curie-Weiss Ising model in a significant way, \cite{estimation_15} studies the performance
of the pseudo-likelihood estimator for the one-parameter Ising model  with the temperature parameter $\beta>0$ unknown for a general matrix $\A_N$, with the magnetization parameter $B_1=0$. In this paper the author allows the coupling matrix $\A_N$ to have both positive and negative values, and gives a sufficient criterion for estimation of $\beta$ at the optimal rate $\sqrt{N}$, which covers both graphical models as well as spin glass models. In a follow up work, in \cite{estimation_14} the authors extend this to show that the estimation rate for $\beta$ depends on the order of the log normalizing constant $\log Z_N(\cdot,{\bm B})$ in a local neighborhood of the truth $\beta$. Using this, they demonstrate phase transitions in the rate of the pseudo-likelihood estimator, which is typically dictated by the critical temperature of the Ising model. The problem of estimating both the parameters $(\beta,B_1)$ in the Ising model was first studied in \cite{ghosal2020joint}, where the authors assume that the coupling matrix $\A_N$ is non-negative entry-wise. Under this assumption, \cite{ghosal2020joint} show that joint estimation is possible at the optimal rate $\sqrt{N}$ if either the coupling matrix is irregular (see \eqref{irregap}) or non-mean field (see \eqref{bddegr}). In contrast, if the coupling matrix is both regular and mean-field, they give an example to show that joint consistent estimation may be impossible. In a more recent paper \cite{ChenSenWu2024o}, the authors show that joint estimation is possible for spin glass models, where the coupling matrix $\A_N$ can take both positive and negative values.

\textcolor{black}{Prior to this work, a number of studies have explored statistical inference in Potts models and more general Markov random fields (see, for example, \cite{estimation_1, estimation_2, estimation_3, estimation_4, estimation_5, estimation_6, estimation_7, estimation_8, estimation_9, estimation_10, estimation_11}). While these contributions provide valuable insights and methodological developments, a fully rigorous treatment of consistency for joint parameter estimators in general Potts models with $q>2$ colors has not yet been established. To the best of our knowledge, the present work is the first to address this question. In fact, even in the single-parameter setting, rigorous results on consistent estimation in the Potts framework are largely absent, with the notable exception of the Curie--Weiss Potts model (see \cite{ellis1992, snbh1, snbh2}). As indicated above, studying general Potts models requires the development of new analytical tools tailored to Potts models, which we expect will also be useful for future investigations.}

\textcolor{black}{A natural motivation for our work arises from the extensive literature on \emph{exponential random graph models} (ERGMs), which can be thought of as analogues of the Ising model with higher dimensional tensors. Sampling from ERGMs plays a central role in both parameter estimation and hypothesis testing, and Glauber dynamics provide a standard and widely used approach for this purpose. The mixing properties of Glauber dynamics in ERGMs have been studied in several key works, including \cite{bhamidi2011mixing} and \cite{DeMuseEaslickYin2019}, which demonstrate interesting phase transition properties in the mixing rate. In fact, even in the specialized Curie-Weiss Ising model, mixing rates can be either polynomial or exponential depending on the parameter regime. For details, we refer the interested reader to \cite{LevinLuczakPeres2010, DingLubetzkyPeres2009, SamantaMukherjeeZhang2024} and references therein. For the Curie--Weiss Potts model, \cite{HeLok2025} have studied mixing rate for Glauber dynamics, whereas \cite{eichels}, \cite{ellis1990limit} and \cite{gandolfo} study CLT for the magnetization. On the inferential side, the problem of parameter estimation in ERGMs has also received significant attention which demonstrates challenges of their own; see, for instance, \cite{egrm_2, MukherjeeXu2023, egrm_4}.}
\textcolor{black}{Similar to the Potts case, the most well-studied tensor for higher order binary models is the complete tensor case (p-spin Curie Weiss model), which has been studied recently in  \cite{tensor1, tensor2, tensor3, tensor4} using the perfect symmetry of the complete tensor.} 

\textcolor{black}{Going in a different direction, another question of interest is the problem of structure learning, i.e.~to recover the whole graph/matrix $\A_N$, which is a high-dimensional parameter estimation problem. Indeed, in this case one Ising/Potts sample will not suffice, and one needs access to i.i.d.~samples. In this setting,  \cite{estimation_13,estimation_16,comp_2, lokhov1, lokhov2} study graph recovery and support recovery, and establish tight sample complexity bounds, for Ising models. 
Other questions of interest for Ising-type models include community detection on SBM \citep{berthet2019exact}, property testing  \citep{neykov2019property}, and structure detection \citep{cao2022hightemperature}.}

\subsection{Our contributions}

In this paper we study bivariate estimation of parameters in a Potts model with $q$ colors, using the pseudo-likelihood method of \cite{besag1974spatial,besag1975statistical}. Prior to our work, the existing literature focuses exclusively on the Ising case ($q=2$), or on the Curie--Weiss Potts case. Going from the Ising to general Potts case requires us to investigate conditions under which the pseudo-likelihood estimator exists (see \eqref{psiNdef96}). The exact characterization is delicate for $q>2$ colors, more so because the characterization for $q=2$ in \cite[Theorem 1.2 (a)]{ghosal2020joint} is not entirely correct. The correct characterization in the Ising case was established recently in \cite{ChenSenWu2024o}, and in this work we establish the corresponding result for the Potts model. In particular, we require that there exist two colors for which the corresponding local fields are well separated (see Theorem \ref{thm:theorem1} for details). Another challenge is the characterization of the subset of the parameter space for the Curie-Weiss Potts model where the local magnetization vector has $\sqrt{N}$ fluctuation, in terms of the Hessian of the variational objective $H_{\beta,\bm B}(\cdot)$ (see \eqref{Hdefn6882}). This is carried out in Lemma \ref{c222}, utilizing tools from linear algebra, coupled with a careful application of the inverse function theorem. This lemma is crucially used to show non-existence of consistent estimators for Potts models on dense Erd\H{o}s-R\'enyi graphs. Showing that the estimation is possible at the optimal rate of $\sqrt{N}$ in the irregular case (Theorem \ref{irrgr}) is more delicate for the Potts case with $q>2$ colors. A fine analysis is needed to show that the RHS of \eqref{fstpp3} is strictly positive, which translates into a variation bound for the gradient of the free energy function $\psi_N$ (see \eqref{psiNdef96}) from its average.
But perhaps most significantly, utilizing the non-linear large deviations framework developed in \cite{chatterjee2016nonlinear} and \cite{mukherjeebasak}, in this paper we  develop a concentration result for mean-field Potts models (see Lemma \ref{epnt376}). This result shows that the local fields for all colors are close to the optimizers of the variational problem resulting from the non-linear large deviations. This is of possible independent interest, particularly if one wants to go beyond the law of large numbers, and study a CLT under Potts models.

\subsection{Main Results}
In this section, we state the main results of this paper. As mentioned above, our main goal is to derive a consistent estimator of the parameter $(\beta,\bm B)$, when a single vector ${\bm X}$ is observed from the model \eqref{eq:pmf}. The classical method of maximum likelihood (ML) estimation is not practical in this  framework, because of the presence of the intractable normalizing constant $Z_N(\beta,\bm B)$, which is hard to compute and difficult to approximate using MCMC techniques; see \cite{bhamidi2011mixing}. A computationally efficient alternative in the literature \cite{besag1974spatial, besag1975statistical, estimation_15, estimation_14, ghosal2020joint, daskalakis2020logistic} is to consider the maximum pseudo-likelihood (MPL) estimator, given by:

\begin{equation*}\label{mpl_defcom}
    (\hat{\beta}_N, \hat{\bm B}_N) := \argmax_{(\beta,\bm B)\in \R^q} L_N(\beta,\bm B) := \argmax_{(\beta,\bm B)\in \R^q}\prod_{i=1}^N \p_{\beta, \bm B}(X_i|(X_j)_{j\ne i}) 
\end{equation*}
 provided the pseudo-likelihood function $L_N$ has a unique maximizer. Indeed, the conditional distribution of $X_i$ given $(X_j)_{j\ne i}$ is easy to compute, and is given by:
\begin{equation}\label{defcondprob881}
    \p_{\beta,\bm B}(X_i=r|(X_j)_{j\ne i})=\frac{\exp\left\{\beta m_{i,r}(\bm X) +   B_{r} \right\}}{\sum_{s=1}^{q} \exp\left\{\beta m_{i,s}(\bm X) + B_{s}\right\}}=:\theta_{i,r}({\bm X}) 
\end{equation}
where $m_{i,r}(\bm x) := \sum_{j=1}^N a_{ij}\mathbbm{1}_{x_j=r}$ for $\bm x\in [q]^N$. We will often drop ${\bm X}$ from the notation $\theta_{i,r}$ for simplicity.
The pseudo-likelihood function $L_N$ is thus given by:
$$L_N(\beta,\bm B) := \frac{\exp\left\{\beta\sum_{i=1}^{N}\sum_{r=1}^q m_{i,r}(\bm X)\mathbbm{1}_{X_i=r} + \sum_{i=1}^{N} \sum_{r=1}^{q} B_{r} \mathbbm{1}_{X_{i} =r}\right\}}{\prod_{i=1}^{N} \sum_{r=1}^{q} \exp\left\{\beta m_{i,r}(\bm X) + B_{r} \right\}}$$
and hence, the log pseudo-likelihood function is given by:
\begin{align}\label{eq:ln}
&\ell_N(\beta,\bm B) :=\notag\\ &\beta\sum_{i=1}^{N}\sum_{r=1}^q m_{i,r}(\bm X)\mathbbm{1}_{X_i=r} + \sum_{i=1}^{N} \sum_{r=1}^{q} B_{r} \mathbbm{1}_{X_{i} =r} - \sum_{i=1}^N \log\left(\sum_{r=1}^{q} \exp\left\{\beta m_{i,r}(\bm X) + B_{r} \right\}\right).
\end{align}
The MPL estimator can be obtained by setting the partial derivatives of $\ell_N$ to $0$, which in turn requires the exact expressions of these partial derivatives. Towards this, we have:
\begin{equation}\label{parbeta}
\frac{\partial \ell_N(\beta,\bm B)}{\partial \beta} = \sum_{i=1}^{N}\sum_{r=1}^q m_{i,r}(\bm X)\mathbbm{1}_{X_i=r} - \sum_{i=1}^N \frac{\sum_{r=1}^q m_{i,r}(\bm X) \exp\left\{\beta m_{i,r}(\bm X) + B_{r} \right\}}{\sum_{r=1}^q \exp\left\{\beta m_{i,r}(\bm X) + B_{r} \right\}}~,
\end{equation}

\begin{equation}\label{parB}
\frac{\partial \ell_N(\beta,\bm B)}{\partial B_s} = \sum_{i=1}^{N} \mathbbm{1}_{X_i=s} - \sum_{i=1}^N \frac{\exp\left\{\beta m_{i,s}(\bm X) +B_{s} \right\}}{\sum_{r=1}^q \exp\left\{\beta m_{i,r}(\bm X) + B_{r} \right\}}\quad(1\le s\le q-1)~.
\end{equation}
Henceforth, we will call the equation $\nabla L_N(\beta,\bm B) = \boldsymbol 0$,
the pseudo-likelihood equation. Of course, if the MPL estimator $(\hat{\beta}_N, \hat{\bm B}_N)$ exists, then it is a solution of the pseudo-likelihood equation.

Before stating our first main result about the behavior of the MPL estimator, we introduce two assumptions on the coupling matrix  $\bm A_N$ that we will assume throughout the rest of the paper:
\begin{align}\label{as1}
  \sup_{N\ge 1} \|\bm A_N\|_1=\sup_{N\ge 1}~\max_{i\in [N]} \sum_{j=1}^N a_{ij} =: \gamma < \infty,\\
\label{as2}
    \liminf_{N\to\infty} \frac{{\bf 1}'\bm A_N {\bf 1}}{N} =\liminf_{N\rightarrow \infty} ~\frac{1}{N}\sum_{1\le i,j\le N} a_{ij} >0.
\end{align}
Here $\|.\|_1$ denotes the $\ell_1$ operator norm of a matrix, and ${\bf 1}$ is the constant vector of size $N$ with all entries $1$. These conditions are standard in the literature for inference in Ising models, which corresponds to the case $q=2$ (see Eq. (1.2) and (1.3) in \cite{ghosal2020joint}; also see \cite{deb2020detecting, mukherjee2018global}). Note that when $\bm A_N$ is the scaled adjacency of a graph \eqref{adjacency_coupling_defn}, condition \eqref{as1} becomes equivalent to the maximum degree $d_{\max}(G_N)$ of $G_N$ being of the same order as its average degree $\bar{d}(G_N) := \frac{1}{N}\sum_{i=1}^N d_i(G_N)$, where $d_i(G_N)$ denotes the degree of the vertex $i$ in $G_N$ (i.e. $d_{\max}(G_N) = O(\bar{d}(G_N))$). What this essentially says, is that there is no vertex in the graph with \textit{atypically high} degree. Condition \eqref{as2} is always true in this case, and in fact, one has
$$\frac{1}{N}\sum_{1\le i,j\le N} a_{ij} = 1.$$

We are now ready to state the first main result of this paper, which gives an upper bound to the estimation error in terms of the quantity $T_N(\bm x)$ defined as:

 \begin{equation}\label{eq:Tndefn}
    	T_N(\bm x):=\sum_{1 \leqslant  r<s \leqslant  q} \left( \frac{1}{N}\sum_{i=1}^N (m_{i,r}(\bm x)-m_{i,s}(\bm x))^2 - (\overline{m}_r(\bm x)-\overline{m}_s(\bm x))^2\right).
    \end{equation}
    where $\overline{m}_r(\bm x) := N^{-1}\sum_{i=1}^N m_{i,r}(\bm x)$. 

\begin{theorem}\label{thm:theorem1}
Suppose $\bm X$ is a sample from the Potts model \eqref{eq:pmf}, where the coupling matrix $\bm A_N$ has non-negative entries, and satisfies conditions \eqref{as1} and \eqref{as2}. If $(\beta,\bm B) \in \Theta := (0,\infty)\times \mathbb{R}^{q-1}$, then the following conclusions hold:
\begin{enumerate}
    \item[(a)] The MPL estimator $(\hat{\beta}_N, \hat{\bm B}_N)$ exists if $\bm X \in \Omega_N \bigcap \Lambda_N$, where 
    \begin{align*}
  \Lambda_N :=&\{{\bm y}\in [q]^N:\text{ for every }r\in [q] \text{ there exists }i\in [N], \text{ such that }y_i=r\},\\
    \Omega_N :=&
     \{\bm y \in [q]^N: \text{there exist}~1\le r < s \le q~\text{and}~ 1\le i,j,k,
    \ell~\text{all distinct, such that}\\&~\{y_i,y_j\} = \{r,s\}=\{y_k,y_\ell\}, \{\tm_i^{r,s}({\bm y}), \tm_j^{r,s}({\bm y})\} < \{\tm_k^{r,s}({\bm y}),\tm_\ell^{r,s}({\bm y})\}\}
     \end{align*}
     where
      $\tm_u^{r,s}(\bm y) := m_{u,r}(\bm y) - m_{u,s}(\bm y).$
\vspace{0.1in}
    \item[(b)] If $T_N(\bm X)^{-1}=o_\p(\sqrt{N})$ and the MPL estimator exists,
then 
\begin{eqnarray*}\label{eq:consistency}
\|(\hat{\beta}_N-\beta, \hat{\bm B}_N-\bm B) \|_2=O_{\p}\left(\frac{1}{\sqrt{N}T_N(\bm X)}\right).
\end{eqnarray*}

    \item[(c)]
    In particular, if \textcolor{black}{$T_N(\bm X)^{-1}=O_\p(1)$}, 
    then 
    \begin{equation}\label{joincg5}
        \p_{\beta,\bm B} \left(\bm X \in \Omega_N\bigcap \Lambda_N\right) \rightarrow 1~\text{as}~N\rightarrow \infty.
    \end{equation}
   Consequently, the MPL estimator $(\hat{\beta}_N, \hat{\bm B}_N)$ exists with probability tending to $1$, and satisfies
    \begin{equation}\label{joincg6}
        \|(\hat{\beta}_N-\beta, \hat{\bm B}_N-\bm B) \|_2=O_{\p}\left(\frac{1}{\sqrt{N}}\right).
    \end{equation}

\end{enumerate}
\end{theorem}
\vspace{0.1in}

\begin{remark}\label{refcorl1}
Note that we are able to prove the existence of the joint MPL estimator 
$(\hat{\beta}_N,\hat{\bm B}_N)$ (with high probability) only in the regime $T_N(\bm X)^{-1} = O_\p(1)$, 
but not in the entire regime $T_N(\bm X)^{-1} = o_\p(\sqrt{N})$. 
This suffices to guarantee the $\sqrt{N}$-consistency of the MPL estimator whenever 
$T_N(\bm X)^{-1} = O_\p(1)$ (part (b) of Theorem~\ref{thm:theorem1}). 
In particular, this setting covers the cases where $\bm A_N$ is the adjacency matrix 
of a sequence of bounded-degree graphs (see Section~\ref{bdegr}) or asymptotically 
irregular graphs (see Section~\ref{sec:irregnm}). 
Part (c) extends the result to the full regime $T_N(\bm X)^{-1} = o_\p(\sqrt{N})$, 
though only under the additional assumption that the MPL estimator exists in this regime.
\end{remark}

The proof of Theorem \ref{thm:theorem1} is given in Section \ref{sec:int}. We now study the two most general types of interaction structures to which the joint consistency result, Theorem~\ref{thm:theorem1}, applies. In fact, in both these cases, $T_N(\bm X)^{-1} = O_\p(1)$, and hence, the joint MPL estimator is $\sqrt{N}$-consistent.

\subsubsection{Non mean-field interactions}\label{bdegr} Throughout this subsection, we will assume that:
\begin{equation}\label{bddegr}
    \liminf_{N\rightarrow \infty} \frac{1}{N}\sum_{1\le i,j\le N} a_{ij}^2 > 0~.
\end{equation}
Condition \eqref{bddegr} is often referred to as the \textit{non mean-field condition}. Note that if the coupling matrix is the scaled adjacency of a graph, then condition \eqref{bddegr} simply means that the average degree of the graph is bounded. This, coupled with condition \eqref{as1} implies that the maximum degree of the graph is bounded. The following theorem shows that the joint MPL estimator is $\sqrt{N}$-consistent for the Potts model with interaction matrix $\bm A_N$ satisfying \eqref{bddegr}.

\begin{theorem}\label{bdgr}

\textcolor{black}{Suppose $\bm X$ is an observation from the Potts model \eqref{eq:pmf} where the interaction matrix $\bm A_N$ satisfies} the conditions \eqref{as1}, \eqref{as2} and \eqref{bddegr}. Then,
$$\|(\hat{\beta}_N-\beta,\hat{\bm B}_N - \bm B)\|_2 = O_\p\left(\frac{1}{\sqrt{N}}\right)~.$$
\end{theorem}

The proof of Theorem \ref{bdgr} is given in Section \ref{sec:bdgrproof} of the appendix. As mentioned above, if the underlying interaction structure is the adjacency matrix of a deterministic graph scaled appropriately \eqref{adjacency_coupling_defn}, then Theorem \ref{bdgr} applies as long as the graph is of bounded degree and $\liminf_{N\rightarrow \infty} \bar{d}(G_N)>0$. This covers as special cases, the classical Ising models on lattices, that have finite-range interactions, and $d$-regular graphs with $d$ fixed.

\subsubsection{Irregular interactions}\label{sec:irregnm}
Throughout this subsection, we will assume that:
\begin{equation}\label{irregap}
    \liminf_{N\rightarrow \infty} \frac{1}{N}\sum_{i=1}^N (R_i - \bar{R})^2 > 0
\end{equation}
where $R_i := \sum_{j=1}^N a_{ij}$ and $\bar{R} := \frac{1}{N}\sum_{i=1}^N R_i$.
Note that if the coupling matrix is the scaled adjacency of a graph \eqref{adjacency_coupling_defn}, then Condition \eqref{irregap} says that the graph is asymptotically irregular. The following theorem shows that the joint MPL estimator is $\sqrt{N}$-consistent for the Potts model with interaction matrix $\bm A_N$ satisfying \eqref{irregap}.

\begin{theorem}\label{irrgr}
Suppose that $\bm B\ne \boldsymbol{0}$, and \textcolor{black}{$\bm X$ is an observation from the Potts model \eqref{eq:pmf} where the interaction matrix $\bm A_N$ satisfies} the conditions \eqref{as1}, \eqref{as2} and \eqref{irregap}. Then, 
$$\|(\hat{\beta}_N-\beta,\hat{\bm B}_N - \bm B)\|_2 = O_\p\left(\frac{1}{\sqrt{N}}\right)~.$$
\end{theorem}

The proof of Theorem \ref{irrgr} is given in Section \ref{sec:irrgrproof} of the appendix. Common examples of interaction structures satisfying all the necessary assumptions of Theorem \ref{irrgr} are the scaled adjacencies of the complete bipartite graph $K_{m,n}$ and a disjoint union of the cliques 
{$K_m$ and $K_n$} where $N=m+n$ and $\frac{m}{N} \rightarrow \alpha \in (0,1) \setminus \{\frac{1}{2}\}$ as $N\rightarrow \infty$. In general, it follows from the theory of graphons (see \cite{lovasz2012} for a survey on graph limit theory and the literature of graphons) that if $G_N$ is a sequence of dense graphs converging to a graphon $W$  such that the function $x\mapsto \int_0^1 W(x,y) dy$ is not constant Lebesgue almost everywhere, then all the necessary assumptions of Theorem \ref{irrgr} are satisfied. This includes the above two examples as special cases, as well as dense stochastic block models on $N$ nodes with two communities $C_1$ and $C_2$ of sizes $m$ and $n$ respectively, where $m/N \rightarrow \alpha$, between-group connection probability $q$ and within-group connection probabilities $p_1$ (within community $C_1$) and $p_2$ (within community $C_2$), satisfying:
\begin{equation}\label{sbmgraphon8}
   \alpha(p_1-q)\ne (1-\alpha)(p_2-q).
\end{equation}
In this case, the limiting graphon is given by:
\[
W(x,y)=
\begin{cases}
p_1, & (x,y)\in [0,\alpha]\times[0,\alpha],\\[6pt]
p_2, & (x,y)\in (\alpha,1]\times(\alpha,1],\\[6pt]
q, & (x,y)\in [0,\alpha]\times(\alpha,1]\ \cup\ (\alpha,1]\times[0,\alpha].
\end{cases}
\]
If the block sizes are asymptotically equal (i.e. $\alpha = 1/2$), then condition \eqref{sbmgraphon8} reduces to unequal within-group connection probabilities (i.e. $p_1\ne p_2)$. On the other hand, if the within-group connection probabilities are the same (i.e. $p_1=p_2$), and unequal to the between-group connection probability $q$, then condition \eqref{sbmgraphon8} amounts to asymptotically unequal block sizes (i.e. $\alpha \ne 1/2$).
 See Section 1.2 in \cite{ghosal2020joint} for a detailed discussion on such examples.

Having shown that joint consistent estimation at rate $N^{-1/2}$ is possible for non mean-field and irregular interactions, we now go to the opposite extreme, where consistent joint estimation is impossible. This happens in the Curie-Weiss Potts model where the coupling matrix is the adjacency of the complete graph (scaled by $N$) (see \eqref{cwpotts1}) and more generally, in the Erd\H{o}s-R\'enyi Potts model, where the coupling matrix is given by:
\begin{equation}\label{erny}
    a_{ij} := \frac{g_{ij}}{Np} 
\end{equation}
 with $G:= ((g_{ij}))_{1\le i,j\le N}$ being the adjacency of an Erd\H{o}s-R\'enyi random graph $\mathcal{G}(N,p)$ with $p>0$ fixed. Note that in the latter model, the coupling matrix is random, so we will consider the problem of estimation under the joint law $\p_{\beta,\bm B}^{\mathrm{ER}}$ of $\bm X$ and $G$ on $[q]^N \times \{0,1\}^{\binom{N}{2}}$.
 Throughout the rest of the paper, we will use the notation $\mathcal{P}([q])$ to denote the set of all probability measures on $[q]$, i.e.
\begin{equation}\label{defpqfirst6}
    \mathcal{P}([q]) := \left\{\bm v \in [0,1]^q: \sum_{r=1}^q v_r =1\right\}.
\end{equation}

\begin{theorem}\label{jinest}
    For each $\bm m \in \mathcal{P}([q])$, let $\Theta_{\bm m}$ be the set of all $(\beta,\bm B) \in (0,\infty)\times \mathbb{R}^{q-1}$ such that the function 
    \begin{equation}\label{Hdefn6882}
        H_{\beta,\bm B}(\bm t) := \frac{\beta}{2} \sum_{r=1}^q t_r^2 + \sum_{r=1}^{q} B_r t_r - \sum_{r=1}^q t_r \log t_r
    \end{equation} 
    has the unique global maximizer $\bm m$ on the set $\mathcal{P}([q])$, and $$\bm u^\top \nabla^2 H_{\beta,\bm B}(\bm m) \bm u<0\text{  for all }\bm u \in T := \{\bm u \in \R^q\setminus{\{\bf 0\}}: \sum_{r=1}^q u_r = 0\}.$$ Then the product measure $\nu:= \bm m^N \times \mathcal{G}(N,p)$ is contiguous to the measure $\p_{\beta,\bm B}^{\mathrm{ER}}$ for every $(\beta,\bm B) \in \Theta_{\bm m}$. Consequently, whenever $|\Theta_{\bm m}| \ge 2$, under $\p_{\beta,\bm B}^{\mathrm{ER}}$ there does not exist any sequence of estimators (functions of $(\bm X, G)$) which is consistent for $(\beta,\bm B)$ in $\Theta_{\bm m}$. 
\end{theorem}

The proof of Theorem \ref{jinest} is based on a contiguity argument that is presented in Section \ref{proof:jinest} of the appendix.

\begin{remark} 
The ambient Hessian of the function $H_{\beta,\bm B}$ is given by:
$$\nabla^2 H_{\beta,\bm B}(\bm t) := \mathrm{diag}\left((\beta-t_r^{-1})_{1\le r \le q}\right).$$ This implies that for $\beta \le 1$, $H_{\beta,\bm B}$ is negative definite for all $\bm t$ in $(0,1)^q$, and hence, the function $H_{\beta, \bm B}$ is strictly concave in this case. Therefore, for $\beta \le 1$, any stationary point of $H_{\beta,\bm B}$ must be its unique maximizer. By a Lagrangian argument (see the proof of Lemma \ref{c222}), an interior stationary point $\bm m \in \mathcal{P}([q])$ of $H_{\beta, \bm B}$ is characterized by the system of equations:
  $$\beta(m_r-m_q) + B_r = \log \frac{m_r}{m_q}\qquad\text{for}~r\in [q-1].$$
   Therefore, for any $\bm m \in \mathcal{P}([q])$, we have:
   \begin{equation}\label{nonsingempt68}
       \left\{\left(\beta,\log \frac{m_1}{m_q} + \beta(m_q-m_1),\ldots, \log \frac{m_{q-1}}{m_q} + \beta(m_q-m_{q-1})\right): 0<\beta \le 1\right\}\subseteq \Theta_{\bm m}
   \end{equation}
   and note that the LHS of \eqref{nonsingempt68} is a non-empty affine (straight) line segment in $\R^q$, containing a continuum of points.
\end{remark}

Finally, we establish consistency results for the partial MPL estimators of $\beta$ and $\bm{B}$, treating each parameter as known while estimating the other. This setting is comparatively simpler and requires weaker assumptions than those needed for joint consistent estimation. For every fixed $\bm B$, the partial MPL estimator of $\beta$ is defined as the unique maximizer of the function $\beta \to \ell_N(\beta,\bm B)$, and for every fixed $\beta$, the partial MPL estimator of $\bm B$ is defined as the unique maximizer of the function $\bm B \to \ell_N(\beta,\bm B)$, if they exist. The following theorem gives consistency rates of the partial MPL estimator of $\beta$ in terms of the quantity  
\begin{equation}\label{eq:Undefn}
U_N(\bm x)=\frac{1}{N} \sum_{1 \leqslant  r<s \leqslant  q} \sum_{i=1}^N (m_{i,r}(\bm x) - m_{i,s}(\bm x))^2,
\end{equation}  
and also that of $\bm B$.

\begin{theorem}\label{partialestm}
Suppose $\bm X$ is a sample from the Potts model \eqref{eq:pmf}, where the coupling matrix $\bm A_N$ satisfies conditions \eqref{as1} and \eqref{as2}, and $(\beta,\bm B) \in \Theta := (0,\infty)\times \mathbb{R}^{q-1}$. 
Then, the following are true:
\vspace{0.1in}
\begin{itemize}
	\item [(a)] The partial MPL estimator $\hat{\bm B}_N$ exists with probability $1-o(1)$ for all $(\beta,\bm B) \in (0,\infty)\times \mathbb{R}^{q-1}$, and the partial MPL estimator $\hat{\beta}_N$ exists with probability $1-o(1)$ whenever 
    \begin{equation}\label{unrateexact648}
    U_N(\bm X)^{-1} = o_\p(\sqrt{N}).
    \end{equation}
	\vspace{0.05in}
	
    \item [(b)] The partial MPL estimator $\hat{\bm B}_N$ satisfies:
    $$\|\hat{\bm B}_N - \bm B\|_2 = O_\p\left(\frac{1}{\sqrt{N}}\right).$$
\vspace{0.05in}

    \item [(c)] If $U_N(\bm X)^{-1} = o_\p(\sqrt{N})$, then the partial MPL estimator $\hat{\beta}_N$ satisfies:
    $$|\hat{\beta}_N - \beta| = O_\p\left(\frac{1}{\sqrt{N}U_N(\bm X)}\right).$$
    \vspace{0.05in}

    \item [(d)]
    Define  \[   
\beta_c := 
     \begin{cases}
       q &\quad\text{if}~q\le 2\\
       \frac{2(q-1)}{q-2} \log (q-1) &\quad\text{otherwise}. \\ 
     \end{cases}
\]
    If $\bm B \ne \boldsymbol{0}$, or if $\bm B = \boldsymbol{0}$ and $\beta \liminf_{N\rightarrow \infty} \frac{{\bf 1}'\bm A_N {\bf 1}}{N}> \beta_c$, then  $U_N(\bm X)^{-1} = O_\p(1)$. In these cases, $\hat{\beta}_N$ is $\sqrt{N}$-consistent for $\beta$.
\vspace{0.15in}

    \item [(e)]  Finally, for the model $\p_{\beta,\bm B}^{\mathrm{ER}}$ described above (see \eqref{erny}), no consistent sequence of estimators exists for $\beta <\beta_c$ (here $\lim_{N\rightarrow \infty} \frac{{\bf 1}'\bm A_N {\bf 1}}{N} =1$ almost surely) when $\bm B = \boldsymbol{0}$.
\vspace{0.05in}

\end{itemize}
\end{theorem}

The proof of Theorem \ref{partialestm} is given in Section \ref{proof:partial8} of the appendix. 

\subsection{Future directions} \textcolor{black}{As a possible future direction, a first question is to relax the assumption of non-negativity of the coupling matrix $\A_N$, and extend our results to the case of spin glass Potts models (similar to what was done for the special case of $q=2$ in \cite{ChenSenWu2024o}). Another interesting question is to go beyond concentration results, and develop central limit theorems for the magnetization vector, similar to  what was done in \cite{deb2020fluctuations} for $q=2$. This will ultimately lead to the  construction of asymptotically valid confidence intervals, a very useful inferential task. Possibly a more challenging direction is to go beyond quadratic interaction models, and study general Gibbs measures with higher order tensors, such as cubics, quartics, and so on. The exponential random graph models (ERGMs) fall under this class of higher order tensor models, and have proved notoriously hard for inference purposes. In particular, the phenomenon of \enquote{degeneracy} for ERGMs has a body of growing literature, both in empirical and rigorous work (see \cite{SnijdersEtAl2006, handcock2003assessing, egrm_2, Mukherjee2020} and references therein).} Other related and more general models in which one may seek to establish analogous results on the joint consistency of parameter estimators include the XY model \citep{Kenna2005TheXM}, the Ashkin--Teller model \citep{AshkinTeller1943, AounDoberGlazman2024}, and the $O(N)$ model \citep{KirkpatrickNawaz2016}.

\subsection{Outline of the paper}
\textcolor{black}{The rest of the paper is organized as follows. The proof of our main result (Theorem \ref{thm:theorem1}) is given in Section \ref{sec:int}.  
In Section \ref{sec:numstudy} we illustrate our theoretical results with a simulation study. In Appendices \ref{sec:bdgrproof} and \ref{sec:irrgrproof}, we prove Theorems \ref{bdgr} and \ref{irrgr}, respectively. Appendices \ref{proof:jinest} and \ref{proof:partial8} are dedicated to the proofs of the remaining main results of the paper, namely Theorems \ref{jinest} and \ref{partialestm}, respectively. In Appendix \ref{mestsec4}, we prove a general result on convergence of $Z$-estimators, whereas in Appendix \ref{techres}, we develop necessary tools for bounding the derivatives of the log pseudolikelihood, both of which are crucial in establishing consistency and rates of convergence for our MPL estimators (Theorems \ref{thm:theorem1} and \ref{partialestm}). In Appendix \ref{prcw7} we prove Lemma \ref{cwresult7}, which is a crucial step towards proving Theorem \ref{bdgr}. In Appendix \ref{sec:necth1p3}, we prove some results necessary for verifying Theorem \ref{jinest}. Finally, Appendix \ref{sec:othertechnle} contains additional technical lemmas necessary for proving some of the main results of the paper.}

\section{Proof of Theorem \ref{thm:theorem1}}\label{sec:int}

This section is dedicated to proving the main result of this paper (Theorem \ref{thm:theorem1}). Further technical lemmas necessary for proving these results are given in the appendix.

\vspace{0.1in}

\noindent (a) Suppose that $\bm X \in \Omega_N \bigcap \Lambda_N$. Then, there exist $1\le a\ne b \le q$ and $1\le i,j,k,l \le N$ all distinct, such that $\{X_i,X_j\} = \{a,b\}$, $\{X_k,X_\ell\} = \{a,b\}$ and $\{\tm_i^{a,b}(\bm X), \tm_j^{a,b}(\bm X)\} < \{\tm_k^{a,b}(\bm X),\tm_\ell^{a,b}(\bm X)\}$ (recall that $\tm_u^{r,s}(\bm X) := m_{u,r}(\bm X) - m_{u,s}(\bm X)$). Since the function $\ell_N$ is concave (see Lemma \ref{hesdet}), in order to show the existence of the MPL estimator, it suffices to show that:
$$\lim_{\|(\beta,\bm B)\|_\infty \rightarrow \infty} \ell_N(\beta,\bm B) = -\infty.$$

 Without loss of generality, assume that $X_i=X_k=a$ and $X_j=X_\ell= b$. Since 
 \begin{eqnarray*}
     &&h_w(\beta,\bm B)\\ &:=& \beta\sum_{r=1}^q m_{w,r}(\bm X)\mathbbm{1}_{X_w=r} +  \sum_{r=1}^{q} B_{r} \mathbbm{1}_{X_{w} =r} - \log\left(\sum_{r=1}^{q} \exp\left\{\beta m_{w,r}(\bm X) + B_{r} \right\}\right) \le 0
 \end{eqnarray*}
 for all $1\le w\le N$, it suffices to show that at least one of $h_i(\beta,\bm B), h_j(\beta,\bm B), h_k(\beta,\bm B)$ and $h_\ell(\beta,\bm B)$ goes to $-\infty$ as $\|(\beta,\bm B)\|_\infty \rightarrow \infty$. Now, note that:
 \begin{equation}\label{logitlast6}
     \exp(h_w(\beta,\bm B)) = \frac{1}{1+\sum_{r\ne X_w} \exp\left(\beta (m_{w,r}(\bm X) - m_{w,X_w}(\bm X))+ (B_r - B_{X_w})\right)}~,
 \end{equation}
Setting $w = i, j, k, \ell$ in \eqref{logitlast6}, it suffices to show that at least one of the following four quantities:
\begin{enumerate}
    \item $\beta (m_{i,r}(\bm X) - m_{i,a}(\bm X))+ (B_r - B_a)$,
    \item $\beta (m_{k,r}(\bm X) - m_{k,a}(\bm X))+ (B_r - B_{a})$,
    \item  $\beta (m_{j,r}(\bm X) - m_{j,b}(\bm X))+ (B_r - B_{b})$, and
    \item $\beta (m_{\ell,r}(\bm X) - m_{\ell,b}(\bm X))+ (B_r - B_{b})$
\end{enumerate}
goes to $+\infty$ as  $\|(\beta,\bm B)\|_\infty \rightarrow \infty$, for at least one $r\in [q]$. So, let us assume that none of them goes to $+\infty$, i.e. there exists a constant $K \in (0,\infty)$, such that all of (1), (2), (3) and (4) are bounded above by $K$ along the sequence $(\beta,\bm B)$ whose norm goes to $+\infty$, for all $r\in [q]$. In this case, putting $r=b$ in (1) and $r= a$ in (4), we get:
\begin{eqnarray*}
   && -\beta \tm_i^{a,b} (\bm X) - (B_a-B_b) \le K\quad\text{and}\quad \beta \tm_\ell^{a,b} (\bm X) + (B_a-B_b) \le K\\&\implies& \quad \beta \le \frac{2K}{\tm_\ell^{a,b} (\bm X)- \tm_i^{a,b} (\bm X)}~. 
\end{eqnarray*}

Similarly, putting $r=b$ in (2) and $r=a$ in (3), we get:
\begin{eqnarray*}
   && -\beta \tm_k^{a,b} (\bm X)- (B_a-B_b) \le K\quad\text{and}\quad \beta \tm_j^{a,b} (\bm X) + (B_a-B_b) \le K\\ &\implies& \quad \beta \ge -\frac{2K}{\tm_k^{a,b}(\bm X) - \tm_j^{a,b} (\bm X)}~.
\end{eqnarray*}

Thus, 
\begin{equation*}\label{betaboundK7}
    |\beta| \le K_0 := \max\left\{\frac{2K}{\tm_\ell^{a,b}(\bm X) - \tm_i^{a,b}(\bm X)}~,~\frac{2K}{\tm_k^{a,b}(\bm X) - \tm_j^{a,b}(\bm X)}\right\}.
\end{equation*}
Now, choose any two distinct colors $s, t\in [q]$. \textcolor{black}{Since $\bm X \in \Lambda_N$, we can choose $1\le u\ne v\le N$ such that $X_u = s$ and $X_v = t$}. If either $h_u(\beta,\bm B)$ or $h_v(\beta,\bm B) \rightarrow -\infty$, then once again we are done. So, assume otherwise, i.e. both are bounded below by some constant, which implies that there exists $\varepsilon > 0$ such that: $$\min \left\{\exp(h_u(\beta,\bm B)), \exp(h_v(\beta,\bm B))\right\} > \varepsilon,$$
which, in view of \eqref{logitlast6}, implies that:
\begin{equation}\label{esn1}
    \beta \tm_u^{r,s}(\bm X)  + B_r-B_s < -\log \varepsilon\quad\text{for all}~r\ne s
\end{equation}
and
\begin{equation}\label{esn2}
  \beta \tm_v^{r,t}(\bm X)  + B_r-B_t< -\log \varepsilon\quad\text{for all}~r\ne t. 
\end{equation}
Now, putting $r=t$ in \eqref{esn1} and $r=s$ in \eqref{esn2}, we get:
$$B_t-B_s < -\log \varepsilon -\beta \tm_u^{t,s}(\bm X)\quad\text{and}\quad B_s-B_t < -\log \varepsilon -\beta \tm_v^{s,t}(\bm X).$$
By Assumption \ref{as1}, there exists $\gamma \in (0,\infty)$ such that $\sup_{N\ge 1, i \in [N]}\sum_{j=1}^N a_{ij} <\gamma$, which implies that $|\tm_u^{t,s}(\bm X)| \le 2\gamma$ and $|\tm_v^{t,s}(\bm X)|\le 2\gamma$.
This now implies that
$|B_t - B_s| \le -\log \varepsilon + 2K_0\gamma$. Setting $s=q$, we thus have $|B_t| 
\le -\log \varepsilon + 2K_0\gamma$, i.e. $B_t$ is bounded. Hence, $(\beta,\bm B)$ is a bounded sequence, a contradiction. This proves (a).

\vspace{0.1in}

\noindent (b)~ 
It follows from Lemma \ref{bp12} (taking $\lambda = 0$) that $\|\nabla \ell_N(\beta,\bm B)\|_2 = O_\p(\sqrt{N})$. Part (b) now follows from Proposition \ref{genZthlem} (on taking $w_N(\beta, \bm B) = \nabla \ell_N(\beta, \bm B)$, $a_N =\sqrt{N}$ and $h_N(\bm X) = N T_N(\bm X)$) and Lemma \ref{hesdet}. \qed

\vspace{0.1in}

\noindent (c)~It suffices to check \eqref{joincg5}, as the other conclusions follow from parts (a) and (b). To this effect, define 
$$E_N(\delta) := \left\{\bm x \in [q]^N: T_N(\bm x) < \delta\right\}.$$ 
It follows from the tightness of $T_N(\bm X)^{-1}$, that
\begin{equation}\label{tightdel}
	\lim_{\delta \rightarrow 0} \sup_{N\ge 1} \mathbb{P}(\bm X\in E_N(\delta)) = 0.
\end{equation} 
Suppose that $\bm x \in E_N(\delta)^c$ for some fixed $\delta > 0$. Fixing $\bm x$, for notational convenience, we will abbreviate $\tm_{u}^{r,s}(\bm x)$ by $\tm_u^{r,s}$. Then, we have:
$$T_N(\bm x) = \frac{1}{2N^2}\sum_{r<s} \sum_{i,j} \left(\tm_i^{r,s} - \tm_j^{r,s}\right)^2 \ge \delta,$$
Hence, there exist colors $a<b$, such that
\begin{equation}\label{tocontr66}
  \frac{1}{N^2} \sum_{i,j} \left(\tm_i^{a,b} - \tm_j^{a,b}\right)^2 \ge \frac{4\delta}{q(q-1)}. 
\end{equation}
It follows from \eqref{as1} that for any $r,s\in [q]$ and $i\in [N]$, $\tm_{i}^{r,s} \in [-\gamma, \gamma]$, where $\gamma$ is defined in \eqref{as1}. Fixing a positive integer $R>\frac{9q\gamma}{4\sqrt{\delta}}$, and setting  $t_p:=\frac{p\gamma}{R}$ for $p\in \mathbb{Z}$, note that
 $$[-\gamma, \gamma]=\bigcup_{-R\le p\le R-1}[t_p,t_{p+1}].$$
 Let us define
$$S_p := \left\{i\in [N]: \tm_i^{a,b} \ge \frac{p\gamma}{R}\right\}$$ and $$p_0 = p_0(\varepsilon) := \max \left\{p\in \mathbb{Z}\cap[-R,R+1]: |S_p|  \ge \varepsilon N\right\}-1.$$
We now claim that whenever $ \varepsilon\in \Big(0, \delta/(9q^2\gamma^2)\Big)$, $|S_{p_0+1}|\ge \varepsilon N$ and $|S_{p_0-1}^c|\ge \varepsilon N$.
\vspace{0.1in}

\noindent {\bf Proof of Claim:}
Note that $|S_{-R}| = N$, $|S_{R+1}| = 0$ and $|S_p|$ is decreasing in $p$, so such a $p_0$ exists. Also, it follows directly from the definition of $p_0$, that:
$$|S_{p_0+1}| \ge \varepsilon N\quad\text{and}\quad |S_{p_0+2}| <\varepsilon N.$$
We claim that $|S_{p_0-1}^c| \ge \varepsilon N$. If this is not true, then there must exist at least $(1-2\varepsilon)N$ many $i$'s for which $\tm_i^{a,b} \in I := [(p_0-1)\gamma/R, (p_0+2)\gamma/R)$. Hence, we have the following: 
\begin{eqnarray*}
	&&\frac{1}{N^2} \sum_{i,j} (\tm_i^{a,b} - \tm_j^{a,b})^2\\ &\le& \frac{2}{N^2} \sum_{i \in S_{p_0+2}, 1\le j\le N} (\tm_i^{a,b} - \tm_j^{a,b})^2 + \frac{2}{N^2} \sum_{i \in S_{p_0-1}^c, 1\le j\le N} (\tm_i^{a,b} - \tm_j^{a,b})^2\\ &+& \frac{1}{N^2} \sum_{i,j:~ \tm_i^{a,b}, \tm_j^{a,b} \in I} (\tm_i^{a,b} - \tm_j^{a,b})^2\\&\le & \frac{2|S_{p_0+2}|}{N}(4\gamma^2) + \frac{2|S_{p_0-1}^c|}{N}(4\gamma^2) + \frac{9\gamma^2}{R^2}\\ &\le & 16\varepsilon \gamma^2  + \frac{9\gamma^2}{R^2} < \frac{4\delta}{q(q-1)}. 
\end{eqnarray*} 
where the last inequality follows from the choice of $\varepsilon$ and $R$ above. But this contradicts \eqref{tocontr66}, thus verifying the claim. 

Using the claim, the colors $a$ and $b$ satisfy $\tm_i^{a,b} < (p_0-1)\gamma/R$ for at least $\varepsilon N$ many $i$ and $\tm_i^{a,b} \ge p_0\gamma/R$ for at least $\varepsilon N$ many $i$. Let us now define the following four events:
$$\mathcal{E}_{1,a} :=  \left\{\bm x\in [q]^N: x_i\ne a~\forall i \in S_{p_0}\right\} , \quad \mathcal{E}_{1,b} := \left\{\bm x\in [q]^N: x_i\ne b~\forall i \in S_{p_0}\right\}$$ $$\mathcal{E}_{2,a} :=  \left\{\bm x\in [q]^N: x_i\ne a~\forall i \in S_{p_0}^c\right\} , \quad \mathcal{E}_{2,b} := \left\{\bm x\in [q]^N: x_i\ne b~\forall i \in S_{p_0}^c\right\}.$$
We claim that each of the above four sets, intersected with $E_N(\delta)^c$, has probability $o(1)$ of containing $\bm X$ as $N\rightarrow \infty$.
\vspace{0.1in}

\noindent \textbf{Proof of Claim:}
Define a function {$g:\mathbb{R}\mapsto [0,\infty)$ as:
\[   
g(x) := 
\begin{cases}
	0 &\quad\text{if}~x\le 0\\
	x^2 &\quad\text{if}~x\in (0,0.5)\\
	x-0.25 &\quad\text{if}~x >0.5.\\ 
\end{cases}
\]
Then it is straightforward to check that $g$ is differentiable on $\mathbb{R}$ with derivative bounded by $1$. Also, $g$ is non-decreasing, strictly positive on the positive axis, and bounded on compact intervals.
Let us now define:
{\color{black}\begin{align*}h_1({\bm x}) := &\max_{c\in F} \max_{r,s \in [q]} \sum_{i=1}^N \left(\mathbbm{1}_{x_i\ne r} - \mathbb{P}(X_i\ne r|X_j=x_j,{j\ne i})\right) g\left(\tm_i^{r,s}({\bm x}) - c\right),\\
h_2({\bm x}) := &\max_{c\in F} \max_{r,s \in [q]} \sum_{i=1}^N \left(\mathbbm{1}_{x_i\ne r} - \mathbb{P}(X_i\ne r|X_j=x_j,{j\ne i})\right) g\left(c - \tm_i^{r,s}({\bm x})\right),
\end{align*}}
where $$F := \left\{\frac{p\gamma}{R}: p \in \mathbb{Z}\cap[-R,R+1]\right\}.$$
Whenever $\bm x\in \mathcal{E}_{1,a}\cap E_N(\delta)^c$, we have: 
\begin{eqnarray*}\label{kappa1scn}
 h_1(\bm x) &\ge& \sum_{i=1}^N \left(\mathbbm{1}_{x_i\ne a} - \mathbb{P}(X_i\ne a|X_j=x_j,{j\ne i})\right) g\left(\tm_i^{a,b}({\bm x}) - \frac{p_0\gamma}{R}\right)\nonumber\\ &\ge&  \alpha \sum_{i\in S_{p_0}} g\left(\tm_i^{a,b}(\bm x) - \frac{p_0\gamma}{R}\right) \ge \alpha \sum_{i\in S_{p_0+1}} g\left(\tm_i^{a,b}(\bm x) - \frac{p_0\gamma}{R}\right) \ge \alpha \varepsilon N g\left(\frac{\gamma}{R}\right), 
\end{eqnarray*} 
where 
\begin{equation}\label{alphadef4}
    \alpha := q^{-1}\exp\left\{-\beta \gamma -2\|\bm B\|_\infty\right\} > 0.
\end{equation}

 Here the first inequality in the second line uses the fact that $x_i\ne a$ for $i\in S_{p_0}$, by definition of $\mathcal{E}_{1,a}$, and the last inequality uses the fact that $|S_{p_0+1}|\ge \varepsilon N$.
Applying Lemma \ref{bp12} with $b_{itrs}:= \mathbbm{1}_{t\ne r}$, $\lambda =1$, $t=N [\alpha \varepsilon g\Big(\frac{\gamma}{R}\Big)]^2$ and a union bound we get
\begin{align*}
   \p(\bm X\in \mathcal{E}_{1,a}\cap E_N(\delta)^c) &\le   \p\left(|h_1(\bm X)|>\alpha \varepsilon N g\Big(\frac{\gamma}{R}\Big)\right)\\ &\le 2(2R+2)q^2 \exp\Big(-CN \Big[\alpha \varepsilon g\Big(\frac{\gamma}{R}\Big)\Big]^2\Big)\to 0.
\end{align*}
One can similarly show that $\p(\bm X \in \mathcal{E}_{1,b}\cap E_N(\delta)^c)= o(1)$. Also, for $\bm x \in \mathcal{E}_{2,a}\cap E_N(\delta)^c$, a similar calculation gives
\begin{eqnarray}\label{kappa1scn4}
 h_2(\bm x) &\ge& \sum_{i=1}^N \left(\mathbbm{1}_{x_i\ne a} - \mathbb{P}(X_i\ne a|X_j=x_j,{j\ne i})\right) g\left(\frac{p_0\gamma}{R} - \tm_i^{a,b}({\bm x})\right)\nonumber\\ &\ge&\alpha \sum_{i\in S_{p_0}^c} g\left(\frac{p_0\gamma}{R} - \tm_i^{a,b}(\bm x) \right)\alpha \sum_{i\in S_{p_0-1}^c} g\left(\frac{p_0\gamma}{R} - \tm_i^{a,b}(\bm x)\right) \ge \alpha \varepsilon N g\left(\frac{ \gamma}{R}\right).
\end{eqnarray}
Once again, by Lemma \ref{bp12}, the probability that $\bm X$ satisfies \eqref{kappa1scn4} is $o(1)$ as $N\rightarrow \infty$, which proves that $\p(\bm X \in \mathcal{E}_{2,a}\cap E_N(\delta)^c) = o(1)$.
One can similarly show that $\p(\bm X \in \mathcal{E}_{2,b}\cap E_N(\delta)^c)= o(1)$, which completes the proof of the claim.
\\

When ${\bm x}$ belongs to $\mathcal{E}_{1,a}^c\bigcap \mathcal{E}_{1,b}^c\bigcap \mathcal{E}_{2,a}^c\bigcap \mathcal{E}_{2,b}^c$, there exist $i,j,k,\ell \in [N]$, such that:
$x_i=a, x_j = b, x_k=a, x_\ell=b$, and $$\max\{\tm_i^{a,b} (\bm x),\tm_j^{a,b} (\bm x)\} < \frac{(p_0)\gamma}{R} \le \min \{\tm_k^{a,b} (\bm x),\tm_\ell^{a,b} (\bm x)\}~,$$ and so ${\bm x}\in \Omega_N$ (as introduced in the statement of Theorem \ref{thm:theorem1} (a)). Hence, for all $\delta >0$, we have:
\begin{align*}\p(\bm X \in \Omega_N^c) \le &\p (\bm X\in E_N(\delta))+\p( \bm X\in \Omega_N^c\cap E_N(\delta)^c)\\
\le &\p (\bm X\in E_N(\delta))+\p\left(\bm X\in \Big( \mathcal{E}_{1,a}\cup \mathcal{E}_{1,b}\cup \mathcal{E}_{2,a}\cup \mathcal{E}_{2,b}\Big)\cap E_N(\delta)^c\right).
\end{align*}
The second term in the RHS  converges to $0$ on letting $N\to \infty$ as shown above, and the first term converges to $0$ on taking a supremum over $N$ followed by $N\to \infty$ using \eqref{tightdel}, thus showing that
$ \p(\bm X \in \Omega_N^c) = o(1).$


Next, we show that $\p(\bm X \notin \Lambda_N) = o(1)$. For $\bm x\in \Lambda_N^c$, there exists $r\in [q]$ such that $x_i \neq r$ for all $i$. This implies that for all $i\in [N]$, we have $$\sum_{s \neq r}\mathbbm{1}_{x_i=s}=1 \Rightarrow \sum_{s \neq r}(\mathbbm{1}_{x_i=s}-\theta_{i,s}({\bm x})) = 1- \sum_{s \ne r}\theta_{i,s}(\bm x)~,$$
where $ \theta_{i,s}(\bm x) = \mathbb{P}(X_i=s|X_j=x_j,{j\ne i})$
 is as in \eqref{defcondprob881}.
Now, we have $1- \sum_{s \ne r}\theta_{i,s}(\bm x) \ge \alpha$ (see \eqref{alphadef4} for the definition of $\alpha$), which implies: 
$$\sum_{s\ne r}\sum_{i=1}^N(\mathbbm{1}_{x_i=s}-\theta_{i,s}({\bm x}))\ge \frac{N}{q \exp(\beta \gamma+2 \max_{r\in [q] }|B_r|)}.$$
Now, by Lemma \ref{bp12} (taking $b_{itrs}:= \mathbbm{1}_{t\ne r}$ and $g \equiv 1$), we have:
$$   \sum_{s \neq r} \sum_{i=1}^N(\mathbbm{1}_{X_i=s}-\theta_{i,s}({\bm x}))=O_{\p}(\sqrt{N}).
$$
Therefore, $\p(\bm X\in \Lambda_N^c) = o(1)$, completing the proof of \eqref{joincg5}. The proof of \eqref{joincg6} will follow from part (b) above.
\vspace{0.1in}


\section{Numerical Study}\label{sec:numstudy}

To illustrate our theoretical results, we simulate observations from the Potts model on the Erd\H{o}s-R\'enyi graph $\mathcal{G}(N,p)$ with $N=100$ and $200$, and $p=0.025$ and $0.25$. The $p=0.025$ case illustrates the sparse regime where joint estimation is possible, and the $p=0.25$ case illustrates the dense regime where joint estimation is not possible on the level-sets of the map $(\beta,\bm B) \mapsto \arg \max_{\bm t} H_{\beta,\bm B}(\bm t)$ provided this maximizer is unique. We do this for $q=3$ for the ease of representation of the numerical results through $3D$ graphs. One can easily derive using Lagrange multipliers, that for each such unique maximizer $\bm m$, the inestimability curve $\Theta_{\bm m}$ is a straight line that has equation:
$$B_1 = \log\left(\frac{m_1\exp(\beta m_3)}{m_3\exp(\beta m_1)}\right)\quad\text{and}\quad B_2 = \log\left(\frac{m_2\exp(\beta m_3)}{m_3\exp(\beta m_2)}\right)$$
We call this straight line the \textit{line of inestimability}, which is plotted in blue in Figure \ref{fig:solutionpathsbm}. We take $\bm m = (0.2,0.5,0.3)$, and for each value of $\beta$ within the range $0$ to $2$ in increments of $0.01$, compute the MPL estimates based on samples generated from the Erd\H{o}s-R\'enyi Potts model with the corresponding true parameters lying on the line of inestimability. We use Gibbs sampling for the simulation.

For both $N=100$ and $N=200$, the green points, representing the MPL estimates for the sparse case, lie very close to the line of inestimability, thereby supporting our result that joint estimation is always possible in the sparse (bounded-degree) case. The closeness increases from $N=100$ to $N=200$, as indicated by the $N^{-1/2}$ rate of convergence. On the other hand, the red points, representing the MPL estimates for the dense case, seem to scatter away from the blue line, a phenomenon which is best observed from the significantly high number of red points that have very large values of $\hat{\beta}$ in comparison to the height of the blue line. The green points, on the other hand, have $\hat{\beta}$ values lying more or less within the height limits of the line of inestimability. This demonstrates the inestimability phenomenon in the dense case.

\begin{figure}
		\begin{minipage}[b]{0.49\linewidth}
			\centering
			\includegraphics[width=\textwidth]{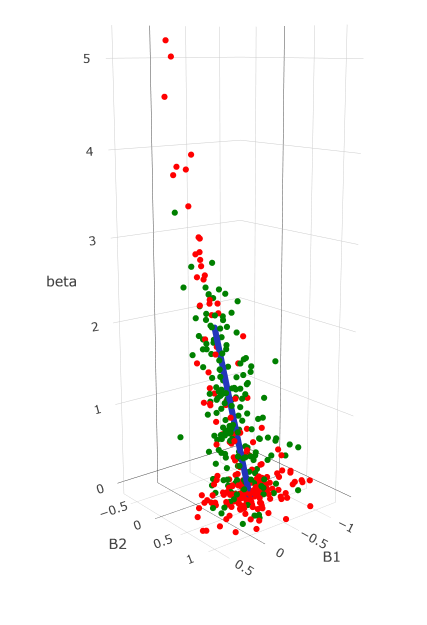}
			\caption*{(a)}
		\end{minipage}
		\begin{minipage}[b]{0.49\linewidth}
			\centering
			\includegraphics[width=\textwidth]{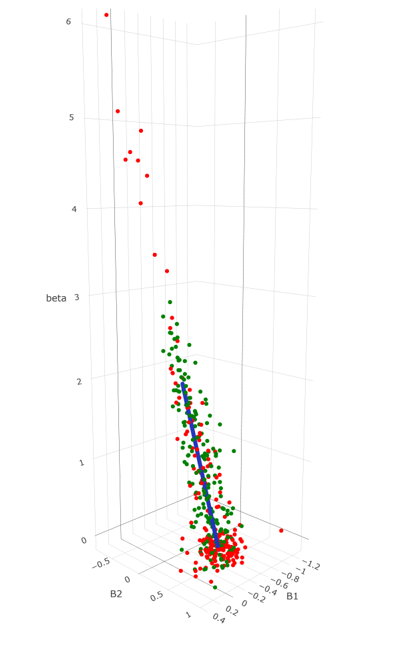}
			\caption*{(b)}
		\end{minipage}
		\caption{\small{Plot of the MPL estimate $(\hat{\beta},\hat{B}_1,\hat{B}_2)$ for the Potts model on $\mathcal{G}(N,p)$ with (a)  $N=100$ and (b)  $N=200$. Blue line denotes the line of inestimability, green points denote the MPL estimates for $p=0.025$ (the sparse case) and red points denote the MPL estimates for $p=0.25$ (the dense case).}}
		\label{fig:solutionpathsbm} 
	\end{figure}

\section{Acknowledgement}
The first author gratefully acknowledges support from the National University of Singapore Start-Up Grant A-0008523-00-00 and the AcRF Tier 1 grants A-8001449-00-00 and A-8002932-00-00.
The second author gratefully acknowledges NSF support (DMS-2515519, DMS-2113414) during this research. The third author gratefully acknowledges NSF support (DMS-2124222) during this research. We also thank Moumanti Podder for her valuable assistance with an earlier version of this manuscript.

\bibliographystyle{plainnat}
\bibliography{potts}

\begin{thebibliography}{75}
\providecommand{\natexlab}[1]{#1}
\providecommand{\url}[1]{\texttt{#1}}
\expandafter\ifx\csname urlstyle\endcsname\relax
  \providecommand{\doi}[1]{doi: #1}\else
  \providecommand{\doi}{doi: \begingroup \urlstyle{rm}\Url}\fi

\bibitem[Ali et~al.(2008)Ali, Farag, and Gimel'farb]{estimation_1}
Asem~M. Ali, Aly~A. Farag, and Georgy Gimel'farb.
\newblock Analytical method for mgrf potts model parameter estimation.
\newblock In \emph{2008 19th International Conference on Pattern Recognition},
  pages 1--4. IEEE, 2008.

\bibitem[Anandkumar et~al.(2012)Anandkumar, Tan, Huang, and
  Willsky]{estimation_13}
Animashree Anandkumar, Vincent Y.~F. Tan, Furong Huang, and Alan~S. Willsky.
\newblock High-dimensional structure estimation in ising models: Local
  separation criterion.
\newblock \emph{The Annals of Statistics}, 40\penalty0 (3):\penalty0
  1346--1375, 2012.
\newblock \doi{10.1214/12-AOS1009}.

\bibitem[Aoun et~al.(2024)Aoun, Dober, and Glazman]{AounDoberGlazman2024}
Yacine Aoun, Moritz Dober, and Alexander Glazman.
\newblock Phase diagram of the ashkin--teller model.
\newblock \emph{Communications in Mathematical Physics}, 405:\penalty0 37,
  2024.
\newblock \doi{10.1007/s00220-023-04925-0}.

\bibitem[Ashkin and Teller(1943)]{AshkinTeller1943}
J.~Ashkin and E.~Teller.
\newblock Statistics of two-dimensional lattices with four components.
\newblock \emph{Physical Review}, 64\penalty0 (5-6):\penalty0 178--184, 1943.
\newblock \doi{10.1103/PhysRev.64.178}.

\bibitem[Basak and Mukherjee(2017)]{mukherjeebasak}
Anirban Basak and Sumit Mukherjee.
\newblock Universality of the mean-field for the potts model.
\newblock \emph{Probability theory and related fields}, 168:\penalty0 557--600,
  2017.

\bibitem[Berthet et~al.(2019)Berthet, Rigollet, and
  Srivastava]{berthet2019exact}
Quentin Berthet, Philippe Rigollet, and Piyush Srivastava.
\newblock Exact recovery in the ising blockmodel.
\newblock \emph{The Annals of Statistics}, 47\penalty0 (4):\penalty0
  1805--1834, 2019.
\newblock \doi{10.1214/17-AOS1620}.

\bibitem[Besag(1974)]{besag1974spatial}
Julian Besag.
\newblock Spatial interaction and the statistical analysis of lattice systems.
\newblock \emph{Journal of the Royal Statistical Society. Series B
  (Methodological)}, 36:\penalty0 192--236, 1974.

\bibitem[Besag(1975)]{besag1975statistical}
Julian Besag.
\newblock Statistical analysis of non-lattice data.
\newblock \emph{The Statistician}, 24\penalty0 (3):\penalty0 179--195, 1975.

\bibitem[Bhamidi et~al.(2011)Bhamidi, Bresler, and Sly]{bhamidi2011mixing}
Shankar Bhamidi, Guy Bresler, and Allan Sly.
\newblock Mixing time of exponential random graphs.
\newblock \emph{The Annals of Applied Probability}, 21\penalty0 (6):\penalty0
  2146--2170, 2011.
\newblock \doi{10.1214/10-AAP740}.

\bibitem[Bhattacharya and Mukherjee(2018)]{estimation_14}
Bhaswar~B. Bhattacharya and Sumit Mukherjee.
\newblock Inference in ising models.
\newblock \emph{Bernoulli}, 24\penalty0 (1):\penalty0 493--525, 2018.

\bibitem[Bhowal and Mukherjee(2023)]{snbh22}
Sanchayan Bhowal and Somabha Mukherjee.
\newblock Supplement to ``limit theorems and phase transitions in the tensor
  curie--weiss potts model''.
\newblock \emph{Information and Inference: A Journal of the IMA}, pages 1--43,
  2023.
\newblock \doi{10.48550/arXiv.2307.01052}.

\bibitem[Bhowal and Mukherjee(2025{\natexlab{a}})]{snbh1}
Sanchayan Bhowal and Somabha Mukherjee.
\newblock Limit theorems and phase transitions in the tensor curie-weiss potts
  model.
\newblock \emph{Information and Inference: A Journal of the IMA}, 14\penalty0
  (2):\penalty0 iaaf014, 05 2025{\natexlab{a}}.
\newblock ISSN 2049-8772.
\newblock \doi{10.1093/imaiai/iaaf014}.
\newblock URL \url{https://doi.org/10.1093/imaiai/iaaf014}.

\bibitem[Bhowal and Mukherjee(2025{\natexlab{b}})]{snbh2}
Sanchayan Bhowal and Somabha Mukherjee.
\newblock Rates of convergence of the magnetization in the tensor curie--weiss
  potts model.
\newblock \emph{Journal of Statistical Physics}, 192\penalty0 (2):\penalty0 2,
  2025{\natexlab{b}}.
\newblock \doi{10.1007/s10955-024-03382-w}.

\bibitem[Boas et~al.(2018)Boas, Jiang, Merks, Prokopiou, and
  Rens]{Boas2018PottsBio}
S.~E.~M. Boas, Y.~Jiang, R.~M.~H. Merks, S.~A. Prokopiou, and E.~G. Rens.
\newblock Cellular potts model: Applications to vasculogenesis and
  angiogenesis.
\newblock \emph{Probabilistic Cellular Automata}, 27:\penalty0 279--310, 2018.

\bibitem[Bornholdt(2021)]{Bornholdt2021PottsMarkets}
Stefan Bornholdt.
\newblock A q-spin potts model of markets: Gain--loss asymmetry in stock
  indices as an emergent phenomenon.
\newblock \emph{arXiv preprint arXiv:2112.06290}, 2021.

\bibitem[Bosconti et~al.(2015)Bosconti, Corallo, Fortunato, Gentile, Massafra,
  and Pell\`e]{Bosconti2015SocialPotts}
Cristian Bosconti, Angelo Corallo, Laura Fortunato, Antonio~A. Gentile, Andrea
  Massafra, and Piergiuseppe. Pell\`e.
\newblock Reconstruction of a real world social network using the potts model
  and loopy belief propagation.
\newblock \emph{Frontiers in Psychology}, 6, 2015.

\bibitem[Bresler(2015)]{comp_2}
Guy Bresler.
\newblock Efficiently learning ising models on arbitrary graphs.
\newblock In \emph{Proceedings of the forty-seventh annual ACM symposium on
  Theory of computing}, pages 771--782, 2015.

\bibitem[Cao et~al.(2022)Cao, Neykov, and Liu]{cao2022hightemperature}
Yuan Cao, Matey Neykov, and Han Liu.
\newblock High-temperature structure detection in ferromagnets.
\newblock \emph{Information and Inference: A Journal of the IMA}, 11\penalty0
  (1):\penalty0 55--102, 2022.
\newblock \doi{10.1093/imaiai/iaaa032}.

\bibitem[Celeux et~al.(2002)Celeux, Forbes, and
  Peyrard]{Celeux2002PottsSegmentation}
Gilles Celeux, Florence Forbes, and Nathalie Peyrard.
\newblock Em-based image segmentation using potts models with external field.
\newblock Technical Report RR-4456, INRIA, 2002.

\bibitem[Chatterjee(2007{\natexlab{a}})]{chatterjee2007stein}
Sourav Chatterjee.
\newblock Stein’s method for concentration inequalities.
\newblock \emph{Probability theory and related fields}, 138\penalty0
  (1-2):\penalty0 305--321, 2007{\natexlab{a}}.

\bibitem[Chatterjee(2007{\natexlab{b}})]{estimation_15}
Sourav Chatterjee.
\newblock Estimation in spin glasses: A first step.
\newblock \emph{The Annals of Statistics}, 35\penalty0 (5):\penalty0
  1931--1946, 2007{\natexlab{b}}.

\bibitem[Chatterjee and Dembo(2016)]{chatterjee2016nonlinear}
Sourav Chatterjee and Amir Dembo.
\newblock Nonlinear large deviations.
\newblock \emph{Advances in Mathematics}, 299:\penalty0 396--450, 2016.

\bibitem[Chatterjee and Diaconis(2013)]{egrm_2}
Sourav Chatterjee and Persi Diaconis.
\newblock Estimating and understanding exponential random graph models.
\newblock \emph{The Annals of Statistics}, 41\penalty0 (5):\penalty0
  2428--2461, 2013.

\bibitem[Chen et~al.(2024)Chen, Sen, and Wu]{ChenSenWu2024o}
Wei-Kuo Chen, Arnab Sen, and Qiang Wu.
\newblock Joint parameter estimations for spin glasses.
\newblock \emph{arXiv preprint arXiv:2406.10760}, 2024.
\newblock URL \url{https://arxiv.org/abs/2406.10760}.

\bibitem[Comets and Gidas(1991)]{comets}
Francis Comets and Basilis Gidas.
\newblock Asymptotics of maximum likelihood estimators for the curie--weiss
  model.
\newblock \emph{The Annals of Statistics}, 19\penalty0 (2):\penalty0 557--578,
  1991.
\newblock \doi{10.1214/aos/1176348111}.

\bibitem[Daskalakis et~al.(2020)Daskalakis, Dikkala, and
  Panageas]{daskalakis2020logistic}
Constantinos Daskalakis, Nishanth Dikkala, and Ioannis Panageas.
\newblock Logistic regression with peer-group effects via inference in
  higher-order ising models.
\newblock In \emph{International Conference on Artificial Intelligence and
  Statistics}, pages 3653--3663. PMLR, 2020.

\bibitem[Deb and Mukherjee(2023)]{deb2020fluctuations}
Nabarun Deb and Sumit Mukherjee.
\newblock Fluctuations in mean-field ising models.
\newblock \emph{The Annals of Applied Probability}, 33\penalty0 (3):\penalty0
  1961--2003, June 2023.
\newblock \doi{10.1214/22-AAP1857}.

\bibitem[Deb et~al.(2024)Deb, Mukherjee, Mukherjee, and Yuan]{deb2020detecting}
Nabarun Deb, Rajarshi Mukherjee, Sumit Mukherjee, and Ming Yuan.
\newblock Detecting structured signals in {Ising} models.
\newblock \emph{The Annals of Applied Probability}, 34\penalty0 (1A):\penalty0
  1--45, 2024.
\newblock \doi{10.1214/23-AAP1929}.

\bibitem[DeMuse et~al.(2019)DeMuse, Easlick, and Yin]{DeMuseEaslickYin2019}
Ryan DeMuse, Terry Easlick, and Mei Yin.
\newblock Mixing time of vertex-weighted exponential random graphs.
\newblock \emph{J. Comput. Appl. Math.}, 362:\penalty0 443--459, 2019.

\bibitem[Descombes et~al.(1999)Descombes, Morris, Zerubia, and
  Berthod]{estimation_3}
Xavier Descombes, Robin~D Morris, Josiane Zerubia, and Marc Berthod.
\newblock Estimation of markov random field prior parameters using markov chain
  monte carlo maximum likelihood.
\newblock \emph{IEEE Transactions on Image Processing}, 8\penalty0
  (7):\penalty0 954--963, 1999.

\bibitem[Ding et~al.(2009)Ding, Lubetzky, and Peres]{DingLubetzkyPeres2009}
Jian Ding, Eyal Lubetzky, and Yuval Peres.
\newblock The mixing time evolution of glauber dynamics for the mean-field
  ising model.
\newblock \emph{Communications in Mathematical Physics}, 289:\penalty0
  725--764, 2009.
\newblock \doi{10.1007/s00220-009-0781-9}.

\bibitem[Eichelsbacher and Martschink(2015)]{eichels}
Peter Eichelsbacher and Bastian Martschink.
\newblock On rates of convergence in the curie--weiss--potts model with an
  external field.
\newblock \emph{Annales de l'Institut Henri Poincar{\'e}, Probabilit{\'e}s et
  Statistiques}, 51\penalty0 (1):\penalty0 252--282, 2015.
\newblock \doi{10.1214/14-AIHP599}.

\bibitem[Ellis(1985)]{ellis1985entropy}
Richard~S. Ellis.
\newblock \emph{Entropy, Large Deviations, and Statistical Mechanics}.
\newblock Springer, New York, 1985.

\bibitem[Ellis and Wang(1990)]{ellis1990limit}
Richard~S. Ellis and Kongming Wang.
\newblock Limit theorems for the empirical vector of the curie-weiss-potts
  model.
\newblock \emph{Stochastic Processes and their Applications}, 35\penalty0
  (1):\penalty0 59--79, 1990.

\bibitem[Ellis and Wang(1992)]{ellis1992}
Richard~S. Ellis and Kongming Wang.
\newblock Limit theorems for maximum likelihood estimators in the
  curie--weiss--potts model.
\newblock \emph{Stochastic Processes and their Applications}, 40\penalty0
  (2):\penalty0 251--288, 1992.

\bibitem[Ellis et~al.(1980)Ellis, Newman, and Rosen]{ellis1980limit}
Richard~S. Ellis, Charles~M. Newman, and Jay~S. Rosen.
\newblock Limit theorems for sums of dependent random variables occurring in
  statistical mechanics.
\newblock \emph{Zeitschrift f{\"u}r Wahrscheinlichkeitstheorie und Verwandte
  Gebiete}, 51:\penalty0 153--169, 1980.

\bibitem[Gandolfo et~al.(2010)Gandolfo, Ruiz, and Marc]{gandolfo}
Daniel Gandolfo, Jean Ruiz, and Wouts Marc.
\newblock Limit theorems and coexistence probabilities for the curie–weiss
  potts model with an external field.
\newblock \emph{Stochastic Processes and their Applications}, 120:\penalty0
  84--104, 2010.

\bibitem[Ghosal and Mukherjee(2020)]{ghosal2020joint}
Promit Ghosal and Sumit Mukherjee.
\newblock Joint estimation of parameters in ising model.
\newblock \emph{The Annals of Statistics}, 48\penalty0 (2):\penalty0 785--810,
  2020.

\bibitem[Gimenez et~al.(2013)Gimenez, Frery, and Flesia]{estimation_2}
Javier Gimenez, Alejandro~C. Frery, and Ana~Georgina Flesia.
\newblock Inference strategies for the smoothness parameter in the {P}otts
  model.
\newblock In \emph{2013 IEEE International Geoscience and Remote Sensing
  Symposium-IGARSS}, pages 2539--2542. IEEE, 2013.

\bibitem[Graner and Glazier(1992)]{Graner1992CellularPotts}
Fran{\c{c}}ois Graner and James~A. Glazier.
\newblock Simulation of biological cell sorting using a two-dimensional
  extended potts model.
\newblock \emph{Physical Review Letters}, 69\penalty0 (13):\penalty0
  2013--2016, 1992.
\newblock \doi{10.1103/PhysRevLett.69.2013}.

\bibitem[Handcock et~al.(2003)Handcock, Robins, Snijders, Moody, and
  Besag]{handcock2003assessing}
Mark~S. Handcock, Garry Robins, Tom Snijders, Jim Moody, and Julian Besag.
\newblock Assessing degeneracy in statistical models of social networks.
\newblock Technical report, Working paper, 2003.

\bibitem[He and Lok(2025)]{HeLok2025}
Roxanne He and Jackie Lok.
\newblock On approximating the potts model with contracting glauber dynamics.
\newblock \emph{Probability in the Engineering and Informational Sciences},
  2025.
\newblock \doi{10.1017/S0269964825000336}.
\newblock Published online November 7, 2025.

\bibitem[Ising(1925)]{ising}
Ernst Ising.
\newblock Beitrag zur theorie des ferromagnetismus.
\newblock \emph{Zeitschrift für Physik}, 31:\penalty0 253--258, 1925.

\bibitem[Kenna(2005)]{Kenna2005TheXM}
Ralph Kenna.
\newblock The xy model and the berezinskii--kosterlitz--thouless phase
  transition.
\newblock \emph{arXiv preprint arXiv:cond-mat/0512356}, 2005.
\newblock \doi{10.48550/arXiv.cond-mat/0512356}.

\bibitem[Kirkpatrick and Nawaz(2016)]{KirkpatrickNawaz2016}
Kay Kirkpatrick and Tayyab Nawaz.
\newblock Asymptotics of mean-field {O(N)} models.
\newblock \emph{Journal of Statistical Physics}, 165\penalty0 (6):\penalty0
  1114--1140, 2016.
\newblock \doi{10.1007/s10955-016-1667-9}.

\bibitem[Levada et~al.(2008{\natexlab{a}})Levada, Mascarenhas, and
  Tann{\'u}s]{estimation_10}
Alexandre L.~M. Levada, Nelson D.~A. Mascarenhas, and Alberto Tann{\'u}s.
\newblock Pseudolikelihood equations for potts mrf model parameter estimation
  on higher order neighborhood systems.
\newblock \emph{IEEE Geoscience and Remote Sensing Letters}, 5\penalty0
  (3):\penalty0 522--526, 2008{\natexlab{a}}.

\bibitem[Levada et~al.(2008{\natexlab{b}})Levada, Mascarenhas, Tann{\'u}s, and
  Salvadeo]{estimation_11}
Alexandre L.~M. Levada, Nelson D.~A. Mascarenhas, Alberto Tann{\'u}s, and
  Denis~HP Salvadeo.
\newblock Spatially non-homogeneous potts model parameter estimation on
  higher-order neighborhood systems by maximum pseudo-likelihood.
\newblock In \emph{Proceedings of the 2008 ACM symposium on Applied computing},
  pages 1733--1737, 2008{\natexlab{b}}.

\bibitem[Levada et~al.(2009)Levada, Mascarenhas, and
  Tann\'us]{Levada2009PottsImage}
Alexandre L.~M. Levada, Nelson D.~A. Mascarenhas, and Alberto Tann\'us.
\newblock Pseudo-likelihood equations for potts model on higher-order
  neighborhood systems: A quantitative approach for parameter estimation in
  image analysis.
\newblock \emph{Brazilian Journal of Probability and Statistics}, 23\penalty0
  (2):\penalty0 120--140, 2009.

\bibitem[Levin et~al.(2010)Levin, {\L}uczak, and Peres]{LevinLuczakPeres2010}
David~A. Levin, Malwina~J. {\L}uczak, and Yuval Peres.
\newblock Glauber dynamics for the mean-field ising model: Cutoff, critical
  power law, and metastability.
\newblock \emph{Probability Theory and Related Fields}, 146\penalty0
  (1-2):\penalty0 223--265, 2010.
\newblock \doi{10.1007/s00440-008-0173-0}.

\bibitem[Lokhov et~al.(2018)Lokhov, Vuffray, Misra, and Chertkov]{lokhov1}
Andrey~Y. Lokhov, Marc Vuffray, Sidhant Misra, and Michael Chertkov.
\newblock Optimal structure and parameter learning of ising models.
\newblock \emph{Science Advances}, 4\penalty0 (3):\penalty0 e1700791, 2018.
\newblock \doi{10.1126/sciadv.1700791}.

\bibitem[Lov{\'a}sz(2012)]{lovasz2012}
L{\'a}szl{\'o} Lov{\'a}sz.
\newblock \emph{Large Networks and Graph Limits}, volume~60 of \emph{American
  Mathematical Society Colloquium Publications}.
\newblock American Mathematical Society, Providence, RI, 2012.
\newblock \doi{10.1090/coll/060}.

\bibitem[McGrory et~al.(2009)McGrory, Titterington, Reeves, and
  Pettitt]{estimation_5}
Clare~A McGrory, D~Michael Titterington, Robert Reeves, and Anthony~N Pettitt.
\newblock Variational bayes for estimating the parameters of a hidden potts
  model.
\newblock \emph{Statistics and Computing}, 19\penalty0 (3):\penalty0 329--340,
  2009.

\bibitem[Moltchanova et~al.(2005)Moltchanova, Pitk\"aniemi, and
  Haapala]{Moltchanova2005PottsHaplotype}
Elena~V. Moltchanova, Janne Pitk\"aniemi, and Laura. Haapala.
\newblock Potts model for haplotype associations.
\newblock \emph{BMC Genetics}, 6\penalty0 (Suppl 1):\penalty0 S64, 2005.

\bibitem[Mukherjee et~al.(2018)Mukherjee, Mukherjee, and
  Yuan]{mukherjee2018global}
Rajarshi Mukherjee, Sumit Mukherjee, and Ming Yuan.
\newblock Global testing against sparse alternatives under ising models.
\newblock \emph{The Annals of Statistics}, 46\penalty0 (5):\penalty0
  2062--2093, 2018.

\bibitem[Mukherjee et~al.(2021)Mukherjee, Son, and Bhattacharya]{tensor2}
Somabha Mukherjee, Jaesung Son, and Bhaswar~B. Bhattacharya.
\newblock Fluctuations of the magnetization in the p-spin curie--weiss model.
\newblock \emph{Communications in Mathematical Physics}, 387:\penalty0
  681--728, 2021.
\newblock \doi{10.1007/s00220-021-04072-6}.

\bibitem[Mukherjee et~al.(2022)Mukherjee, Son, and Bhattacharya]{tensor1}
Somabha Mukherjee, Jaesung Son, and Bhaswar~B. Bhattacharya.
\newblock Estimation in tensor ising models.
\newblock \emph{Information and Inference: A Journal of the IMA}, 11\penalty0
  (4):\penalty0 1457--1500, 2022.
\newblock \doi{10.1093/imaiai/iaac007}.

\bibitem[Mukherjee et~al.(2024)Mukherjee, Son, Ghosh, and Mukherjee]{tensor4}
Somabha Mukherjee, Jaesung Son, Swarnadip Ghosh, and Sourav Mukherjee.
\newblock Efficient estimation in tensor curie--weiss and
  erd{\H{o}}s--r{\'e}nyi ising models.
\newblock \emph{Electronic Journal of Statistics}, 18\penalty0 (1):\penalty0
  2405--2449, 2024.
\newblock \doi{10.1214/24-EJS2255}.

\bibitem[Mukherjee et~al.(2025)Mukherjee, Son, and Bhattacharya]{tensor3}
Somabha Mukherjee, Jaesung Son, and Bhaswar~B. Bhattacharya.
\newblock Phase transitions of the maximum likelihood estimators in the p-spin
  curie--weiss model.
\newblock \emph{Bernoulli}, 31\penalty0 (2):\penalty0 1502--1526, 2025.
\newblock \doi{10.3150/24-BEJ1779}.

\bibitem[Mukherjee(2020)]{Mukherjee2020}
Sumit Mukherjee.
\newblock Degeneracy in sparse ergms with functions of degrees as sufficient
  statistics.
\newblock \emph{Bernoulli}, 26\penalty0 (2):\penalty0 1016--1043, 2020.
\newblock \doi{10.3150/19-BEJ1135}.

\bibitem[Mukherjee and Xu(2023)]{MukherjeeXu2023}
Sumit Mukherjee and Yuanzhe Xu.
\newblock Statistics of the two-star ergm.
\newblock \emph{Bernoulli}, 29\penalty0 (1):\penalty0 24--51, 2023.
\newblock \doi{10.3150/21-BEJ1448}.

\bibitem[Neykov and Liu(2019)]{neykov2019property}
Matey Neykov and Han Liu.
\newblock Property testing in high-dimensional ising models.
\newblock \emph{The Annals of Statistics}, 47\penalty0 (5):\penalty0
  2472--2503, 2019.
\newblock \doi{10.1214/18-AOS1754}.

\bibitem[Okabayashi et~al.(2011)Okabayashi, Johnson, and Geyer]{estimation_4}
Saisuke Okabayashi, Leif Johnson, and Charles~J. Geyer.
\newblock Extending pseudo-likelihood for potts models.
\newblock \emph{Statistica Sinica}, pages 331--347, 2011.

\bibitem[Pereyra et~al.(2013)Pereyra, Dobigeon, Batatia, and
  Tourneret]{estimation_7}
Marcelo Pereyra, Nicolas Dobigeon, Hadj Batatia, and Jean-Yves Tourneret.
\newblock Estimating the granularity coefficient of a potts-markov random field
  within a markov chain monte carlo algorithm.
\newblock \emph{IEEE Transactions on Image Processing}, 22\penalty0
  (6):\penalty0 2385--2397, 2013.

\bibitem[Pereyra et~al.(2014)Pereyra, Whiteley, Andrieu, and
  Tourneret]{estimation_8}
Marcelo Pereyra, Nick Whiteley, Christophe Andrieu, and Jean-Yves Tourneret.
\newblock Maximum marginal likelihood estimation of the granularity coefficient
  of a potts-markov random field within an mcmc algorithm.
\newblock In \emph{2014 IEEE Workshop on Statistical Signal Processing (SSP)},
  pages 121--124. IEEE, 2014.

\bibitem[Potts(1952)]{potts}
Renfrey~B. Potts.
\newblock Some generalized order-disorder transformations.
\newblock In \emph{Mathematical proceedings of the cambridge philosophical
  society}, volume~48, pages 106--109. Cambridge University Press, 1952.

\bibitem[Ravikumar et~al.(2010)Ravikumar, Wainwright, and
  Lafferty]{estimation_16}
Pradeep Ravikumar, Martin~J Wainwright, and John~D Lafferty.
\newblock High-dimensional ising model selection using $\ell_{1}$-regularized
  logistic regression.
\newblock \emph{The Annals of Statistics}, 38\penalty0 (3):\penalty0
  1287--1319, 2010.

\bibitem[Rosu et~al.(2015)Rosu, Giovannelli, Giremus, and Vacar]{estimation_9}
Roxana-Gabriela Rosu, Jean-Fran{\c{c}}ois Giovannelli, Audrey Giremus, and
  Cornelia Vacar.
\newblock Potts model parameter estimation in bayesian segmentation of
  piecewise constant images.
\newblock In \emph{2015 IEEE International Conference on Acoustics, Speech and
  Signal Processing (ICASSP)}, pages 4080--4084. IEEE, 2015.

\bibitem[Samanta et~al.(2024)Samanta, Mukherjee, and
  Zhang]{SamantaMukherjeeZhang2024}
Ramkrishna~Jyoti Samanta, Somabha Mukherjee, and Jiang Zhang.
\newblock Mixing phases and metastability for the glauber dynamics on the
  $p$-spin curie--weiss model.
\newblock \emph{arXiv preprint arXiv:2412.16952}, 2024.
\newblock \doi{10.48550/arXiv.2412.16952}.
\newblock Version 2, revised February 20, 2025.

\bibitem[Snijders et~al.(2006)Snijders, Pattison, Robins, and
  Handcock]{SnijdersEtAl2006}
Tom A.~B. Snijders, Philippa~E. Pattison, Garry~L. Robins, and Mark~S.
  Handcock.
\newblock New specifications for exponential random graph models.
\newblock \emph{Sociological Methodology}, 36\penalty0 (1):\penalty0 99--153,
  2006.
\newblock \doi{10.1111/j.1467-9531.2006.00176.x}.

\bibitem[Song et~al.(2016)Song, Si, Herrmann, and Feng]{estimation_6}
Sanming Song, Bailu Si, J~Michael Herrmann, and Xisheng Feng.
\newblock Local autoencoding for parameter estimation in a hidden potts-markov
  random field.
\newblock \emph{IEEE Transactions on Image Processing}, 25\penalty0
  (5):\penalty0 2324--2336, 2016.

\bibitem[Stivala et~al.(2020)Stivala, Robins, and Lomi]{egrm_4}
Alex Stivala, Garry Robins, and Alessandro Lomi.
\newblock Exponential random graph model parameter estimation for very large
  directed networks.
\newblock \emph{PloS one}, 15\penalty0 (1):\penalty0 e0227804, 2020.

\bibitem[Takaishi(2005)]{Takaishi2005PottsFinance}
Tetsuya Takaishi.
\newblock Simulations of financial markets in a potts-like model.
\newblock \emph{International Journal of Modern Physics C}, 16\penalty0 (8),
  2005.

\bibitem[Vuffray et~al.(2016)Vuffray, Misra, Lokhov, and Chertkov]{lokhov2}
Marc Vuffray, Sidhant Misra, Andrey~Y. Lokhov, and Michael Chertkov.
\newblock Interaction screening: Efficient and sample-optimal learning of ising
  models.
\newblock In \emph{Advances in Neural Information Processing Systems 29
  (NeurIPS 2016)}, 2016.

\bibitem[Wu(1982)]{wu1982potts}
F.~Y. Wu.
\newblock The potts model.
\newblock \emph{Reviews of Modern Physics}, 54\penalty0 (1):\penalty0 235--268,
  1982.
\newblock \doi{10.1103/RevModPhys.54.235}.

\bibitem[Zukovic and Hristopulos(2008)]{Zukovic2008SpatialPotts}
Milan Zukovic and Dionissios~T. Hristopulos.
\newblock Simulations of environmental spatial data using ising and potts
  models.
\newblock In \emph{SigmaPhi Conference}, Kolympari, Greece, 2008.

\end{thebibliography}

\newpage
\begin{appendix}

\section{Proof of Theorem \ref{bdgr}}\label{sec:bdgrproof}
In view of Theorem \ref{thm:theorem1} (c), it suffices to show that under condition \eqref{bddegr}, we have $T_N(\bm X)^{-1} = O_\p(1)$. Now, for every $\delta > 0$, define the set:
\begin{equation}\label{endef86}
    E_N(\delta) := \{\bm x\in [q]^N: T_N(\bm x) \le \delta\}.
\end{equation}
In order to prove that $T_N(\bm X)^{-1} = O_\p(1)$, it suffices to show the stronger conclusion that there exists $\delta\in (0,1)$ such that $\p_{\beta,\bm B}(\bm X\in E_N(\delta)) = o(1)$.
Towards this, \textcolor{black}{using the definition of $T_N$ from \eqref{eq:Tndefn}, we have:} 
\begin{eqnarray}
   T_N(\bm x)  &=& \sum_{1 \leqslant  r<s \leqslant  q} \left( \frac{1}{N}\sum_{i=1}^N (m_{i,r}(\bm x)-m_{i,s}(\bm x))^2 - (\overline{m}_r(\bm x)-\overline{m}_s(\bm x))^2\right)\notag\\ &=& \frac{1}{N}\sum_{i=1}^N \sum_{1\le r<s\le q} (m_{i,r}(\bm x) - m_{i,s}(\bm x))^2 - \sum_{1\le r<s\le q} (\overline{m}_r(\bm x)-\overline{m}_s(\bm x))^2\label{mixedmirGini}\\&=& \frac{1}{N}\sum_{i=1}^N \left(q\sum_{r=1}^q m_{i,r}(\bm x)^2 - \left(\sum_{r=1}^q m_{i,r}(\bm x)\right)^2\right) - \left(q \sum_{r=1}^q \overline{m}_r(\bm x)^{\,2} \;-\; \bar{R}^{\,2}\right)\notag\\&=&\frac{q}{N}\sum_{i=1}^N \sum_{r=1}^q \left[\left(m_{i,r}(\bm x)-\frac{R_i}{q}\right)^2  - \left( \overline{m}_r(\bm x) - \frac{\bar{R}}{q}\right)^2\right]\notag\\&=& \frac{q}{N}\sum_{i=1}^N \sum_{r=1}^q\left(m_{i,r}(\bm x) - \overline{m}_r(\bm x) - \frac{1}{q}(R_i-\bar{R})\right)^2.\label{TNderv}
\end{eqnarray}
where $R_i := \sum_{j=1}^N a_{ij}\quad\text{and} \quad \bar{R} := \frac{1}{N}\sum_{i=1}^N R_i.$ \textcolor{black}{(as in \eqref{irregap}).}
Hence, for $\bm x\in E_N(\delta)$ (as defined in \eqref{endef86}), we have by the Cauchy-Schwarz inequality:
\begin{eqnarray*}
    &&\left|\sum_{i=1}^N \sum_{r=1}^q m_{i,r}(\bm x) x_{i,r}- \sum_{i=1}^N \sum_{r=1}^q \left(\overline{m}_r(\bm x) + \frac{1}{q}(R_i-\bar{R})\right)x_{i,r}\right|\\&\le& \sqrt{qN}\sqrt{\sum_{i=1}^N \sum_{r=1}^q\left(m_{i,r}(\bm x) - \overline{m}_r(\bm x) - \frac{1}{q}(R_i-\bar{R})\right)^2} \\&\le& N\sqrt{\delta}.
\end{eqnarray*}
Since  $\sum_{i=1}^N \sum_{r=1}^q (R_i-\bar{R}) x_{i,r} = \sum_{i=1}^N (R_i-\bar{R}) = 0$, the above display gives
\begin{equation}\label{apx2}
    \left|\sum_{i=1}^N \sum_{r=1}^q m_{i,r}(\bm x) x_{i,r}- \sum_{r=1}^q \overline{m}_r(\bm x) \sum_{i=1}^N  x_{i,r}\right| \le N\sqrt{\delta}.
\end{equation}
Next, define the following function:
\begin{equation}\label{frdef56}
  f_r(t_1,\ldots,t_q) := \frac{\exp\{\beta t_r + B_r \}}{\sum_{s=1}^q \exp\{\beta t_s + B_s \}}~. 
\end{equation}
Recalling that: 
$$\theta_{i,r}(\bm x)= \frac{\exp\left\{\beta m_{i,r}(\bm x) +  B_{r} \right\}}{\sum_{s=1}^{q} \exp\left\{\beta m_{i,s}(\bm x) + B_{s} \right\}}=f_r(m_{i,1}(\bm x),\cdots, m_{i,q}(\bm x))\quad (\text{see \eqref{defcondprob881}})$$
and noting that the first order partial derivatives of $f_r$ are bounded by $\beta$, we have the following from the mean-value theorem:
\begin{eqnarray*}
&& \Big|f_r(m_{i,1}(\bm x),\ldots,m_{i,q}(\bm x))-f_r(n_{i,1}(\bm x),\ldots,n_{i,q}(\bm x))\Big| \lesssim \sum_{u=1}^q |m_{i,u}(\bm x) - n_{i,u}(\bm x)|
\end{eqnarray*}
where
$$n_{i,r}(\bm x) := \overline{m}_r(\bm x) + \frac{1}{q}(R_i-\bar{R}).$$
This implies that for $\bm x \in E_N(\delta)$,
\begin{eqnarray}\label{intmd}
\notag\sum_{i=1}^N\Big|\theta_{i,r}(\bm x) -f_r(n_{i,1}(\bm x),\ldots,n_{i,q}(\bm x))\Big|
&\lesssim&\sum_{i=1}^N\sum_{u=1}^q|m_{i,u}(\bm x)-n_{i,u}(\bm x)|\\
\notag&\le& \sqrt{qN}\sqrt{\sum_{i=1}^N\sum_{u=1}^q (m_{i,u}(\bm x) - n_{i,u}(\bm x))^2}\\& \le & N\sqrt{\delta}.
\end{eqnarray}
Invoking \eqref{intmd} together with the observation that
$$f_r(n_{i,1}(\bm x),\ldots,n_{i,q}(\bm x)) = \frac{\exp\{\beta \overline{m}_r(\bm x) + B_r\}}{\sum_{s=1}^q \exp\{\beta \overline{m}_s(\bm x) + B_s\} }$$ implies that for $\bm x \in E_N(\delta)$ we have:
\begin{equation}\label{apx33}
    \left|\sum_{i=1}^N \theta_{i,r}(\bm x) - \frac{N\exp\{\beta \overline{m}_r(\bm x) + B_r\}}{\sum_{s=1}^q \exp\{\beta \overline{m}_s(\bm x) + B_s\} }\right|  \lesssim N\sqrt{\delta}
\end{equation}
Again using \eqref{intmd} and the fact that $R_i \le \gamma$ \textcolor{black}{(see \eqref{as1})}, for $\bm x\in E_N(\delta)$ we have:
$$ \left|\sum_{i=1}^N R_i \left[\theta_{i,r}(\bm x) - \frac{\exp\{\beta \overline{m}_r(\bm x) + B_r\}}{\sum_{s=1}^q \exp\{\beta \overline{m}_s(\bm x) + B_s\} }\right]\right|  \lesssim N\sqrt{\delta},\quad\text{i.e.}$$

\begin{equation}\label{apx44}
\left|\sum_{i=1}^N \sum_{j=1}^N a_{ij}\theta_{i,r}(\bm x) - N \bar{R} \frac{\exp\{\beta \overline{m}_r(\bm x) + B_r\}}{\sum_{s=1}^q \exp\{\beta \overline{m}_s(\bm x) + B_s\} }\right| \lesssim N\sqrt{\delta}.
\end{equation}
Next, define the sets
\begin{align*}
    C_{N}(\delta) :=& \left\{\bm x\in [q]^N: \max_{r\in [q]} \left|\sum_{i=1}^N (x_{i,r} - \theta_{i,r}(\bm x))\right| \le N\delta\right\},\\ D_N(\delta) := &\left\{\bm x\in [q]^N: \max_{r\in [q]} \left|\sum_{i=1}^N \left(m_{i,r}(\bm x) - \sum_{j=1}^N a_{ij}\theta_{i,r}(\bm x)\right)\right| \le N\delta\right\}.
    \end{align*}



 By the triangle inequality, we have: 
{
\begin{eqnarray}\label{apx666}
&&\left|\bar{R}\sum_{i=1}^N x_{i,r} - \sum_{i=1}^N m_{i,r}(\bm x)\right|\nonumber\\&\le& \left|\bar{R}\sum_{i=1}^N x_{i,r} - \bar{R}\sum_{i=1}^N \theta_{i,r}(\bm x)\right| + \left|\bar{R}\sum_{i=1}^N \theta_{i,r}(\bm x) - \frac{N\bar{R}\exp\{\beta \overline{m}_r(\bm x) + B_r\}}{\sum_{s=1}^q \exp\{\beta \overline{m}_s(\bm x) + B_s\} }\right|\nonumber\\&+& \left|\frac{N\bar{R}\exp\{\beta \overline{m}_r(\bm x) + B_r\}}{\sum_{s=1}^q \exp\{\beta \overline{m}_s(\bm x) + B_s\} } - \sum_{i=1}^N \sum_{j=1}^N a_{ij}\theta_{i,r}(\bm x)\right| + \left|\sum_{i=1}^N \sum_{j=1}^N a_{ij}\theta_{i,r}(\bm x) - \sum_{i=1}^N m_{i,r}(\bm x) \right|.\nonumber\\
\end{eqnarray}}
Suppose now, that $\bm x \in F_N(\delta) := C_N(\delta)\bigcap D_N(\delta) \bigcap E_N(\delta)$ for some $\delta \in (0,1)$. Since $\bm x\in C_N(\delta)$, we have: $$ \left|\bar{R}\sum_{i=1}^N x_{i,r} - \bar{R}\sum_{i=1}^N \theta_{i,r}(\bm x)\right| \le N\delta \gamma \lesssim N\sqrt{\delta}.$$ Next, in view of \eqref{apx33} (since $\bm x \in E_N(\delta)$), we have:
$$\left|\bar{R}\sum_{i=1}^N \theta_{i,r}(\bm x) - \frac{N\bar{R}\exp\{\beta \overline{m}_r(\bm x) + B_r\}}{\sum_{s=1}^q \exp\{\beta \overline{m}_s(\bm x) + B_s\} }\right| \lesssim N\sqrt{\delta} \gamma \lesssim N\sqrt{\delta}.$$ 
By \eqref{apx44} (since $\bm x \in E_N(\delta)$), we have:
$$\left|\frac{N\bar{R}\exp\{\beta \overline{m}_r(\bm x) + B_r\}}{\sum_{s=1}^q \exp\{\beta \overline{m}_s(\bm x) + B_s\} } - \sum_{i=1}^N \sum_{j=1}^N a_{ij}\theta_{i,r}(\bm x)\right| \lesssim N\sqrt{\delta}.$$
Finally, since $\bm x\in D_N(\delta)$, we have:
$$\left|\sum_{i=1}^N \sum_{j=1}^N a_{ij}\theta_{i,r}(\bm x) - \sum_{i=1}^N m_{i,r}(\bm x) \right| \le N\delta \le N\sqrt{\delta}.$$
Plugging these inequalities in \eqref{apx666}, we have:
\begin{equation}\label{34postrb}
    \left|\bar{R}\sum_{i=1}^N x_{i,r} - \sum_{i=1}^N m_{i,r}(\bm x)\right| \lesssim 4N\sqrt{\delta}.
\end{equation}
Denoting $\bar{x}_r := N^{-1}\sum_{i=1}^N x_{i,r}$, for all $\bm x\in [q]^N$ we have:
\begin{eqnarray*}
    &&\left|\sum_{i=1}^N \sum_{r=1}^q x_{i,r}m_{i,r}(\bm x) - N \bar{R}\sum_{r=1}^q\bar{x}_r^2\right|\nonumber\\&\le& \left|\sum_{i=1}^N \sum_{r=1}^q x_{i,r}m_{i,r}(\bm x) - N \sum_{r=1}^q\overline{m}_r(\bm x)\bar{x}_r\right| + \left| N \sum_{r=1}^q\bar{x}_r\Big(\overline{m}_r(\bm x) -  \bar{R}\bar{x}_r\Big)\right|\nonumber\\&\lesssim& N\sqrt{\delta}~,
\end{eqnarray*}
where we use \eqref{apx2} and \eqref{34postrb} respectively, along with the trivial bound $\bar{x}_r\le 1$. With $K$ denoting the implied constant in the above display, we get
\begin{eqnarray}\label{eq:f}
\notag\p_{\beta,\bm B}(\bm X \in F_N(\delta)) &=& \frac{1}{Z_N(\beta,\bm B)} \sum_{\bm x\in F_N(\delta)} \exp\left(\frac{\beta}{2} \sum_{i=1}^N\sum_{r=1}^q m_{i,r}(\bm x) x_{i,r} + \sum_{i=1}^N\sum_{r=1}^{q} B_r x_{i,r}\right)\\&\le& \frac{e^{K N\sqrt{\delta}}}{Z_N(\beta,\bm B)} \sum_{\bm x\in F_N(\delta)} \exp\left(\frac{N\beta\bar{R}}{2}\sum_{r=1}^q \bar{x}_r^2 + N\sum_{r=1}^{q} B_r\bar{x}_r\right)\\\notag&=& e^{K N\sqrt{\delta}}~\frac{Z_N^{\mathrm{CW}}(\beta\bar{R},\bm B)}{Z_N(\beta,\bm B)} ~\p_{\beta\bar{R},\bm B}^{\mathrm{CW}}(\bm X \in F_N(\delta))\\\notag&\le& e^{K N\sqrt{\delta}}~\frac{Z_N^{\mathrm{CW}}(\beta\bar{R},\bm B)}{Z_N(\beta,\bm B)} ~\p_{\beta\bar{R},\bm B}^{\mathrm{CW}}(\bm X \in E_N(\delta))
\end{eqnarray}
where for $(\beta,\bm B) \in (0,\infty)\times \mathbb{R}^{q-1}$, we \textcolor{black}{define} the Curie-Weiss Potts model with parameter $(\beta,\bm B)$ as:
\begin{equation}\label{cwpotts1}
    \p_{\beta,\bm B}^{\mathrm{CW}}(\bm x) = \frac{1}{Z_N^{\mathrm{CW}}(\beta,\bm B)} \exp\left(\frac{N\beta}{2}\sum_{r=1}^q \bar{x}_r^2 + N\sum_{r=1}^{q} B_r\bar{x}_r\right)~.
\end{equation}

Now, it follows from the Gibbs variational principle/mean field lower bound (see Equation (1.8) in \cite{mukherjeebasak}) that:
\begin{eqnarray*}
\log Z_N(\beta,\bm B) &\ge& \sup_{\bm t \in \mathcal{P}([q])}\frac{\beta}{2} \sum_{i,j}a_{ij} \sum_{r=1}^q t_r^2 - N\sum_{r=1}^q t_r \log t_r + N\sum_{r=1}^{q} B_r t_r  \\ &=& \sup_{\bm t \in \mathcal{P}([q])}\frac{N\beta}{2} \bar{R} \sum_{r=1}^q t_r^2 - N\sum_{r=1}^q t_r \log t_r + N\sum_{r=1}^{q} B_r t_r~. 
\end{eqnarray*}
Also, it follows from equation (2.3) in \cite{mukherjeebasak} that:
$$\log Z_N^{\mathrm{CW}}(\beta\bar{R},\bm B) = \sup_{\bm t\in \mathcal{P}([q])} \left[ \frac{N\beta}{2}\bar{R} \sum_{r=1}^q t_r^2 - N\sum_{r=1}^q t_r \log t_r + N\sum_{r=1}^{q} B_rt_r\right] + o(N)~.$$
Combining the above two displays gives
$$\frac{Z_N^{\mathrm{CW}}(\beta\bar{R},\bm B)}{Z_N(\beta,\bm B)} \le e^{o(N)},$$
which along with \eqref{eq:f} gives
\begin{equation}\label{cwcompr}
\p_{\beta,\bm B}(\bm X \in F_N(\delta)) \le e^{K N(\sqrt{\delta} + o(1))}~\p_{\beta\bar{R},\bm B}^{\mathrm{CW}}(\bm X \in E_N(\delta)).
\end{equation}
The following lemma bounds $\p_{\beta\bar{R},\bm B}^{\mathrm{CW}}(\bm X \in E_N(\delta))$.

\begin{lem}\label{cwresult7}
Suppose the interaction matrix $\bm A_N$ satisfies the conditions \eqref{as1}, \eqref{as2} and \eqref{bddegr}. There exist constants $A,B>0$ depending only on $\beta,B,\gamma, q$, such that:
    $$\p_{\beta\bar{R},\bm B}^{\mathrm{CW}}(\bm X \in E_N(A)) \le e^{-BN}.$$
\end{lem}

Lemma \ref{cwresult7} is proved in Appendix \ref{prcw7}. Using Lemma \ref{cwresult7} together with \eqref{cwcompr}, we know that for every $\delta \in (0,A)$,
$$\p_{\beta,\bm B}(\bm X \in F_N(\delta)) \le \exp\{N(K\sqrt{\delta} -B + o(1))\}.$$
We can now choose $\delta = \min\left\{\frac{A}{2}, \frac{B^2}{4K^2}\right\}$ small enough, such that 
\begin{equation}\label{abvdisp98}
    \p_{\beta,\bm B}(\bm X \in F_N(\delta)) \le e^{-\frac{NB}{2}+o(1  )}=o(1).
\end{equation}
Recalling that $F_N(\delta)=C_N(\delta)\cap D_N(\delta)\cap E_N(\delta)$, we get
$$\p_{\beta,\bm B}(\bm X \in E_N(\delta)) \le \p_{\beta,\bm B} (\bm X \in F_N(\delta)) + \p\left(\bm X \notin C_N(\delta)\bigcap D_N(\delta)\right).$$
By \eqref{abvdisp98}, the first term in the RHS of the above inequality is $o(1)$. It also follows directly from Lemma \ref{bp12} (on taking $b_{itrs} := \mathbbm{1}_{t=r}$ and $g\equiv 1$), that $\p_{\beta,\bm B}(\bm X\notin C_N(\delta)) = o(1)$. Finally, note that:
$$\sum_{i=1}^N \left(m_{i,r}(\bm x) - \sum_{j=1}^N a_{ij} \theta_{i,r}(\bm x)\right) = \sum_{j=1}^N R_j x_{j,r} - \sum_{i=1}^N R_i \theta_{i,r}(\bm x) = \sum_{i=1}^N R_i (x_{i,r} - \theta_{i,r}(\bm x)).$$ It now follows from Lemma \ref{bp12} (on taking $b_{itrs} := R_i\mathbbm{1}_{t=r}$ and $g \equiv 1$), that $\p_{\beta,\bm B}(\bm X\notin D_N(\delta)) = o(1)$.  Hence, $\p_{\beta,\bm B}(\bm X\notin C_N(\delta) \bigcap D_N(\delta)) = o(1)$. We thus conclude that
$\p_{\beta,\bm B}(\bm X \in E_N(\delta)) = o(1)$, thereby completing the proof of Theorem \ref{bdgr}. \qed

\section{Proof of Theorem \ref{irrgr}}\label{sec:irrgrproof}
Since we have already established $\sqrt{N}$-consistency of the MPL estimator $(\hat{\beta}_N,\hat{\bm B}_N)$ in the non mean-field setup of Subsection \ref{bdegr} (Theorem \ref{bdgr}), we will assume that condition \eqref{bddegr} does not hold, i.e. we will assume the following \textit{mean-field} condition:
\begin{equation}\label{unbddegr}
    \lim_{N\rightarrow \infty} \frac{1}{N}\sum_{1\le i,j\le N} a_{ij}^2 = 0.
\end{equation} 
Note that if the coupling matrix is the scaled adjacency of a graph \eqref{adjacency_coupling_defn}, then the mean field condition says that the average degree of the graph goes to $\infty$. Once again, in view of Theorem \ref{thm:theorem1} (b), we just need to show that $T_N(\bm X)^{-1} = O_\p(1)$. The sequence of measures $\mu_N := N^{-1}\sum_{i=1}^N \delta _{R_i}$ are supported on the compact set $[0,\gamma]$, and so by Prokhorov's theorem and possibly passing to subsequences, we can assume that $\mu_N\xrightarrow{w} \mu$ for some probability measure $\mu$ on $[0,\gamma]$. The dominated convergence theorem now implies that:
\begin{equation}\label{meascv}
    \lim_{N\rightarrow\infty} \frac{1}{N} \sum_{i=1}^N (R_i-\bar{R})^2 = \int_{[0,\gamma]} (\theta - \e_\mu(\theta))^2 d \mu(\theta).
\end{equation}
Using \eqref{TNderv} we have
\begin{eqnarray}\label{eq:tn2}
\notag T_N(\bm X) &=& \frac{q}{N}\sum_{i=1}^N \sum_{r=1}^q\left(m_{i,r}(\bm X) - \overline{m}_r(\bm X) - \frac{1}{q}(R_i-\bar{R})\right)^2\\ &=& \frac{q}{2N^2}\sum_{1\le i,j\le N} \sum_{r=1}^q\left(m_{i,r}(\bm X) - m_{j,r}(\bm X) - \frac{1}{q}(R_i-R_j)\right)^2.
\end{eqnarray}
Setting $u_{i,r}(\bm X) := m_{i,r}(\bm X)- R_i/q$ and recalling the definition of $f_r$ and $\theta_{i,r}({\bm X})$  (see \eqref{frdef56} and \eqref{defcondprob881}),
$$f_r(t_1,\ldots,t_q) = \frac{\exp\{\beta t_r + B_r \}}{\sum_{s=1}^q \exp\{\beta t_s + B_s \}}~,$$ we have:
\begin{eqnarray*}
&&|\theta_{i,r}(\bm X) - \theta_{j,r}(\bm X)|\\&=& |f_r(m_{i,1}(\bm X),\ldots, m_{i,q}(\bm X)) - f_r(m_{j,1}(\bm X),\ldots, m_{j,q}(\bm X))|\\&=& |f_r(u_{i,1}(\bm X),\ldots,u_{i,q}(\bm X)) - f_r(u_{j,1}(\bm X),\ldots,u_{j,q}(\bm X))|\\&\le& \beta \sum_{s=1}^q |u_{i,s}(\bm X) - u_{j,s}(\bm X)| 
\end{eqnarray*}
where the last step follows from mean-value theorem, and the fact that all the partial derivatives of $f_r$ are bounded by $\beta$. Hence, we have:
\begin{eqnarray*}
\frac{1}{N}\sum_{i=1}^N \sum_{r=1}^q (\theta_{i,r}(\bm X) - \bar{\theta}_r(\bm X))^2 &=&\frac{1}{2N^2}\sum_{1\le i,j\le N} \sum_{r=1}^q (\theta_{i,r}(\bm X) - \theta_{j,r}(\bm X))^2\\&\le& \frac{q^2 \beta^2}{2N^2} \sum_{1\le i,j\le N}  \sum_{s=1}^q (u_{i,s}(\bm X) - u_{j,s}(\bm X))^2\\&=& q\beta^2 T_N(\bm X),
\end{eqnarray*}
where the last equality uses \eqref{eq:tn2}.
Hence, for showing that $T_N(\bm X)^{-1} = O_\p(1)$, it is enough to show that:
\begin{equation}\label{betameancd}
\left[\sum_{i=1}^N \sum_{r=1}^q (\theta_{i,r}(\bm X) - \bar{\theta}_r(\bm X))^2\right]^{-1} = O_\p(N^{-1})~.
\end{equation}
Towards this, let $\mathcal{S}_{N,q}$ denote the set of all $\bm y := ((y_{i,r}))_{i\in [N], r\in [q]} \in [0,1]^{Nq}$, such that $\sum_{r=1}^q y_{i,r} =1$ for all $i\in [N]$. Define functions $h_N: \mathcal{S}_{N,q}\to \mathbb{R}$ and $I_N: \mathcal{S}_{N,q}\to \mathbb{R}$ as:
\begin{equation}\label{hNdef96}
  h_N(\bm y) := \frac{\beta}{2} \sum_{1\le i,j\le N} \sum_{r=1}^q a_{ij} y_{i,r}y_{j,r} + \sum_{i=1}^N \sum_{r=1}^{q} B_r y_{i,r}\quad\text{and}\quad I_N(\bm y) := \sum_{i=1}^N \sum_{r=1}^q y_{i,r}\log y_{i,r}~.  
\end{equation}
 Also, define
 \begin{equation}\label{psiNdef96}
    \psi_N(\bm y) := h_N(\bm y) - I_N(\bm y) .
 \end{equation}
 Then, we have:
\begin{equation}\label{partder4460}
    \frac{\partial \psi_N(\bm y)}{\partial y_{i,r}} = \beta \sum_{j=1}^N a_{ij} y_{j,r} + B_r - 1-\log y_{i,r}~.
\end{equation}
Denoting $\nabla_{\cdot r} (\bm y) := ((\partial \psi_N/\partial y_{i,r}))_{i\in [N]}$, we have:
\begin{equation}\label{2ndstep556}
    \textcolor{black}{\nabla_{\cdot r} }(\bm y) = \beta \bm A_N\bm y_{\cdot r} + (B_r -1)\boldsymbol{1}_N - \log \bm y_{\cdot r}
\end{equation}
 where $\bm y_{\cdot r} := ((y_{i,r}))_{i\in [N]}$. Also, define $\overline{\nabla}(\bm y) \in \mathbb{R}^N$ and $\widetilde{\nabla}(\bm y) \in \mathbb{R}^{Nq}$ as:
 \begin{equation}\label{deftildedelt}
       \overline{\nabla}(\bm y) := \frac{1}{q}\sum_{r=1}^q \textcolor{black}{\nabla_{\cdot r} }(\bm y) \quad\text{and}\quad \widetilde{\nabla} (\bm y) := (\onb(\bm y),\ldots,\onb(\bm y)). 
 \end{equation}
    Note that $\widetilde{\nabla}(\bm y)$ is obtained from $\nabla \psi_N(\bm y)$ by replacing each row $(\partial \psi_N(\bm y)/\partial y_{i,r})_{r\in [q]}$ by the constant vector of its average. Therefore, whenever some vector $\bm y \in \mathcal{S}_{N,q}$ satisfies the conditions:
    \begin{equation}\label{3rdineq776}
       \sum_{i=1}^N \sum_{r=1}^q (y_{i,r}-\bar{y}_r)^2 \le N\delta_N\quad\text{and}\quad \min_{i\in [N], r\in [q]} y_{i,r} \ge \alpha
    \end{equation}
 for some non-negative sequence $\delta_N \rightarrow 0$ and $\alpha> 0$, then with $\widetilde{\bm y} \in \mathcal{S}_{N,q}$ defined as $\widetilde{y}_{i,r} := \bar{y}_r$ for all $i,r$, we have: 

\begin{eqnarray*}
    &&\|\textcolor{black}{\nabla_{\cdot r} }(\bm y)-\onb(\bm y)\|\\ 
    &\ge& \|\nabla_{\cdot r} (\widetilde{\bm y}) - \onb(\widetilde{\bm y})\| - \|\textcolor{black}{\nabla_{\cdot r} }(\bm y) - \nabla_{\cdot r} (\widetilde{\bm y})\| - \|\onb(\bm y) - \onb(\widetilde{\bm y})\|\\
    &\ge& \|\nabla_{\cdot r} (\widetilde{\bm y}) - \onb(\widetilde{\bm y})\| - \|\bm y_{\cdot r} - \widetilde{\bm y}_{\cdot r}\|\left(\beta \|\bm A_N\|_2 +\frac{1}{\alpha}\right)-\frac{1}{q}\sum_{s=1}^q \|\bm y_{\cdot s} - \widetilde{\bm y}_{\cdot s}\|\left(\beta \|\bm A_N\|_2 +\frac{1}{\alpha}\right)\\&\ge& \|\nabla_{\cdot r} (\widetilde{\bm y})- \onb(\widetilde{\bm y})\| - 2\sqrt{N\delta_N} (\beta\gamma + \alpha^{-1})\\&=& \|\nabla_{\cdot r} (\widetilde{\bm y})- \onb(\widetilde{\bm y})\| - o\left(\sqrt{N}\right).
\end{eqnarray*}
In the above display, the second inequality uses the expression of $\psi_N({\bm y})$ in \eqref{2ndstep556}, and the third inequality uses \eqref{3rdineq776}.
To bound the RHS above, again use \eqref{2ndstep556} to note that
\begin{eqnarray*}
    &&\|\nabla_{\cdot r} (\widetilde{\bm y})- \onb(\widetilde{\bm y})\|^2\\ &=& \sum_{i=1}^N \left[\beta \bar{y}_r R_i + B_r-1-\log \bar{y}_r - \frac{1}{q}\sum_{s=1}^q \left(\beta \bar{y}_s R_i + B_s-1-\log \bar{y}_s\right)\right]^2\\&\ge& \inf_{\bm t\in [\alpha,1]^q} \sum_{i=1}^N \left(\beta (t_r-\bar{t}) R_i + B_r-\frac{1}{q}\sum_{s=1}^q B_s-(\log t_r - q^{-1}\sum_{s=1}^q \log t_s) \right)^2\\&=& N \inf_{\bm t\in [\alpha,1]^q} \int_{0}^\gamma \left(\beta (t_r-\bar{t}) \theta + B_r-\frac{1}{q}\sum_{s=1}^q B_s - (\log t_r - q^{-1}\sum_{s=1}^q \log t_s)\right)^2~d\mu(\theta)\\ &+& o(N),
\end{eqnarray*} 
where the last inequality uses the weak convergence of $\mu_N$ to $\mu$.
Using the above two displays, and noting that $\theta_{i,r}({\bm X})\ge \alpha :=  q^{-1}\exp\left\{-\beta \gamma -2\|\bm B\|_\infty\right\}$, on the event $\sum_{i=1}^N \sum_{r=1}^q (\theta_{i,r}(\bm X) - \bar{\theta}_r(\bm X))^2 \le N\delta_N$ we have:
\begin{eqnarray}\label{fstpp3}
    &&N^{-1/2}\|\nabla_{\cdot r} (\boldsymbol{\theta}(\bm X)) - \onb (\boldsymbol{\theta}(\bm X))\|+o(1)\nonumber\\ &\ge& \sqrt{\inf_{\bm t\in [\alpha,1]^q} \int_{0}^\gamma \left(\beta (t_r-\bar{t}) \theta + B_r-\frac{1}{q}\sum_{s=1}^q B_s - (\log t_r - q^{-1}\sum_{s=1}^q \log t_s)\right)^2~d\mu(\theta)},\nonumber\\
\end{eqnarray}
 where $\boldsymbol{\theta}(\bm X) := ((\theta_{i,r}(\bm X)))_{i\in [N],r\in [q]}$. We now claim that the RHS of \eqref{fstpp3} is strictly positive for at least one $r\in [q]$. Given the claim,
on the event $\sum_{i=1}^N \sum_{r=1}^q (\theta_{i,r}(\bm X) - \bar{\theta}_r(\bm X))^2 \le N\delta_N$ we have the existence of a constant $c>0$, free of $N$, such that
\begin{equation*}\label{controm4}
    \max_{r\in [q]}\|\nabla_{\cdot r} (\boldsymbol{\theta}(\bm X))- \onb(\boldsymbol{\theta}(\bm X))\| \ge (c-o(1))\sqrt{N}.
\end{equation*}
Consequently, for any non-negative sequence $\delta_N \rightarrow 0$ we have
\begin{eqnarray*}
&&\p\left(\sum_{i=1}^N \sum_{r=1}^q (\theta_{i,r}(\bm X) - \bar{\theta}_r(\bm X))^2 \le N\delta_N\right)\\ &\le&  \p\left(\max_{r\in [q]}\|\nabla_{\cdot r} (\boldsymbol{\theta}(\bm X))- \onb (\boldsymbol{\theta}(\bm X))\| \ge  (c-o(1))\sqrt{N}\right)\to 0,
\end{eqnarray*}
 where the last limit uses Lemma~\ref{epnt376} (b). This
 establishes \eqref{betameancd} and completes the proof of Theorem \ref{irrgr}. 
 \\

 It thus remains to prove the claim that the RHS of \eqref{fstpp3} is strictly positive. To this effect,
assume by contradiction that the RHS  is $0$ for all $r\in [q]$.
By the dominated convergence theorem, the map
\begin{equation}\label{measureintg68}
    \bm t\mapsto \int_{0}^\gamma \left(\beta (t_r-\bar{t}) \theta + B_r-\frac{1}{q}\sum_{s=1}^q B_s - (\log t_r - q^{-1}\sum_{s=1}^q \log t_s)\right)^2~d\mu(\theta)
\end{equation}
is continuous on the compact set $[p,1]^q$, and hence, attains its infimum over $[p,1]^q$ at some point $\bm t^{(r)}\in [p,1]^q$. 
First suppose that $t_r^{(r)}\ne \bar{t}^{(r)}$ for some $r\in [q]$, whence we must have the following under $\mu$:
$$\theta \stackrel{a.s.}{=} \beta^{-1}(t_r^{(r)}-\bar{t}^{(r)})^{-1} \left[B_r-\frac{1}{q}\sum_{r=1}^q B_r - (\log t_r^{(r)} - q^{-1}\sum_{s=1}^q \log t_s^{(r)})\right]$$
a contradiction, since $\mu$ is not a degenerate measure in view of \eqref{meascv} and condition \eqref{irregap}.

This forces $t_r^{(r)} = \bar{t}^{(r)}$ for all $r\in [q]$. In this case, Jensen's inequality gives 
$$\log t_r^{(r)} - q^{-1}\sum_{s=1}^q \log t_s^{(r)} \ge 0$$
for all $r\in [q]$. Suppose further, that there exists $r\in [q]$ such that $$B_r - q^{-1}\sum_{s=1}^q B_s<0.$$ In this case, 
$B_r-\frac{1}{q}\sum_{s=1}^q B_s - (\log t_r^{(r)} - q^{-1}\sum_{s=1}^q \log t_s^{(r)}) < 0$
and hence, the integral \eqref{measureintg68} is strictly positive for that $r$, a contradiction.

The only case left for consideration is  when $B_r-\frac{1}{q}\sum_{s=1}^q B_s\ge 0$ for all $r\in [q]$. This however forces $B_1=\ldots=B_q = 0$, which was ruled out in the hypothesis of Theorem \ref{irrgr} (in fact, in this case, the minimum value is exactly $0$, attained at any constant vector $\bm t \in [\alpha,1]^q)$. This completes the proof of our claim, and subsequently, the theorem. \qed

\section{Proof of Theorem \ref{jinest}}\label{proof:jinest}
In this section, we prove Theorem \ref{jinest}. To begin with, note that by Lemma \ref{contg7}, the product measure $\bm m^N := \otimes_{i=1}^N \bm m$ is contiguous to the Curie-Weiss Potts measure $\p_{\beta,\bm B}^{\mathrm{CW}}$ in \eqref{cwpotts1}.  We next claim that for any $(\beta,\bm B) \in (0,\infty)\times \R^{q-1}$, the product measure $\p_{\beta,\bm B}^{\mathrm{CW}}\times \mathcal{G}(N,p)$ is contiguous to the measure $\p_{\beta,\bm B}^{\mathrm{ER}}$. Towards proving this, note that in view of Proposition 6.1 of \cite{estimation_14} it suffices to show that:
\begin{equation}\label{kldiv}
    D\left(\p_{\beta,\bm B}^{\mathrm{CW}}\times \mathcal{G}(N,p) || \p_{\beta,\bm B}^{\mathrm{ER}}\right) = O(1)
\end{equation}
where $D$ denotes the Kullback-Leibler divergence. Now, we have:
\begin{eqnarray*}
&&D\left(\p_{\beta,\bm B}^{\mathrm{CW}}\times \mathcal{G}(N,p) || \p_{\beta,\bm B}^{\mathrm{ER}}\right)\\ &=& \sum_{(\bm x, G)\in [q]^N\times \{0,1\}^{\binom{N}{2}}} (\p_{\beta,\bm B}^{\mathrm{CW}}\times \mathcal{G}(N,p))(\bm x, G)\log\frac{(\p_{\beta,\bm B}^{\mathrm{CW}}\times \mathcal{G}(N,p))(\bm x, G)}{\p_{\beta,\bm B}^{\mathrm{ER}}(\bm x,G)}\\&=& \e_{\mathcal{G}(N,p)} \left[\sum_{\bm x\in [q]^N} \p_{\beta,\bm B}^{\mathrm{CW}}(\bm x) \log \frac{\p_{\beta,\bm B}^{\mathrm{CW}}(\bm x)}{\p_{\beta,\bm B}^{\mathrm{ER}}(\bm x|G)}\right]\\&=& \e_{\mathcal{G}(N,p)} \log \left(\frac{Z_{\beta,\bm B}^{\mathrm{ER}}}{Z_{\beta,\bm B}^{\mathrm{CW}}}\right) - \sum_{\bm x \in [q]^N} \p_{\beta,\bm B}^{\mathrm{CW}}(\bm x) \left[\e_{\mathcal{G}(N,p)} \left(\widetilde{H}_{\beta,\bm B}^{\mathrm{ER}} (\bm x)\right) - \widetilde{H}_{\beta,\bm B}^{\mathrm{CW}} (\bm x)  \right]
\end{eqnarray*}
where $\widetilde{H}_{\beta,\bm B}^{\mathrm{ER}} (\bm x)$ and $\widetilde{H}_{\beta,\bm B}^{\mathrm{CW}} (\bm x)$ follow from the expression $$\frac{\beta}{2} \sum_{1\le i,j\le k} a_{ij}\mathbbm{1}_{x_i=x_j} + \sum_{i=1}^n \sum_{r=1}^{q} B_r\mathbbm{1}_{x_i=r}$$ with $a_{ij}$ replaced by \eqref{erny} and $1/N$, respectively. Note that $\e_{\mathcal{G}(N,p)} \left(\widetilde{H}_{\beta,\bm B}^{\mathrm{ER}} (\bm x)\right) = \widetilde{H}_{\beta,\bm B}^{\mathrm{CW}} (\bm x)$ and hence, we have by Jensen's inequality:
\begin{equation}\label{jenserny}
    D\left(\p_{\beta,\bm B}^{\mathrm{CW}}\times \mathcal{G}(N,p) || \p_{\beta,\bm B}^{\mathrm{ER}}\right) = \e_{\mathcal{G}(N,p)} \log \left(\frac{Z_{\beta,\bm B}^{\mathrm{ER}}}{Z_{\beta,\bm B}^{\mathrm{CW}}}\right) \le \log \left(\frac{\e_{\mathcal{G}(N,p)}(Z_{\beta,\bm B}^{\mathrm{ER}})}{Z_{\beta,\bm B}^{\mathrm{CW}}}\right).
\end{equation}
Finally, note that:
\begin{eqnarray*}
&&\e_{\mathcal{G}(N,p)}(Z_{\beta,\bm B}^{\mathrm{ER}}) \\&=& \sum_{\bm x \in [q]^N}\e_{\mathcal{G}(N,p)}\left[\exp\left(\frac{\beta}{2Np} \sum_{1\le i,j\le N} g_{ij}\mathbbm{1}_{x_i=x_j} + \sum_{i=1}^N \sum_{r=1}^{q} B_r\mathbbm{1}_{x_i=r}\right)\right]\\&=& \sum_{\bm x \in [q]^N} e^{\sum_{i=1}^N \sum_{r=1}^{q} B_r\mathbbm{1}_{x_i=r}}\prod_{1\le i<j\le N} \e_{\mathcal{G}(N,p)} \left[\exp\left(\frac{\beta}{Np}g_{ij}\mathbbm{1}_{x_i=x_j}\right)\right]\\&\le & \sum_{\bm x \in [q]^N} e^{\sum_{i=1}^N \sum_{r=1}^{q} B_r\mathbbm{1}_{x_i=r}}\prod_{1\le i<j\le N} \exp\left(\frac{\beta}{N}\mathbbm{1}_{x_i=x_j} + \frac{\beta^2}{8N^2p^2}\right)\\&\le& \sum_{\bm x \in [q]^N} e^{\beta^2/(16p^2)}\exp\left(\frac{\beta}{N}\sum_{1\le i<j\le N}\mathbbm{1}_{x_i=x_j} + \sum_{i=1}^N \sum_{r=1}^{q} B_r\mathbbm{1}_{x_i=r} \right)\\&\le& e^{\beta^2/(16p^2)}  Z_{\beta,\bm B}^{\mathrm{CW}}
\end{eqnarray*}
where in going from the second to the third line, we used the fact that for a Bernoulli random variable $Y$ with $\e Y=p$, 
$$\e e^{t(Y-p)} \le \exp\left(\frac{t^2}{8}\right),$$ which is a direct consequence of Hoeffding's lemma.
It now follows from \eqref{jenserny} that:
$$D\left(\p_{\beta,\bm B}^{\mathrm{CW}}\times \mathcal{G}(N,p) || \p_{\beta,\bm B}^{\mathrm{ER}}\right) \le \frac{\beta^2}{16p^2}$$
thereby establishing \eqref{kldiv} and completing the proof of our claim.

    Lemma \ref{contg7} coupled with our claim, establishes that the product measure $\nu:= \bm m^N \times \mathcal{G}(N,p)$ is contiguous to the measure $\p_{\beta,\bm B}^{\mathrm{ER}}$ for every $(\beta,\bm B) \in \Theta_{\bm m}$, thereby completing the proof of Theorem \ref{jinest}. \qed 

\section{Proof of Theorem \ref{partialestm}}\label{proof:partial8}

    
In this section, we prove Theorem \ref{partialestm}.

    \noindent (a) Fixing $\beta>0$, it follows from Lemma \ref{hesdet} part (b)  that $\nabla_{\bm B}^2 \ell_N (\beta,\bm B)$ is negative definite, hence ${\bm B}\mapsto \ell_N (\beta,\bm B)$ is strictly concave.
    Define:
    \begin{equation}\label{eq:a4n}
        A_{4,N} := \{\bm x \in [q]^N: ~\text{there exists}~r\in [q]~ \text{such that for all}~i \in [N], ~x_i \ne r\}.
    \end{equation}
  
   Lemma \ref{prevunp} part (b) gives that if $\bm X\in A_{4,N}^c$, then $\ell_N(\beta,\bm B)\to -\infty$ as $\|{\bm B}\|\to\infty$. Consequently, we have the existence of a  unique global maximizer at some $\bm B \in \mathbb{R}^{q-1}$, i.e. $\hat{\bm B}_N$ exists on the event $\{\bm X\in A_{4,N}^c\}$.
   
   Similarly, fixing ${\bm B}\in \R^{q-1}$, on the event $\{U_N({\bm X})\ne 0\}$, the function  $\beta\mapsto \ell_N(\beta,{\bm B})$ is stictly concave (see Lemma \ref{hesdet} part (c)). 
Define:
   \begin{equation}\label{eq:a3n}
     A_{2,N} := \{\bm x \in [q]^N:m_{i,x_i}(\bm x) = \min_{r\in [q]}m_{i,r}(\bm x)~\text{for all}~i\in [N]\} 
   \end{equation}
   \begin{equation}\label{eq:a2n}
       A_{3,N} := \{\bm x \in [q]^N: m_{i,x_i}(\bm x) = \max_{r\in [q]}m_{i,r}(\bm x)~\text{for all}~i\in [N]\}. 
   \end{equation}
   
   Once again, it follows from Lemma   \ref{prevunp} part (a)  that if $\bm X \in A_{2,N}^c \cap A_{3,N}^c$ and $U_N({\bm X})\ne 0$, then the function $\beta\mapsto\ell_N(\beta,\bm B)$ attains a unique global maximum at some $\beta \in \R$, i.e. $\hat{\beta}_N$ exists on the event $\{\bm X \in A_{2,N}^c \cap A_{3,N}^c\} \cap\{U_N(\bm X) \ne 0\}$.
   \\

   To complete the proof of part (a), noting that \eqref{unrateexact648} implies $\p(U_n({\bm X})=0)\to 0$,   it suffices to show that  $$\p(\bm X \in A_{4,N}) = o(1),\qquad \p(\bm X \in A_{2,N} \cup A_{3,N})= o(1).$$

   Towards showing that $\p(\bm X \in A_{4,N}) = o(1)$, note that by Lemma \ref{bp12} with $b_{itrs} := \mathbbm{1}_{t\ne r_0}$ and $g \equiv 1$, we have
\begin{equation}\label{cntre1}
    \sum_{i=1}^N\sum_{t \neq r_0}(\mathbbm{1}_{X_i=t}-\theta_{i,t}(\bm X))=O_{\p}(\sqrt{N}),
\end{equation}
where $\theta_{i,t}({\bm x})=\p(X_i=t|X_j=x_j,j\ne i)$ as in \eqref{defcondprob881}.
Assume $\bm X \in A_{4,N}$. Let $r_0\in [q]$ be such that $X_i \neq r_0$ for all $i$. This implies that for all $i$,  $\sum_{s \neq r_0}\mathbbm{1}_{X_i=s}=1$. Hence, we have: 
$$\sum_{s \neq r_0}(\mathbbm{1}_{X_i=s}-\theta_{i,s}(\bm X)) = 1- \sum_{s \ne r_0}\theta_{i,s}(\bm X) \ge q^{-1} \exp(-\beta \gamma-2 \|\bm B\|_\infty)~.$$ 
This says that the left side of \eqref{cntre1} is $\Omega(N)$ whenever $\bm X \in A_{4,N}$, thereby showing that $\p(\bm X\in A_{4,N}) = o(1)$. 

Next, towards showing that $\p(\bm X\in A_{3,N}) = o(1)$,
invoking Lemma \ref{bp12} applied with $b_{itrs} := \mathbbm{1}_{t=r}$, $\lambda := 0$ and $g(x)=x$, a union bound gives:
\begin{equation}\label{Opside68}
    \sum_{i=1}^N \sum_{r=1}^q \left(\mathbbm{1}_{X_i=r}-\theta_{i,r}(\bm X)\right)m_{i,r}(\bm X) = O_\p(\sqrt{N}).
\end{equation}
 Now, suppose that $\bm X \in A_{3,N}$. Then, $\sum_{r=1}^q m_{i,r}(\bm X)\mathbbm{1}_{X_i=r} = m_{i,X_i}(\bm X) = \max_r m_{i,r}(\bm X)$, and hence
\begin{eqnarray*}
    &&\sum_{i=1}^N \sum_{r=1}^q \left(\mathbbm{1}_{X_i=r}-\theta_{i,r}(\bm X)\right)m_{i,r}(\bm X)\\ &=& \sum_{i=1}^N \left(\max_r m_{i,r}(\bm X) - \sum_{r} m_{i,r}(\bm X)\theta_{i,r}\right)\\&\ge& \frac{1}{q^2(q-1)\gamma}\exp\left(-\beta\gamma -2\|\bm B\|_\infty\right) \sum_{i=1}^N \sum_{r<s} \left(m_{i,r}(\bm X) - m_{i,s}(\bm X)\right)^2\\&=& \frac{1}{q^2(q-1)\gamma}\exp\left(-\beta\gamma -2\|\bm B\|_\infty\right) NU_N(\bm X) = a_N \Omega_\p(\sqrt{N}).
\end{eqnarray*}
In the above display, the inequality on the third line invokes Lemma \ref{c7y} with the choices $$w_r=\theta_{ir}(\bm X)\ge q^{-1} \exp(-\beta \gamma -2\|{\bm B}\|_\infty)=\alpha,\quad  t_r=m_{i,r}({\bm X})\in [0,\gamma],$$
and the last equality holds
for some sequence $a_N \rightarrow \infty$, by \eqref{unrateexact648}. The above, combined with \eqref{Opside68} now gives $\p(\bm X\in A_{3,N}) = o(1)$. The proof of the fact $\p(\bm X\in A_{2,N}) = o(1)$ follows similarly, and we skip the details. 
    \vspace{0.1in}
    

    \noindent (b)~The proof will be carried out by an application of {Proposition} \ref{genZthlem} with $w_N(\bm B) = \nabla_{\bm B} \ell_N(\beta, \bm B)$.  To begin with, noting  the expression of $\nabla_{\bm B}\ell_N (\beta,\bm B)$ (see \eqref{parB}), it follows from Lemma \ref{bp12} that $\|\nabla_{\bm B}\ell_N (\beta,\bm B)\|_2 = O_\p(\sqrt{N})$. Consequently, Assumption (i) of Proposition \ref{genZthlem} follows with $a_N = \sqrt{N}$. Assumption (ii) of Proposition \ref{genZthlem} follows from Lemma \ref{hesdet} (b) with $h_N(\bm X) = N$. This completes the proof of part (b).

    \vspace{0.1in}

    \noindent (c)~The proof of this part is similar to the proof of part (b), so we skip it. 
     \vspace{0.1in}

    \noindent (d)~Define the set $$G_N(\delta) := \{\bm x\in [q]^N : U_N(\bm x) \le \delta\}.$$
Note that for all $\bm x\in G_N(\delta)$, we have by the Cauchy-Schwarz inequality,
$$\left|\frac{\beta}{2}\sum_{i=1}^N \sum_{r=1}^q x_{i,r}(m_{i,r}(\bm x) - \overline{m}_i(\bm x))\right| \le \frac{\beta}{2}\sqrt{qN}\sqrt{\frac{N U_N(\bm X)}{q}} \le \frac{\beta N}{2}\sqrt{\delta}~,$$ i.e.
$$\left|\frac{\beta}{2}\sum_{i=1}^N \sum_{r=1}^q x_{i,r}m_{i,r}(\bm x) - \frac{\beta N}{2q}\bar{R}\right| \le \frac{\beta N}{2}\sqrt{\delta}$$
where $\overline{m}_{i}(\bm x) := q^{-1} \sum_{r=1}^q m_{i,r}(\bm x)$. Therefore, we have:

\begin{eqnarray*}
&&\p_{\beta,\bm B}(\bm X \in G_N(\delta))\\ &=& \sum_{\bm x \in G_N(\delta)} \frac{\exp\left\{\frac{\beta}{2}\sum_{i=1}^N \sum_{r=1}^q x_{i,r}m_{i,r}(\bm x) + \sum_{i=1}^N \sum_{r=1}^{q} B_r x_{i,r}\right\}}{Z_N(\beta,\bm B)}\\&\le& \frac{\exp\left\{\frac{\beta N}{2} \left(\frac{\bar{R}}{q} + \sqrt{\delta}\right)\right\}}{Z_N(\beta,\bm B)}\sum_{\bm x\in G_N(\delta)} \exp\left\{\sum_{i=1}^N \sum_{r=1}^{q} B_r x_{i,r}\right\}
\end{eqnarray*}
where $R_i = \sum_{j=1}^N a_{ij}$ and $\bar{R} = \frac{1}{N}\sum_{i=1}^N R_i$.  Now, note that:
$$\sum_{\bm x\in [q]^n} \exp\left\{\sum_{r=1}^{q} \sum_{i=1}^N B_r x_{i,r}\right\} = \left(\sum_{x\in [q]} \exp\left\{\sum_{r=1}^{q} B_r\mathbbm{1}_{x=r}\right\}\right)^N = \left(\sum_{r=1}^{q} e^{B_r}\right)^N~.$$ Hence, we have:
$$\p_{\beta,\bm B}(\bm X \in G_N(\delta)) \le \frac{1}{Z_N(\beta,\bm B)} \exp\left\{\frac{\beta N}{2} \left(\frac{\bar{R}}{q} + \sqrt{\delta}\right) + N\log\left(\sum_{r=1}^{q} e^{B_r}\right)\right\},\quad\text{i.e.}$$

\begin{equation}\label{singleprt}
    \frac{1}{N} \log \p_{\beta,\bm B}(\bm X \in G_N(\delta)) \le \frac{\beta}{2} \left(\frac{\bar{R}}{q} + \sqrt{\delta}\right) + \log\left( 1+ \sum_{r=1}^{q-1} e^{B_r}\right) - \frac{\log Z_N(\beta,\bm B)}{N}\quad
\end{equation}
Now, it follows from (1.8) and (1.9) in \cite{mukherjeebasak} (by choosing $\mathfrak{q}_i(r)=t_r$) that:
\begin{equation}\label{cwgen222}
\frac{\log Z_N(\beta,\bm B)}{N} \ge \sup_{\boldsymbol{t} \in \mathcal{P}([q])} \left\{ 
 \frac{\beta}{2} \bar{R}\sum_{r=1}^q t_r^2 + \sum_{r=1}^{q} B_rt_r - \sum_{r=1}^q  t_r\log t_r \right\}
\end{equation}
We now divide the proof into two cases.
\vspace{0.1in}

 \noindent \textbf{Case-1:}~$\bm B \ne \boldsymbol{0}$. In this case, choosing $$t_s = \frac{e^{B_s}}{\sum_{r=1}^{q} e^{B_r}}\quad (\text{for}~s\in [q]),$$ Then, $\boldsymbol{t} := (t_1,t_2,\ldots,t_q) \in \mathcal{P}([q])$, we have:
\begin{eqnarray*}
&&\sum_{r=1}^{q} B_rt_r - \sum_{r=1}^q  t_r\log t_r\\&=& \frac{\sum_{r=1}^{q} B_r e^{B_r}}{\sum_{s=1}^{q} e^{B_s}} - \sum_{r=1}^{q}\frac{ e^{B_r} }{\sum_{s=1}^{q} e^{B_s}}\log\left(\frac{ e^{B_r} }{\sum_{s=1}^{q} e^{B_s}}\right) \\&=& \sum_{r=1}^{q} \frac{ e^{B_r} }{\sum_{s=1}^{q} e^{B_s}} \log\left(\sum_{s=1}^{q} e^{B_s}\right) \\&=& \log\left(\sum_{s=1}^{q} e^{B_s}\right).
\end{eqnarray*}
Using \eqref{cwgen222} for this choice of ${\bm t}$ gives 
$$ \frac{\log Z_N(\beta,\boldsymbol B)}{N} \ge \frac{\beta}{2}\bar{R} \sum_{r=1}^qt_r^2+ \log\left(\sum_{r=1}^{q} e^{B_r}\right)>\frac{\beta}{2q}\bar{R}+\log\Big(\sum_{r=1}^q e^{B_r}\Big)~,$$
where the last inequality uses the fact that ${\bm t}\ne (q^{-1},\ldots,q^{-1})$ for ${\bm B}\ne {\bf 0}$.
Along with \eqref{singleprt}, this gives the existence of $\delta > 0$ such that
$$\limsup_{N\rightarrow\infty} \frac{1}{N} \log \p_{\beta,\bm B}(\bm X \in G_N(\delta)) < 0$$
thereby establishing part (d) for the case $\bm B \ne \boldsymbol{0}$.
\\
\vspace{0.1in}

 \noindent \textbf{Case-2:}~ ${\bm B = \boldsymbol{0},  \bm \beta>\bm \beta_c(q).}$ ~
In this case, it follows from Theorem 2.3 (iv) in \cite{gandolfo} that the constant vector $\frac{1}{q}\boldsymbol{1}_q$ is not a maximizer of the function in the RHS of \eqref{cwgen222}). With ${\bm t}$ denoting a  maximizer, using \eqref{cwgen222} we again have
$$ \frac{\log Z_N(\beta,\boldsymbol B)}{N} \ge \frac{\beta}{2}(\sum_{r=1}^qt_r^2)\bar{R}+ \log\left(\sum_{r=1}^{q} e^{B_r}\right)>\frac{\beta}{2q}\bar{R}+\log\Big(\sum_{r=1}^q e^{B_r}\Big)~.$$
As before, this along with \eqref{cwgen222}) gives the existence of $\delta>0$ such that
$$\limsup_{N\rightarrow\infty} \frac{1}{N} \log \p_{\beta,\bm B}(\bm X \in G_N(\delta)) < 0,$$ thereby completing the proof of part (d).
     \vspace{0.1in}

\noindent (e)~  Finally, the non-existence of consistent estimators for the model $\p_{\beta,\boldsymbol{0}}^{\mathrm{ER}}$ for $\beta<\beta_c$ follows from the contuiguity of $
\frac{1}{q}\boldsymbol{1}_q \times \mathcal{G}(N,p)$ to the measure $\p_{\beta,\boldsymbol{0}}^{\mathrm{ER}}$ from Theorem \ref{jinest}, on noting that $(0,\beta_c)\times \{\boldsymbol{0}_{q-1}\}\subseteq \Theta_{(\frac{1}{q},\ldots,\frac{1}{q})}$. Uniqueness of the global maximizer follows from \cite[Thm 2.3 (iii)]{gandolfo}, and positive definiteness follows from  \cite[Lemma 4.7 (i)]{gandolfo}.

\section{Convergence of $Z$-estimators}\label{mestsec4}
In this section, we prove a general result about convergence of $Z$-estimators of the true parameter $\boldsymbol{\tau_0}$ in a parametric family $(\p_{\boldsymbol{\tau}})_{\boldsymbol{\tau} \in \Theta \subseteq \R^d}$ for some $d\ge 1$. To begin with, suppose that $\hat{\bm \tau}_N := \hat{\bm \tau}_N(\bm X) \in \R^d$ is an approximate root of a function $w_N: \R^d\times [q]^N \to \R^d$, i.e. the following is satisfied:
$$ w_N(\hat{\bm \tau}_N,{\bf X}) = 0.$$
 \begin{proposition}\label{genZthlem}
     Suppose that the objective function $w$ satisfies the following two conditions:
     \begin{enumerate}
          \item [(i)]  $\| w_N(\boldsymbol{\tau_0},{\bm X})\|_2 = O_P(a_N)$ for some non-negative sequence $a_N$, where $\boldsymbol{\tau_0}\in \Theta$ is the true parameter,
         \item [(ii)] There exists a function $c: \Theta \to \R$ which is continuous and strictly positive at $\boldsymbol{\tau_0}$, and a non-negative function $h_N:[q]^N\mapsto [0,\infty)$ on the sample space, satisfying: 
         $$\lambda_{\min} (-\nabla w_N({\boldsymbol{\tau}})) \ge c({\boldsymbol{\tau}}) h_N(\bm X)$$ for all ${\boldsymbol{\tau}} \in \Theta$.
     \end{enumerate}
     Then, as long as $a_N = o_P(h_N({\bf X}))$, we have:
     $$\|\hat{\bm \tau}_N - \boldsymbol{\tau_0}\|_2 = O_{\p_{\boldsymbol{\tau_0}}}\left(\frac{a_N}{h_N(\bm X)}\right).$$
 \end{proposition}
 
\begin{proof}
    For each $t\in [0,1]$, define
\[
\boldsymbol{\tau}_t := t \hat{\bm \tau}_N + (1-t)\boldsymbol{\tau_0}
\]
and introduce the following function $g_N: [0,1]\to \R$ :
\[
g_N(t) := (\hat {\bm \tau}_N - \boldsymbol{\tau_0})^\top w_N(\boldsymbol{\tau}_t,{\bf X}).
\]
Then we have
\[
g_N'(t)
= (\hat{\bm \tau}_N - \boldsymbol{\tau_0})^\top \nabla w_N(\theta_t,{\bf X})(\hat{\bm \tau}_N - \boldsymbol{\tau_0}).
\]
Setting
$Y_N := \|\hat{\bm \tau}_N - \boldsymbol{\tau_0}\|_2,$
using assumption (ii) along with the above display gives
\[
g_N'(t)
\le - c({\boldsymbol{\tau}}_t)\, h_N(\bm X)\, \|\hat{\bm \tau}_N - \boldsymbol{\tau_0}\|_2^2
= - c(\boldsymbol{\tau}_t)\, h_N(\bm X)\, Y_N^2 .
\]
Since $c$ is continuous at $\boldsymbol{\tau_0}$ and $c(\boldsymbol{\tau_0})>0$, there exists $r>0$ such that 
$$\inf_{\|\boldsymbol{\tau}-\boldsymbol{\tau_0}\|\le r}c(\boldsymbol{\tau})\ge \frac{c(\boldsymbol{\tau_0})}{2}>0.$$
For $t \le \frac{r}{Y_N}$, we have:
\[
\|\boldsymbol{\tau}_t - \boldsymbol{\tau_0}\|_2 = t Y_N \le r,
\]
and consequently
$$g_N'(t) \le -\frac{c({\bm \tau}_0)}{2} h_N(\bm X) Y_N^2.$$
Using the above bound gives
\begin{align*}
\bigl| g_N(1) - g_N(0) \bigr|
&= - \int_0^1 g_N'(t)\,dt  \\
&\ge - \int_0^{\min\{1,\; r/Y_N\}} g_N'(t)\,dt  \\
&\ge \int_0^{\min\{1,\; r/Y_N\}} \frac{c({\bm \tau}_0)}{2}\, h_N(\bm X)\,Y_N^2 \, dt  \\
&= \frac{c({\bm \tau}_0)}{2}\, h_N(\bm X)\,Y_N^2 \min\Bigl\{1,\frac{r}{Y_N}\Bigr\} \\
&= \frac{c({\bm \tau}_0)}{2}\, h_N(\bm X)\,Y_N \min\{Y_N,r\}.
\end{align*}
On the other hand, using the definition of $g_N(\cdot)$ gives
$$ |g_N(1)-g_N(0)|=|(\hat{\bm \tau}_N -{\bm \tau})^\top w_N({\bm \tau}_0,{\bf X})| \le Y_N \|w_N({\bm \tau}_0,{\bf X})\|_2= O_{P}\bigl(a_N Y_N\bigr).$$
where the last equality uses assumption (i).
Combining the above two displays gives
\begin{align*}
\frac{c({\bm \tau}_0)}{2}\, h_N(\bm X)\,Y_N \min\{Y_N,r\}
&= O_{P}\bigl(a_N Y_N\bigr),
\end{align*}
which implies that
\begin{align*}
\min\{Y_N,r\}
&= O_{\mathbb P}\!\left(\frac{a_N}{h_N(\bm X)}\right).
\end{align*}
Proposition \ref{genZthlem} now follows from the fact that $a_N/h_N(\bm X) \xrightarrow{P} 0$  and $r>0$ is fixed. 
\end{proof}

\section{Necessary tools for analyzing the derivatives of the log pseudo-likelihood}\label{techres}
The goal of this section is to develop necessary tools for bounding the first and second derivatives of the log pseudolikelihood, which  (in view of Proposition \ref{genZthlem}) is crucial in establishing consistency and rates of convergence of the MPL estimators. The 
first step towards doing this, is to bound the $L^2$-norm of the gradient of the log pseudo-likelihood with high probability. This can be achieved via the following general concentration inequality for sums of the random variables $\mathbbm{1}_{X_i=t}$ centered by their conditional means given $(X_j)_{j\ne i}$, which actually holds at all temperatures.

\begin{lem}\label{bp12}
	 For every constant $M >0$ (not depending on $N$) and every differentiable function $g : [-M, M] \to \R$ such that $g'$ is bounded,  there exists a constant $C >0$ (depending only on $g,\beta,{\bm B}, q,\gamma$) such that for every $t>0$, $\lambda\in [0,1]$, every array $(b_{itrs})_{i\in [N], t,r,s\in [q]} \in \R^{Nq}$ and every $r,s\in [q]$, we have:
	
	\begin{eqnarray}
	\p\left(\left| \sum_{i=1}^N \sum_{t=1}^q b_{itrs}\left(\mathbbm{1}_{X_i=t} - \theta_{i,t}({\bm X})\right) g(\bar{m}_i^{r,s}(\bm X))\right|\ge \sqrt{Nt}\right) \le 2 \exp\left(-\frac{CtN}{\|\bm L\|_2^2}\right) \label{eq:2.2}
	\end{eqnarray}
    where $\bar{m}_i^{r,s}(\bm X) := m_{i,r}(\bm X) - \lambda m_{i,s}(\bm X)$, $L_i = L_{i,r,s} := \sum_{t=1}^q |b_{itrs}|$, $\bm L := (L_1,\ldots,L_N)^\top$,
    and $\theta_{i,t}({\bm x})=\p(X_i=t|X_j=x_j, j\ne i)$ as in \eqref{defcondprob881}.
 %
\end{lem}

\begin{proof}
	   Fix $r,s \in [q]$, $c\in F$ and define $$G(\bm X) :=  \sum_{i=1}^N \sum_{t=1}^q b_{itrs} \left(\mathbbm{1}_{X_i=t} - \theta_{i,t}({\bm X})\right) g(\bar{m}_i^{r,s}(\bm X)).$$ 
       We now construct an exchangeable pair $({\bm X},{\bm X}')$ in the following way:
       
       Let $U$ be a discrete uniform random variable on $[N]$ independent of $\bm X$, and conditioned on $U=i$, we simulate $X_i'$ from the conditional distribution of $X_i$ given $(X_j)_{j\ne i}$. Then, we replace the $i^\mathrm{th}$ entry of $\bm X$ by $X_i'$ and call the resulting vector $\bm X'$. Then it is easy to check that $(\bm X,\bm X')$ is an exchangeable pair, i.e. $(\bm X,\bm X') \stackrel{D}{=} (\bm X',\bm X)$. Define:
	$$F(\bm x,\bm x') := \frac{1}{2}\sum_{i=1}^N \sum_{t=1}^q b_{itrs}(g(\bar{m}_i^{r,s}(\bm x)) + g(\bar{m}_i^{r,s}(\bm x))) (\mathbbm{1}_{x_i=t} - \mathbbm{1}_{x_i'=t}).$$
	Plugging in ${\bm x}={\bf X},{\bm x}'={\bm X}'$ we get
	\begin{align}\label{eq:F}F(\bm X,\bm X') =   g(\bar{m}_U^{r,s}(\bm X))\sum_{t=1}^q b_{Utrs} (\mathbbm{1}_{X_U= t}-\mathbbm{1}_{X_U'= t}),
    \end{align}which in turn gives
	\begin{equation*}
		\e\left[F(\bm X,\bm X')|\bm X\right] = \frac{1}{N}\sum_{i=1}^N  \sum_{t=1}^q b_{itrs} g(\bar{m}_i^{r,s}(\bm X))(\mathbbm{1}_{X_i= t}-\theta_{i,t}({\bm X})) = \frac{G(\bm X)}{N}~.
	\end{equation*}
Since $({\bm X},{\bm X}')$ is an exchangeable pair, we have  $\e G(\bm X) =\e F({\bm X}, {\bm X}')=0$.  Setting
$$\Delta(\bm X) := \frac{1}{2N}\e\left(|(G(\bm X) - G(\bm X'))F(\bm X,\bm X')|\Big|\bm X\right)~,$$
if we can show the existence of $K>0$ such that 
\begin{equation}\label{intgensumtr}
    \Delta(\bm X) \le \frac{K}{N^2}\|\bm L\|_2^2 ~,
\end{equation}
 then \cite[Thm 1.5]{chatterjee2007stein} gives, for $t\ge 0$,
$$\p\left(\frac{|G(\bm X)|}{N} \ge t\right) \le 2\exp\left\{-\frac{N^2t^2}{2K\|\bm L\|_2^2}\right\}$$ which further implies that $$\p(|G(\bm X)| \ge \sqrt{Nt}) \le 2\exp\left\{-\frac{Nt}{2K\|\bm L\|_2^2}\right\}$$
thereby establishing \eqref{eq:2.2}. Therefore, all that remains to complete the proof of \eqref{eq:2.2}, is to show \eqref{intgensumtr}. Towards this, for every $1\le i\le N$, $u\in [q]$ and $\bm x\in [q]^N$, define:
$$\bm x_{u}^{(i)} := (x_1,\ldots,x_{i-1},u,x_{i+1},\ldots,x_N)^\top.$$ Then using the definition of $\Delta({\bm X})$ we have:
\begin{eqnarray*}
	2N\Delta(\bm X) &=& \frac{1}{N}\sum_{i=1}^N \sum_{u=1}^q |(G(\bm X) - G(\bm X_u^{(i)}) F(\bm X, \bm X_u^{(i)})| \theta_{i,t}({\bm x})\\&\lesssim &   \frac{1}{N}\sum_{i=1}^N L_i \sum_{u=1}^q |(G(\bm X) - G(\bm X_u^{(i)}),
\end{eqnarray*}
where in the last inequality we use \eqref{eq:F} to get
$$ |F({\bm X},{\bm X}_u^{(i)})|\le \|g\|_\infty \sum_{t=1}^q |b_{itrs}|=\|g\|_\infty L_i.$$
 Now, 
 for any ${\bm X},{\bm Y}\in [q]^n$, using the definition of $G(\cdot)$,  we have:
\begin{eqnarray}\label{eq:sourav}
	&&|G(\bm X) - G(\bm Y)|\notag\\ &\le & \sum_{j=1}^N \sum_{t=1}^q |b_{jtrs}|\left|g(\bar{m}_{j}^{r,s}(\bm X)) X_{j,t} - g(\bar{m}_{j}^{r,s}(\bm Y)) Y_{j,t}\right|\notag\\ &+& \sum_{j=1}^N \sum_{t=1}^q |b_{jtrs}|\left|\theta_{j,t}(\bm X) g(\bar{m}_j^{r,s}(\bm X)) - \theta_{j,t}(\bm Y) g(\bar{m}_j^{r,s}(\bm Y))\right|\notag\\ &\le& \sum_{j=1}^N \sum_{t=1}^q  g(\bar{m}_j^{r,s}(\bm X)) |b_{jtrs}| \left|X_{j,t} - Y_{j,t}\right|+ \sum_{j=1}^N \sum_{t=1}^q Y_{j,t}|b_{jtrs}|  \left|g(\bar{m}_j^{r,s}(\bm X)) - g(\bar{m}_j^{r,s}(\bm Y))\right|\notag\\&+& \sum_{j=1}^N \sum_{t=1}^q |b_{jtrs}|\left|\theta_{j,t}(\bm X) g(\bar{m}_j^{r,s}(\bm X)) - \theta_{j,t}(\bm Y) g(\bar{m}_j^{r,s}(\bm Y))\right|\notag\\&\le & \|g\|_\infty\sum_{j=1}^N\sum_{t=1}^q  |b_{jtrs}| \left|X_{j,t} - Y_{j,t}\right| + \|g'\|_\infty\sum_{j=1}^N \sum_{t=1}^q Y_{j,t} |b_{jtrs}||\bar{m}_j^{r,s}(\bm X) - \bar{m}_j^{r,s}(\bm Y)|
\notag\\ &+& 	\sum_{j=1}^N \sum_{t=1}^q |b_{jtrs}|\left|\theta_{j,t}(\bm X) g(\bar{m}_j^{r,s}(\bm X)) - \theta_{j,t}(\bm Y) g(\bar{m}_j^{r,s}(\bm Y))\right|
\end{eqnarray}
We now bound each of the terms in the RHS of \eqref{eq:sourav}, for the special choice ${\bm Y}={\bm X}_u^{(i)}$. In this case, noting that $X_{j,t}=Y_{j,t}$ for all $j\ne i$, the first term in the RHS of \eqref{eq:sourav} can be bounded as follows:
\begin{align*}
    \sum_{j=1}^N\sum_{t=1}^q  |b_{jtrs}| \left|X_{j,t} - Y_{j,t}\right|= \sum_{t=1}^q  |b_{itrs}| \left|X_{i,t} - Y_{i,t}\right|\le  \sum_{t=1}^q  |b_{itrs}|=L_i.
\end{align*}
For bounding the second term in the RHS of \eqref{eq:sourav}, recalling that $\bar{m}_i^{r,s}({\bm X})=m_i^r({\bm X})-\lambda m_i^s({\bm X})$ we get
\begin{align*}
   \sum_{j=1}^N \sum_{t=1}^q Y_{j,t} |b_{jtrs}||\bar{m}_j^{r,s}(\bm X) - \bar{m}_j^{r,s}(\bm Y)|= &\sum_{j=1}^N \sum_{k=1}^N \sum_{t=1}^q Y_{j,t} a_{jk} |b_{jtrs}||X_{kr} - Y_{kr} -\lambda X_{ks} + \lambda Y_{ks}|\\
   = & \sum_{j=1}^N \sum_{t=1}^q a_{ji} |b_{jtrs}||X_{ir} - Y_{ir} - \lambda X_{is} + \lambda Y_{is}|\\
   \le &(1+\lambda)\sum_{j=1}^Na_{ji}L_j=(1+\lambda) A_{i*}{\bm L},
\end{align*}
where $A_{i*}$ denotes the $i^{\mathrm{th}}$ row of $\bm A$.
Proceeding to bound the third term in the RHS of \eqref{eq:sourav}, 
 define the function
$$\psi_{t}(\alpha_1,\ldots,\alpha_q) :=  g(\alpha_r-\lambda \alpha_s) \frac{\exp\{\beta \alpha_t+B_t\}}{\sum_{s=1}^q \exp\{\beta \alpha_s + B_s\}}~,$$
and note that 
$\|\nabla \psi_t\|_\infty$ is bounded, since $g$ has a bounded derivative. Using this definition we can write $\theta_{j,t}(\bm X) g(\bar{m}_j^{r,s}(\bm X)) = \psi_{t}(m_{j,1}(\bm X),\ldots,m_{j,q}(\bm X))$, and hence mean-value theorem gives
$$|\theta_{j,t}(\bm X) g(\bar{m}_j^{r,s}(\bm X)) - \theta_{j,t}(\bm Y) g(\bar{m}_j^{r,s}(\bm Y))|\lesssim  \sum_{w=1}^q | m_{j,w}(\bm X) - m_{j,w}(\bm Y)|.$$
Using this, the the third term in the RHS of \eqref{eq:sourav} can be bounded as follows:
\begin{eqnarray*}
	&&\sum_{j=1}^N \sum_{t=1}^q |b_{jtrs}|\left|\theta_{j,t}(\bm X) g(\bar{m}_j^{r,s}(\bm X)) - \theta_{j,t}(\bm Y) g(\bar{m}_j^{r,s}(\bm Y))\right|\\
    &\lesssim& \sum_{w=1}^q \sum_{j=1}^N \sum_{t=1}^q |b_{jtrs}| |m_{j,w}(\bm X) - m_{j,w}(\bm Y)|\\&=& \sum_{w=1}^q \sum_{j=1}^N \sum_{t=1}^q a_{ji}|b_{jtrs}||X_{i,w}-Y_{i,w}|\\&\le & \sum_{w=1}^q \sum_{j=1}^N \sum_{t=1}^q a_{ji} |b_{jtrs}|= q\sum_{j=1}^N a_{ji} L_j = qA_{i*}\bm L~.
\end{eqnarray*}
Combining the last three bounds to the RHS of \eqref{eq:sourav} we get $$|G(\bm X)-G(\bm X_u^{(i)})| \lesssim  L_i + A_{i*} \bm L,$$ and consequently,
\begin{eqnarray*}
   2N\Delta(\bm X) \lesssim \frac{1}{N}\sum_{i=1}^N L_i (L_i + A_{i*} \bm L) &=& \frac{1}{N}\left(\|\bm L\|_2^2 + \bm L^\top \bm A \bm L\right)\\&\le& \frac{1}{N} (1+\|\bm A\|_2) \|\bm L\|_2^2\le  \frac{1}{N} (1+\gamma) \|\bm L\|_2^2.
\end{eqnarray*}
This establishes \eqref{intgensumtr}, and consequently, \eqref{eq:2.2}.

\end{proof}

 The second crucial step towards establishing consistency and rates of convergence of the MPL estimators is to provide a lower bound on the minimum eigenvalue of the negative Hessian of the log pseudo-likelihood function. The following lemma achieves this.
  
\begin{lem}\label{hesdet}
 There exists a continuous strictly positive function $C_{q,\gamma}$ on $ [0,\infty)\times \R^{q-1}$, such that: 
\begin{enumerate}
    \item [(a)] $\lambda_{\min}\left(-\nabla^2 \ell_N(\beta,\bm B)\right) \ge C_{q,\gamma}(\beta,\bm B) N T_N(\bm X)$.

    \item [(b)] $\lambda_{\min}\left(-\nabla_{\bm B}^2 \ell_N(\beta,\bm B)\right) \ge  C_{q,\gamma}(\beta,\bm B) N$ and

    \item [(c)] $-\frac{\partial^2 \ell_N(\beta,\bm B)}{\partial \beta^2} \ge C_{q,\gamma}(\beta,\bm B) NU_N(\bm X)$
\end{enumerate}
where $T_N$ and $U_N$ are as defined in \eqref{eq:Tndefn} and \eqref{eq:Undefn}, respectively.
\end{lem}
\begin{proof}
It follows from \eqref{parbeta}, \eqref{parB} and a simple calculation, that the second-order partial derivatives of $\ell_N$ are given by: 
\begin{eqnarray*}\label{beta2}
    \frac{\partial^2 \ell_N(\beta,\bm B)}{\partial \beta^2} &=& -\sum_{i=1}^{N} \sum_{1\le a < b \le q} \left(m_{i,a}(\bm X) - m_{i,b}(\bm X)\right)^{2} \theta_{i,a} (\bm X)\theta_{i, b}(\bm X),\\
\label{Bbeta}
    \frac{\partial^2 \ell_N(\beta,\bm B)}{\partial B_s \partial \beta} &=& -\sum_{i=1}^{N} \sum_{a=1}^q \left(m_{i,s}(\bm X) - m_{i,a}(\bm X)\right) \theta_{i,a}(\bm X) \theta_{i, s} (\bm X)\quad(1\le s \le q-1),\\
\label{B2}
    \frac{\partial^2 \ell_N(\beta,\bm B)}{\partial B_s^2} &=& -\sum_{i=1}^{N} \sum_{a\ne s} \theta_{i, s} (\bm X) \theta_{i, a} (\bm X) \quad(1\le s \le q-1),\\
\label{BB}
    \frac{\partial^2 \ell_N(\beta,\bm B)}{\partial B_r\partial B_s} &=& \sum_{i=1}^{N}  \theta_{i,r} (\bm X) \theta_{i, s} (\bm X)\quad(1\le r\ne s \le q-1)
\end{eqnarray*}
Here $\theta_{i,t}({\bm x})=\p(X_i=t|X_j=x_j,j\ne i)$ is as in \eqref{defcondprob881}.
Using the above expressions, for any $\bm y := (y_0,y_1,\ldots,y_{q-1})\in \mathbb{R}^q$, we have:

\begin{eqnarray}
\notag&&-\bm y^\top \nabla^2\ell_N(\beta,\bm B)\bm y\notag\\
&=&
\sum_{i=1}^N \Bigg[
y_0^2 \sum_{1\le a<b\le q} \big(m_{i,a}(\bm X)-m_{i,b}(\bm X)\big)^2 \theta_{i,a}(\bm X)\theta_{i,b}(\bm X) \notag \\
&&\qquad
+2y_0 \sum_{s=1}^{q-1} y_s \sum_{a=1}^q \big(m_{i,s}(\bm X)-m_{i,a}(\bm X)\big) \theta_{i,a}(\bm X)\theta_{i,s}(\bm X) \notag\\
&&\qquad
+\sum_{s=1}^{q-1} y_s^2 \theta_{i,s}(\bm X)\sum_{a\neq s}\theta_{i,a}(\bm X)
-2\sum_{1\le r<s\le q-1} y_r y_s \theta_{i,r}(\bm X)\theta_{i,s}(\bm X)
\Bigg] \label{longcalculation7878}.
\end{eqnarray}

The first term in the RHS of \eqref{longcalculation7878} (without $y_0^2$) can be simplified as follows:
\begin{eqnarray*}
     &&\sum_{1\le a<b\le q} \big(m_{i,a}(\bm X)-m_{i,b}(\bm X)\big)^2 \theta_{i,a}(\bm X)\theta_{i,b}(\bm X)\\ &=&
 \sum_{1\le r<s\le q-1} \big(m_{i,r}(\bm X)-m_{i,s}(\bm X)\big)^2 \theta_{i,r}(\bm X)\theta_{i,s}(\bm X) \notag\\
&+&  \sum_{r=1}^{q-1} \big(m_{i,r}(\bm X)-m_{i,q}(\bm X)\big)^2 \theta_{i,r}(\bm X)\theta_{i,q}(\bm X) \notag\\
\end{eqnarray*}

The second term in the RHS of \eqref{longcalculation7878} (without $2y_0$) can be simplified as follows:

\begin{eqnarray*}
     &&\sum_{s=1}^{q-1} y_s \sum_{a=1}^q \big(m_{i,s}(\bm X)-m_{i,a}(\bm X)\big) \theta_{i,a}(\bm X)\theta_{i,s}(\bm X)\\
     &=& \sum_{s,r=1}^{q-1} \big(m_{i,s}(\bm X)-m_{i,r}(\bm X)\big)y_s\theta_{i,s}(\bm X)\theta_{i,r}(\bm X)\\ 
&+&  \sum_{s=1}^{q-1} \big(m_{i,s}(\bm X)-m_{i,q}(\bm X)\big)y_s\theta_{i,s}(\bm X)\theta_{i,q}(\bm X) \notag \\
&=&
 \sum_{1\le r<s\le q-1} \big(m_{i,s}(\bm X)-m_{i,r}(\bm X)\big)(y_s-y_r)\theta_{i,s}(\bm X)\theta_{i,r}(\bm X) \notag\\
&+&\sum_{s=1}^{q-1} \big(m_{i,s}(\bm X)-m_{i,q}(\bm X)\big)y_s\theta_{i,s}(\bm X)\theta_{i,q}(\bm X) \notag\\
\end{eqnarray*}

The third term in the RHS of \eqref{longcalculation7878} can be simplified as follows:
\begin{align*}
   &\sum_{s=1}^{q-1} y_s^2 \theta_{i,s}(\bm X)\sum_{a\neq s}\theta_{i,a}(\bm X)\\=&\sum_{s=1}^{q-1}y_s^2\theta_{i,s}({\bm X})\sum_{r=1}^{q-1}\mathbbm{1}_{r\ne s} ({\bm X})\theta_{i,r}({\bm X})+  \sum_{s=1}^{q-1} y_s^2 \theta_{i,s}(\bm X)\theta_{i,q}(\bm X)\\
  =& \sum_{1\le r<s\le q-1} (y_r^2+y_s^2)\theta_{i,r}(\bm X)\theta_{i,s}(\bm X)
+\sum_{s=1}^{q-1} y_s^2 \theta_{i,s}(\bm X)\theta_{i,q}(\bm X)\\
=&\sum_{1\le r<s\le q-1} (y_r-y_s)^2\theta_{i,r}(\bm X)\theta_{i,s}(\bm X)+2\sum_{1\le r<s\le q-1} y_ry_s\theta_{i,r}(\bm X)\theta_{i,s}(\bm X)\\
+&\sum_{s=1}^{q-1} y_s^2 \theta_{i,s}(\bm X)\theta_{i,q}(\bm X).
\end{align*}
Using the three  displays above the RHS of\eqref{longcalculation7878} simplifies to
\begin{eqnarray}
&&
\sum_{i=1}^N \Bigg[
y_0^2 \sum_{1\le r<s\le q-1} \big(m_{i,r}(\bm X)-m_{i,s}(\bm X)\big)^2 \theta_{i,r}(\bm X)\theta_{i,s}(\bm X) \notag\\
&&\qquad
+y_0^2 \sum_{r=1}^{q-1} \big(m_{i,r}(\bm X)-m_{i,q}(\bm X)\big)^2 \theta_{i,r}(\bm X)\theta_{i,q}(\bm X) \notag\\
&&\qquad
+2y_0 \sum_{1\le r<s\le q-1} \big(m_{i,r}(\bm X)-m_{i,s}(\bm X)\big)(y_r-y_s)\theta_{i,r}(\bm X)\theta_{i,s}(\bm X) \notag\\
&&\qquad
+2y_0 \sum_{r=1}^{q-1} \big(m_{i,r}(\bm X)-m_{i,q}(\bm X)\big)y_r\theta_{i,r}(\bm X)\theta_{i,q}(\bm X) \notag\\
&&\qquad
+\sum_{1\le r<s\le q-1} (y_r-y_s)^2 \theta_{i,r}(\bm X)\theta_{i,s}(\bm X)
+\sum_{r=1}^{q-1} y_r^2 \theta_{i,r}(\bm X)\theta_{i,q}(\bm X)
\Bigg]\notag \\
&=&
\sum_{i=1}^N \sum_{1\le r<s\le q-1}
\{(m_{i,r}(\bm X)-m_{i,s}(\bm X)) y_0 + y_r - y_s\}^2
\theta_{i,r}(\bm X)\theta_{i,s}(\bm X) \notag\\
&&\qquad
+\sum_{i=1}^N \sum_{r=1}^{q-1}
\{(m_{i,r}(\bm X)-m_{i,q}(\bm X))y_0 + y_r\}^2
\theta_{i,r}(\bm X)\theta_{i,q}(\bm X).\label{longcalculation78}
\end{eqnarray}

Now, it follows from Condition \eqref{as1} that $m_{i,r}(\bm X) \le \gamma$, and so
$$\theta_{i,r}({\bm X}) \ge \alpha := q^{-1}\exp\{-\beta \gamma - 2\|\bm B\|_\infty\} > 0~.$$ 
Hence, 
\begin{eqnarray}\label{feqn}
-\bm y^\top \nabla^2\ell_N(\beta,\bm B)\bm y &\ge& \alpha^2\sum_{i=1}^N \sum_{1\le r<s\le q-1} \{(m_{i,r}(\bm X) - m_{i,s}(\bm X)) y_0 + y_r - y_s\}^2\nonumber \\&+& \alpha^2\sum_{i=1}^N \sum_{r=1}^{q-1} \{(m_{i,r}(\bm X) - m_{i,q}(\bm X))y_0 + y_r\}^2\nonumber\\&=& \alpha^2 \bm y^\top \bm H_N \bm y 
\end{eqnarray}
where
$$\bm H_N :=\begin{pmatrix}
            \sum_{i}\sum_{1\le r<s\le q}(m_{i,r}(\bm X)-m_{i,s}(\bm X))^2 & \sum_i\bm v_i^T\\
            \sum_i\bm v_i & N(q\bm I-\bm J)
\end{pmatrix}$$
with $\bm I$ and $\bm J$ being the $(q-1)\times (q-1)$ identity matrix and matrix of all ones, respectively, and $$\bm v_i=(qm_{i,1}(\bm X)-\sum_r m_{i,r}(\bm X),\ldots,qm_{i,q-1}(\bm X)-\sum_r m_{i,r}(\bm X))^T \in \R^{q-1}.$$ The above equality \eqref{feqn} follows from repeating the above calculations, and  \eqref{longcalculation78} 
replacing $\theta_{i,r}({\bm X})$ by $1$ throughout. Thus $\bm H_N$ is non-negative definite, and denoting
 $\lambda_1\le\ldots\le \lambda_q$ to be eigenvalues of $\bm H_N$ we have: 
\begin{equation}\label{mr}
    \lambda_1 ~\mathrm{tr}(\bm H_N)^{q-1} = \lambda_1 \sum_{1\le i_1,\ldots,i_{q-1}\le q} \lambda_{i_1}\ldots\lambda_{i_{q-1}} \ge \prod_{i=1}^q \lambda_i = \mathrm{det}(\bm H_N).
\end{equation}
Next, using the fact that $(q{\bm I}-{\bm J})^{-1}=\frac{1}{q}({\bm I}+{\bm J})$, along with the expression for the determinant of a block partitioned matrix, we get:
\begin{eqnarray*}
    &&\mathrm{det}(\bm H_N)\\ &=&\mathrm{det}(N(q\bm I-\bm J)) \left[\sum_{i}\sum_{r<s}(m_{i,r}(\bm X)-m_{i,s}(\bm X))^2 - \frac{1}{Nq}\left(\sum_i\bm v_i\right)^\top (\bm I +\bm J)\left(\sum_i\bm v_i\right)\right].
\end{eqnarray*}
 Next, on noting that the $r^{\mathrm{th}}$ entry of $\sum_i v_ i$ is given by $Nq(\overline m_r(\bm X)-\frac1q\sum_{s=1}^q\overline m_s(\bm X))$ for $1\le r\le q-1$, we have:
\begin{eqnarray*}
&&\frac{1}{Nq}\Big(\sum_{i=1}^N \bm v_i\Big)^\top (\bm I+\bm J)\Big(\sum_{i=1}^N \bm v_i\Big)\\
&=& Nq\left[\sum_{r=1}^{q-1}\left(\overline m_r(\bm X)-\frac1q\sum_{s=1}^q\overline m_s(\bm X)\right)^2
+\left(\sum_{r=1}^{q-1}\left(\overline m_r(\bm X)-\frac1q\sum_{s=1}^q\overline m_s(\bm X)\right)\right)^2\right] \nonumber\\
&=& Nq\left[\sum_{r=1}^{q-1}\left(\overline m_r(\bm X)-\frac1q\sum_{s=1}^q\overline m_s(\bm X)\right)^2
+\left(-\overline m_q(\bm X)+\frac1q\sum_{s=1}^q\overline m_s(\bm X)\right)^2\right] \nonumber\\
&=& Nq\sum_{r=1}^q\left(\overline m_r(\bm X)-\frac1q\sum_{s=1}^q\overline m_s(\bm X)\right)^2 .
\end{eqnarray*}
Hence, we have:
\begin{eqnarray*}
    &&\mathrm{det}(\bm H_N)\\&=&  N^{q-1}q^{q-2}\left[\sum_{i}\sum_{r<s}(m_{i,r}(\bm X)-m_{i,s}(\bm X))^2 - Nq\sum_{r=1}^q \left(\overline{m}_r(\bm X) - \frac{1}{q}\sum_{s=1}^q \overline{m}_s(\bm X)\right)^2\right]\\&=& N^{q-1}q^{q-2}\left[\sum_{i}\sum_{r<s}(m_{i,r}(\bm X)-m_{i,s}(\bm X))^2 - N\sum_{r<s}\left(\overline{m}_r(\bm X) -  \overline{m}_s(\bm X)\right)^2\right]\\&=& N^q q^{q-2} T_N(\bm X),
\end{eqnarray*}
where the last equality uses \eqref{mixedmirGini}. Also, note that:
$$\mathrm{tr}(\bm H_N) = \sum_i \sum_{r<s} (m_{i,r}(\bm X) - m_{i,s}(\bm X))^2 + N(q-1)^2 \le N\left[\binom{q}{2}\gamma^2 + (q-1)^2\right] ~.$$
Hence, \eqref{mr} implies that $$\lambda_1 \ge C_{q,\gamma} N T_N(\bm X)~,$$ where $$C_{q,\gamma} := \frac{q^{q-2}}{\left(\binom{q}{2}\gamma^2 + (q-1)^2\right)^{q-1}}.$$
It now follows from \eqref{feqn} that 
\begin{equation*}
    -\bm y^\top \nabla^2\ell_N(\beta,\bm B)\bm y \ge C_{q,\gamma}\delta^2(\beta,\bm B,\gamma) \|\bm y\|_2^2 NT_N(\bm X)
\end{equation*}
for all $\bm y \in \mathbb{R}^q$, from which part (a) follows on taking $\bm y$ as an eigenvector of $-\nabla^2\ell_N(\beta,\bm B)$ corresponding to its minimum eigenvalue. Part (b) follows from \eqref{feqn} on taking $\bm y = (0,\bm \tilde{\bm y})$, 
and noting that the smallest eigenvalue of the $(q-1)\times (q-1)$ matrix $q\bm I - \bm J$ is $1$. Part (c) follows on taking $y$ in \eqref{feqn} to be the vector $(1,0,\ldots, 0)$ of length $q$. 
\end{proof}




\section{Proof of Lemma \ref{cwresult7}}\label{prcw7}
The proof of Lemma \ref{cwresult7} is based on reducing the Curie-Weiss Potts model to a product measure, by conditioning on a suitable random variable. Towards this, conditional on $\bm X \sim \p_{\beta,\bm B}^{\mathrm{CW}}$, let $Z_1,\ldots,Z_q$ be independent random variables with $Z_r \sim N(\bar{X}_r,(\beta N)^{-1})$. Define $\bm Z := (Z_1.\ldots,Z_q)$.
\begin{lem}\label{condindp}
If $\bm X$ follows the Curie-Weiss Potts model $\p_{\beta,\bm B}^{\mathrm{CW}}$, then the entries of $\bm X|\bm Z$ are independent and identically distributed, with 
$$\p_{\beta,\bm B}^{\mathrm{CW}}(\bm X_i=r|\bm Z)=\frac{e^{\beta Z_r+B_r}}{\sum_{s=1}^qe^{\beta Z_s+B_s}}.$$
\end{lem}

\begin{proof}
Recall from \eqref{cwpotts1} that:
$$\p_{\beta,\bm B}^{\mathrm{CW}}(\bm X) \propto \exp\left(\frac{N\beta}{2}\sum_{r=1}^q \bar{X}_r^2 + N\sum_{r=1}^{q} B_r\bar{X}_r\right)~.$$
Hence, we have:
{\color{black}\begin{eqnarray*}
\p_{\beta,\bm B}^{\mathrm{CW}}(\bm X,\bm Z) &\propto& \p_{\beta,\bm B}^{\mathrm{CW}}(\bm X) \prod_r \mathbb{P}(Z_r|\bm X)\\ &\propto& \exp\left\{\frac{N\beta}{2} \sum_r \left(\bar{X}_r^2 - (Z_r-\bar{X}_r)^2\right) + N\sum_{r=1}^{q}B_r\bar{X}_r\right\}\\ &=& \exp\left\{-\frac{N\beta}{2} \sum_r Z_r^2 + N\beta \sum_r \bar{X}_r Z_r + N\sum_{r=1}^{q}B_r\bar{X}_r\right\}~.
\end{eqnarray*}}
Therefore,
$$\p_{\beta,\bm B}^{\mathrm{CW}}(\bm X|\bm Z) \propto \exp\left\{\sum_{i=1}^N \sum_{r=1}^q X_{i,r}(\beta Z_r + B_r)\right\},$$ and so given ${\bm Z}$ the random variables $(X_1,\cdots,X_N)$ are iid, with
$$\p_{\beta,\bm B}^{\mathrm{CW}}(\bm X_i=r|\bm Z)=\frac{e^{\beta Z_r+B_r}}{\sum_{s=1}^qe^{\beta Z_s+B_s}}.$$ This completes the proof of Lemma \ref{condindp}.
\end{proof}

Returning to the proof of Lemma \ref{cwresult7}, let us define for every $r\in [q]$ and $\bm x \in [q]^N$:
$$S_{N,r}(\bm x) :=  \sum_{i=1}^N \left(m_{i,r}(\bm x) - \overline{m}_r(\bm x) - \frac{1}{q}(R_i-\bar{R})\right)^2~.$$
For ease of notation, we will abbreviate $S_{N,r}(\bm X)$ by $S_{N,r}$. Then, we have:
\begin{eqnarray*}
\mathbb{E}(S_{N,r}|\bm Z) &=& \sum_{i=1}^N \mathbb{E}\left(\left(m_{i,r}(\bm X) - \overline{m}_r(\bm X) - \frac{1}{q}(R_i-\bar{R})\right)^2 \Bigg|\bm {Z}\right)\\
&\ge & \sum_{i=1}^N \mathrm{Var}\left[m_{i,r}(\bm X) - \overline{m}_r(\bm X) - \frac{1}{q}(R_i-\bar{R}) \Bigg|\bm {Z}\right]\\
&=& \sum_{i=1}^N \mathrm{Var}\left[\sum_{j=1}^N  a_{ij} X_{j,r} - \frac{1}{N} \sum_{j=1}^N R_j X_{j,r} - \frac{1}{q}\sum_{j=1}^N \left(a_{ij}-\frac{R_j}{N}\right)\Bigg|\bm Z\right]  \\
&=& \sum_{i=1}^N \mathrm{Var}\left[\sum_{j=1}^N  \left(a_{ij}-\frac{R_j}{N}\right) \left(X_{j,r}-\frac{1}{q}\right)\Bigg|\bm Z\right]\\&=& \mathrm{Var}\left(X_{1,r}| \bm Z\right)\sum_{i=1}^N \sum_{j=1}^N \left(a_{ij}-\frac{R_j}{N}\right)^2
\end{eqnarray*}
where in going from the second to the third line in the above display, we have used the following identity:
$$\overline{m}_r(\bm X) := \frac{1}{N} \sum_{i=1}^N \sum_{j=1}^N a_{ij} X_{j,r} = \frac{1}{N} \sum_{j=1}^N X_{j,r} \sum_{i=1}^N a_{ij} = \frac{1}{N} \sum_{j=1}^N R_j X_{j,r}.$$
The RHS above can be bounded below, as follows: 
$$\sum_{i=1}^N \sum_{j=1}^N \left(a_{ij}-\frac{R_j}{N}\right)^2 = \sum_{i=1}^N \sum_{j=1}^N a_{ij}^2 -\frac{1}{N} \sum_{i=1}^N R_i^2 \ge \sum_{i=1}^N\sum_{j=1}^N a_{ij}^2 - \gamma^2 =\Omega(N),$$
where the last equality uses the non mean-field condition \eqref{bddegr}.
Hence, we have shown that
\begin{align}\label{eq:mean}
\e(S_{N,r}|\bm Z)~ \ge ~\Omega(N)~\mathrm{Var}(X_{1,r}|\bm Z)~.
\end{align}

Next, we will show that for any sequence $\beta_N$ which is bounded away from $0$ and $\infty$, the quantity $\mathrm{Var}_{\beta_N,\bm B}^{\mathrm{CW}}(X_{1,r}|\bm Z)$ is bounded away from $0$. For this, it only suffices to show that $\p_{\beta_N,\bm B}^{\mathrm{CW}}(X_1=r|\bm Z)$ is bounded away from both $0$ and $1$. To this effect, Lemma \ref{condindp} gives:
$$\p_{\beta_N,\bm B}^{\mathrm{CW}}(X_1=r|\bm Z) = \frac{\exp\{\beta_N Z_r + B_r\}}{\sum_{s=1}^q \exp\{\beta_N Z_s + B_s \}}=\frac{1}{1+\sum_{s\neq r}\exp\{\beta_N(Z_s-Z_r)+B_s-B_r\}}~.$$ Define the event $\mathscr{E} := \{\bm Z \in [-2,2]^q\}$, and note that on $\mathscr{E}$ we have:
\begin{eqnarray*}
\frac{1}{1+(q-1)\exp\{4\bar{\beta} + 2\|{\bf B}\|_\infty \}} &\le& \p_{\beta_N,\bm B}^{\mathrm{CW}}(X_1=r|\bm Z)
\le  \frac{1}{1+(q-1) \exp\{-4\bar{\beta}-2\|{\bf B}\|_\infty\}}.
\end{eqnarray*}
where $\bar{\beta}<\infty$ is an upper bound of $\beta_N$. The above bound along with \eqref{eq:mean} gives:
\begin{equation*}
    \e_{\beta_N,\bm B}^{\mathrm{CW}}(S_{N,r}|\bm Z) \ge C N
\end{equation*}
on the event $\mathscr{E}$, for some constant $C>0$  not depending on $N$.
\\

We will now show, using McDiarmid's inequality, that $S_{N,r}$ concentrates around $\e_{\beta_N,\bm B}^{\mathrm{CW}}(S_{N,r}|\bm Z)$. For this,
fix two vectors $\bm x$ and $\bm x'$ in $[q]^N$ that differ exactly in the $k^{\mathrm{th}}$ coordinate. Then, we have:
\begin{eqnarray*}
   && \left|S_{N,r}(\bm x) - S_{N,r}(\bm x')\right|\\ &=& \left|\sum_{i=1}^N \left[\left(m_{i,r}(\bm x) - \overline{m}_r(\bm x) - \frac{1}{q}(R_i-\bar{R})\right)^2 - \left(m_{i,r}(\bm x') - \overline{m}_r(\bm x') - \frac{1}{q}(R_i-\bar{R})\right)^2\right]\right|\\ 
   &\le& 4\gamma \sum_{i=1}^N\left|m_{i,r}(\bm x) - \overline{m}_r(\bm x) - m_{i,r}(\bm x') + \overline{m}_r(\bm x')\right|. 
\end{eqnarray*}
Also we have:
\begin{eqnarray*}
  &&\sum_{i=1}^N \left|m_{i,r}(\bm x) - \overline{m}_r(\bm x) - m_{i,r}(\bm x') + \overline{m}_r(\bm x')\right|\\ &=& \sum_{i=1}^N\left|a_{ik}(x_{k,r}-x_{k,r}') - \frac{1}{N}\sum_{\ell=1}^N  a_{\ell k} (x_{k,r}-x_{k,r}')\right|\\&\le& \sum_{i=1}^N a_{ik} +  \sum_{\ell=1}^N a_{\ell k} \le 2\gamma .
\end{eqnarray*}
Combining the above two, we thus have:
$$\left|S_{N,r}(\bm x) - S_{N,r}(\bm x')\right| \le 8\gamma^2.$$
Hence, by McDiarmid's inequality, we have the following on the event $\mathscr{E}$:
\begin{eqnarray*}
    \p_{\beta_N,\bm B}^{\mathrm{CW}}\left(S_{N,r} <\frac{CN}{2}\Bigg | \bm Z\right)  &\le& \p_{\beta_N,\bm B}^{\mathrm{CW}}\left(S_{N,r} <\e_{\beta_N,\bm B}^{\mathrm{CW}}(S_{N,r}|\bm Z) - \frac{C N}{2}\Bigg | \bm Z\right) \\&\le&  \exp\left(-\frac{C^2 N^2}{128 N \gamma^4}\right) = \exp\left(-\frac{N C^2}{128\gamma^4}\right)
\end{eqnarray*}
Hence, we have:
\begin{eqnarray*}
\mathbb{P}_{\beta_N,\bm B}^{\mathrm{CW}}\left(S_{N,r} \le \frac{C N}{2}\right) &=& \mathbb{E}_{\beta_N,\bm B}^{\mathrm{CW}}~\mathbb{P}_{\beta_N,\bm B}^{\mathrm{CW}} \left(S_{N,r} \le \frac{C N}{2}~\Bigg|~ \bm Z\right) \\&\le& \mathbb{P}_{\beta_N,\bm B} (\mathscr{E}^c) + \mathbb{E}_{\beta_N,\bm B}^{\mathrm{CW}}~\mathbb{P}_{\beta_N,\bm B}^{\mathrm{CW}}\left(S_{N,r} \le \frac{CN}{2}~\Bigg|~ \bm Z\right)\mathbbm{1}_{\mathscr{E}}\\&\le & \mathbb{P}_{\beta_N,\bm B} (\mathscr{E}^c) + \exp\left(-\frac{N C^2}{128\gamma^4}\right).
\end{eqnarray*}
Also, we have
\begin{eqnarray*}
\mathbb{P}_{\beta_N,\bm B}^{\mathrm{CW}} (\mathscr{E}^c) &\le& q\max_{1\le r \le q} \mathbb{P}_{\beta_N,\bm B}(|Z_r| > 2) \\&=& q\max_{1\le r\le q} \mathbb{E}_{\beta_N,\bm B} ~\mathbb{P}_{\beta_N,\bm B} (|Z_r|>2 |\bm X)\\&\le & q\max_{1\le r\le q} \mathbb{E}_{\beta_N,\bm B}~\mathbb{P}_{\beta_N,\bm B} (|Z_r-\bar{X}_r|>1|\bm X)\quad(\text{since}~|\bar{X_r}|\le 1)\\&=& q\max_{1\le r\le q} \mathbb{E}_{\beta_N,\bm B}~ \mathbb{P}_{\beta_N,\bm B} \left(\sqrt{\beta_N N} |Z_r-\bar{X}_r|>\sqrt{\beta_N N}~\Big |\bm X\right)\\&=& q \mathbb{P}(|Z|>\sqrt{\beta_N N})\quad(\text{where}~Z\sim N(0,1))\\&\le& 2q e^{-\underline{\beta} N/2},
\end{eqnarray*}
where $\underline{\beta} > 0$ is a lower bound of $\beta_N$.
Combining all these, we conclude:
$$\mathbb{P}_{\beta_N,\bm B}^{\mathrm{CW}}\left(S_{N,r} \le \frac{C N}{2}\right)  \le e^{-KN}$$ for some constant $K >0$ (not depending on $N$), which invoking \eqref{TNderv} gives:
$$\mathbb{P}_{\beta_N,\bm B}^{\mathrm{CW}}\left(T_N(\bm X) < \frac{Cq}{2}\right)  \le q e^{-KN}\quad\implies\quad \mathbb{P}_{\beta_N,\bm B}^{\mathrm{CW}}\left({\bm X}\in E_N\left(\frac{Cq}{2}\right)\right) \le q e^{-KN},$$
where the set $E_N(\cdot)$ is as in \eqref{endef86}. 

To complete the proof, it suffices to verify that the
sequence $\beta_N=\beta\bar{R}$ is bounded above by $\bar{\beta}<\infty$, and  bounded below by $\underline{\beta}>0$. But this follows on using \eqref{as1} and \eqref{as2}, and recalling that $\beta>0$.


\section{Necessary results for proving Theorem \ref{jinest}}\label{sec:necth1p3}
In this section, we state and prove some lemmas which will be used to verify Theorem \ref{jinest}. 

 \begin{lem}\label{lem:concentration_potts_residual}
Let $\bm X=(X_1,\dots,X_N)\in[q]^N$ be distributed according to the Curie--Weiss Potts
measure $\p_{\beta,\bm B}^{\mathrm{CW}}$ \eqref{cwpotts1}. Let
\begin{equation*}
f_r(\bm t)=\frac{\exp\{\beta t_r+B_r\}}{\sum_{s=1}^q \exp\{\beta t_s+B_s\}},
\qquad \bm t\in\mathbb R^q,\ r\in[q]
\end{equation*}
be as in \eqref{frdef56}, and $f:\R^q\to \mathcal{P}([q])$ (see \eqref{defpqfirst6}) as $f(\bm t) := (f_1(\bm t),
\ldots, f_q(\bm t))$.
Then for every $r\in[q]$ and every $t\ge0$,
\begin{equation*}\label{eq:final_concentration}
\p_{\beta,\bm B}^{\mathrm{CW}}\!\left(\big\|\bar{\bm X}-f(\bar{\bm X})\big\|_\infty > \frac{\beta}{2N}+t\right)
\ \le\
2q\exp\!\left(-\frac{2N t^2}{2+\beta}\right).
\end{equation*}
\end{lem}

\begin{proof}
To begin with, use \eqref{defcondprob881} to note that for all $i\in [N], r\in [q]$ we have:
\begin{equation}\label{eq:cond_prob}
\mathbb P(X_i=r\mid \bm X_{-i})=f_r(\bar{\bm X}^{(i)}), \text{ where }
\bar X^{(i)}_r:=\frac{1}{N}\sum_{j\ne i}X_{j,r}\text{ and }
\bar{\bm X}^{(i)}:=(\bar{ X}^{(i)}_1,\dots,\bar X^{(i)}_q).
\end{equation}
Define $\bm X'$ from $\bm X$ by choosing an index $I$ uniformly at random from $[N]$, and updating the $I^{\mathrm{th}}$ entry of $\bm X$ by a sample from the conditional distribution $\p_{\beta,\bm B}^{\mathrm{CW}}(X_I = \cdot|\bm X_{-I}) = f_{\cdot}(\bar{\bm X}^{(I)})$, keeping the other entries unchanged. It is straightforward to verify that $(\bm X, \bm X')$ is an exchangeable pair. Fix $r\in [q]$
and define the antisymmetric function
\[
F_r(\bm X, \bm Y):=N(\bar X_r-\bar Y_r),\quad\text{which gives}\quad 
F_r(\bm X,\bm X')=X_{I,r}-X_{I,r}'\in [-1,1].\]
Set $u_r(\bm X):=\mathbb E\big[F_r(\bm X,\bm X')\mid \bm X\big],$
and use the tower property to get
\begin{align}
u_r(\bm X)
=\frac1N\sum_{i=1}^N \Big[X_{i,r}
-\mathbb P(X_i=r\mid \bm X_{-i})\Big]
=\bar X_r-\frac1N\sum_{i=1}^N f_r(\bar{\bm X}^{(i)}),\label{eq:f_def}
\end{align}
where we used \eqref{eq:cond_prob} in the last equality. Next, observe that for $\bm t\in\mathbb R^q$ and $r,s\in[q]$ we have
\begin{equation}\label{Jacobian68}
\frac{\partial f_r}{\partial t_s}(\bm t)
=\beta\,f_r(\bm t)\big(\mathbbm{1}\{r=s\}-f_s(\bm t)\big).
\end{equation}
Hence
\[
\|\nabla f_r(\bm t)\|_1
=\sum_{s=1}^q\left|\frac{\partial f_r}{\partial t_s}(\bm t)\right|
=\beta f_r(\bm t)\Big(1-f_r(\bm t) +\sum_{s\ne r}f_s(\bm t)\Big)
=2\beta f_r(\bm t)(1-f_r(\bm t))\le \frac{\beta}{2}.
\]
Consequently, we have:
$$f_r(\bar{\bm X}^{(i)}) - f_r(\bar{\bm X}^{'(i)}) \le \frac{\beta}{2} \max_{s\in [q]} |\bar{X}_s^{(i)} - \bar{X}_s^{'(i)}| \le \frac{\beta}{2}\cdot \frac{1}{N} = \frac{\beta}{2N}.$$
Hence, it follows from \eqref{eq:f_def} that:
\begin{equation*}
    |u_r(\bm X) - u_r(\bm X') |\le \frac{1}{N} + \frac{\beta}{2N}.
\end{equation*}
Consequently, setting
\[
v_r(\bm X):=\frac12\,\mathbb E\Big(|u_r(\bm X)-u_r(\bm X')|\ |F_r(\bm X,\bm X')|\ \Big|\ \bm X\Big),
\]
the above display gives
$v_r(\bm X) \le \frac{2+\beta}{4N}$. Hence, by Theorem 1.5 in \cite{chatterjee2007stein} we have:

\begin{equation*}\label{eq:fr_conc}
\mathbb P_{\beta,\bm B}^{\mathrm{CW}}(|u_r(\bm X)|> t)\le 2\exp\!\left(-\frac{2N t^2}{2+\beta}\right)
\end{equation*}
Next, define $g_r(\bm X) := \bar{X}_r - f_r(\bar{\bm X})$, and use \eqref{eq:f_def} to get:
\begin{eqnarray*}
    |g_r(\bm X) - u_r(\bm X)| \le \frac{1}{N}\sum_{i=1}^N \left|f_r(\bar{\bm X}^{(i)}) - f_r(\bar{\bm X})\right|\le \frac{1}{N}\sum_{i=1}^N \frac{\beta}{2}\max_{s\in [q]} |\bar{X}_s^{(i)} - \bar{X}_s|\le \frac{\beta}{2N}~.
\end{eqnarray*}
Therefore, we have:
$$\p_{\beta,\bm B}^{\text{CW}}\left(|g_r(\bm X)| > \frac{\beta}{2N} + t\right) \le \mathbb P_{\beta,\bm B}^{\mathrm{CW}}(|u_r(\bm X)|> t)\le 2\exp\!\left(-\frac{2N t^2}{2+\beta}\right).$$
The proof of Lemma \ref{lem:concentration_potts_residual} is now complete by a further union bound.
\end{proof}  

\begin{lem}\label{inversefunc}
Suppose the following assumptions hold:
\begin{enumerate}
\item[(i)]
Let ${\bm m}\in P([q])$ (see \eqref{defpqfirst6}) be a solution to the equation $\xi({\bm t})=0$, where $\xi({\bm t}):={\bm t}-f({\bm t})$, and $f:\R^q\mapsto P([q])$ is as in \eqref{defcondprob881}.

\item[(ii)]

Setting $H = H_{\beta,\bm B} : P([q]) \to \mathbb{R}$ by
\[
H(\bm t) := \frac{\beta}{2} \sum_{r=1}^q t_r^2 
+ \sum_{r=1}^{q} B_r t_r 
- \sum_{r=1}^q t_r \log t_r,
\]
as in \eqref{Hdefn6882}, we have 
$\bm u^\top \nabla^2 H(\bm m) \bm u < 0$ for all $\bm u \in T^* \setminus \{\bf 0\}$, where
\begin{equation}\label{tstardef56}
T^* := \Big\{ \bm u \in \mathbb{R}^q : \sum_{r=1}^q u_r = 0 \Big\}.
\end{equation}
\end{enumerate}
 Then, the Jacobian operator $D\xi(\bm m)$ viewed as a linear map on the domain $T^*$, is injective. 
\end{lem}

\begin{proof}
Since $\bm m=f(\bm m)$ (assumption (i)), using \eqref{Jacobian68} the Jacobian of $f$ at $\bm m$ is given by
\[
\frac{\partial f_r}{\partial t_s}(\bm m)
= \beta m_r(\mathbf 1\{r=s\} - m_s),
\]
which can be written in matrix form as
\[
Df(\bm m) = \beta(\operatorname{diag}(\bm m) - \bm m \bm m^\top) =: \beta \Sigma(\bm m)\Rightarrow
D \xi(\bm m) = I - Df(\bm m) = I - \beta \Sigma(\bm m).
\]
On the other hand, for $\bm t\in P([q])$ with strictly positive coordinates, the Hessian of $H$ at ${\bm m}$ is given by
\[
\nabla^2 H(\bm m) = \beta I - \operatorname{diag}(1/m_1,\dots,1/m_q).
\]

Now, for any $\bm u \in \mathbb{R}^q$ we have:
\[
\bm u^\top \Sigma(\bm m) \bm u
= \sum_{r=1}^q m_r u_r^2 
- \Big(\sum_{r=1}^q m_r u_r \Big)^2
= \frac12 \sum_{r,s=1}^q m_r m_s (u_r - u_s)^2 \ge 0.
\]
Thus $\Sigma(\bm m)$ is positive semidefinite. Moreover,
$\bm u^\top \Sigma(\bm m) \bm u = 0$ if and only if $u_r=u_s$ for all $r,s$, which implies that the null space of $\Sigma(\bm m)$ is $\mathrm{span}\{\bm 1\}$.
Consequently, the map $\bm u \mapsto \Sigma(\bm m)\bm u$ restricted to
$T^*$ is invertible.
\\

Now, suppose that $\bm u \in T^*$. Then
\[
\Sigma(\bm m)\operatorname{diag}(1/\bm m)\bm u
=
(\operatorname{diag}(\bm m) - \bm m \bm m^\top)\operatorname{diag}(1/\bm m)\bm u
=
\bm u - \bm m \sum_{r=1}^q u_r
= \bm u,
\]
since $\sum_r u_r=0$. Using the explicit form of $\nabla^2 H(\bm m)$, this yields
\begin{eqnarray*}
    \Sigma(\bm m)\nabla^2 H(\bm m)\bm u
&=&
\Sigma(\bm m)\big(\beta I - \operatorname{diag}(1/\bm m)\big)\bm u\\
&=&
\beta \Sigma(\bm m)\bm u - \bm u\\
&=&
-(I - \beta \Sigma(\bm m))\bm u
=
- D \xi(\bm m)\bm u.
\end{eqnarray*}

Therefore, for all $\bm u \in T^*$,
\begin{equation*}
D\xi(\bm m)\bm u = - \Sigma(\bm m)\nabla^2 H(\bm m)\bm u. 
\end{equation*}


Now, suppose that $D \xi(\bm m)\bm u = \bf 0$ for some $\bm u \in T^*$. Then, the above display gives $\nabla^2H(\bm m) \bm u \in \mathrm{Null}(\Sigma(\bm m))$, which implies that $\nabla^2H(\bm m) \bm u = c\boldsymbol{1}$ for some constant $c$. Hence, $\bm u^\top \nabla^2H(\bm m) \bm u = 0$, which by assumption (ii) gives $\bm u = \bf 0$. Hence, $D\xi(\bm m)$ is injective on the domain $T^*$.
This completes the proof of Lemma \ref{inversefunc}.
\end{proof}

\begin{lem}\label{concxbarcwpotts78}
Assume the setting of the Curie--Weiss Potts model defined in \eqref{cwpotts1}. 
Suppose that $H_{\beta,\bm B}$ (as in \eqref{Hdefn6882}) admits a unique global maximizer $\bm m\in\mathcal P([q])$ (cf.\ Theorem~\ref{jinest}). 
Then for every $\delta>0$,
\[
\limsup_{N\to\infty}\frac1N
\log \p^{\mathrm{CW}}_{\beta,\bm B}
\big(\|\bar{\bm X}-\bm m\|_\infty\ge\delta\big)<0.
\]
In particular, 
\[
\bar{\bm X}\xrightarrow{P}\bm m.
\]
\end{lem}

\begin{proof}
To begin with, define:
\[
\mathcal P_{q,N}
:=
\left\{
\bm v\in\mathcal P([q]) :
Nv_r\in\mathbb Z \text{ for all } r
\right\},
\]
where $\mathbb Z$ denotes the set of all integers. For $\bm v\in\mathcal P_{q,N}$ define 
\[
A_N(\bm v)
=
\{\bm x\in[q]^N : \bar {\bm x}= \bm v\}.
\]
Then for any Borel set $G\subset\mathbb R^q$,
\begin{equation}\label{eqa1cp}
\p^{\mathrm{CW}}_{\beta,\bm B}(\bar {\bm X}\in G)
=
\frac{
\sum_{\bm v\in\mathcal P_{q,N}\cap G}
|A_N(\bm v)|
\exp\!\left\{
N\left(
\frac{\beta}{2}\sum_{r=1}^q v_r^2
+
\sum_{r=1}^q B_r v_r
\right)
\right\}
}{
\sum_{\bm v\in\mathcal P_{q,N}}
|A_N(\bm v)|
\exp\!\left\{
N\left(
\frac{\beta}{2}\sum_{r=1}^q v_r^2
+
\sum_{r=1}^q B_r v_r
\right)
\right\}
}.
\tag{A.1}
\end{equation}
By Lemma S.4.1 in \cite{snbh22},
\[
|A_N(\bm v)|
=
\exp\!\left\{
-N\sum_{r=1}^q v_r\log v_r
\right\}
\cdot e^{o(N)}
\]
uniformly over $\bm v\in\mathcal P_{q,N}$.
Substituting this estimate for $|A_N(\bm v)|$ into \eqref{eqa1cp} and using the fact that $|\mathcal P_{q,N}| \le (N+1)^q = e^{o(N)}$, we have:
\begin{equation}\label{tgr666}
   \p^{\mathrm{CW}}_{\beta,\bm B}(\bar {\bm X}\in G)
\le
e^{o(N)}
\frac{
\sup_{\bm v\in\mathcal P_{q,N}\cap G}
e^{N H_{\beta,\bm B}(\bm v)}
}{
\sup_{\bm v\in\mathcal P_{q,N}}
e^{N H_{\beta,\bm B}(\bm v)}
} 
\end{equation}
Now, note that
\begin{equation}\label{tgr255}  
\sup_{\bm v\in \mathcal P_{q,N}}
H_{\beta,\bm B}(\bm v)
\le
\sup_{\bm v\in \mathcal P([q])}
H_{\beta,\bm B}(\bm v).
\end{equation}
On the other hand, if $\bm m$ is the unique maximizer of
$H_{\beta,B}$ on $\mathcal P([q])$, then
by Lemma S.4.3 in \cite{snbh22}, there exists a sequence
$\bm v_N \in \mathcal P_{q,N}$
such that $\bm v_N \to \bm m$.
By continuity of $H_{\beta,B}$, we have:
\begin{equation}\label{tgr256}
\sup_{\bm v\in \mathcal P_{q,N}}
H_{\beta,\bm B}(\bm v)
\ge
H_{\beta,\bm B}(\bm v_N)
=
H_{\beta,\bm B}(\bm m)
=
\sup_{\bm v\in \mathcal P([q])}
H_{\beta,\bm B}(\bm v).
\end{equation}
Combining \eqref{tgr255} and \eqref{tgr256}, we have:
\[
\sup_{\bm v\in \mathcal P_{q,N}}
H_{\beta,\bm B}(\bm v)
\longrightarrow
\sup_{\bm v\in \mathcal P([q])}
H_{\beta,\bm B}(\bm v)
\qquad \text{as } N\to\infty.
\]
Hence, we have the following from \eqref{tgr666}:
\begin{equation}\label{tgr777}
    \limsup_{N\rightarrow \infty} \frac{1}{N} \log \mathbb \p^{\mathrm{CW}}_{\beta,\bm B}(\bar {\bm X}\in G) \le \sup_{\bm v \in \mathcal P([q]) \bigcap G}  H_{\beta,\bm B}(\bm v) - \sup_{\bm v \in \mathcal P([q])} H_{\beta,\bm B}(\bm v).
\end{equation}
Choosing $G := \{\bm v\in \mathcal P([q]) : \|\bm v - \bm m\|_\infty \ge \delta\}$, the RHS of \eqref{tgr777} equals a negative constant $C_\varepsilon$ (since $\bm m$ is the unique maximizer of $H_{\beta,\bm B}$ and $\bm m \notin G$). This completes the proof of Lemma \ref{concxbarcwpotts78}.
\end{proof}


\begin{lem}\label{c222}
Suppose that $\bm X$ is sampled from the Curie-Weiss Potts model \eqref{cwpotts1} and conditional on $\bm X$, let $Z_1,\ldots,Z_q$ be independent random variables with $Z_r \sim N(\bar{X}_r, (\beta N)^{-1})$. Let $\bm Z := (Z_1,\ldots,Z_q)$ and $\bar{\bm X} := (\bar{X}_1,\ldots,\bar{X}_q)$. 
\begin{enumerate}
\item[(i)] Suppose that  $(\beta,\bm B) \in (0,\infty)\times \R^{q-1}$ is such that the function $$H_{\beta,\bm B}(\bm t) := \frac{\beta}{2} \sum_{r=1}^q t_r^2 + \sum_{r=1}^{q} B_r t_r - \sum_{r=1}^q t_r \log t_r$$ (as defined in \eqref{Hdefn6882}) has a unique global maximizer $\bm m$ on the set $\mathcal{P}([q])$.

\item[(ii)] Suppose further that the quadratic form $\bm u^\top \nabla^2 H_{\beta,\bm B}(\bm m) \bm u$ is strictly negative for all $\bm u \in T := \{\bm u \in \R^q\setminus \{\boldsymbol{0}\}: \sum_{r=1}^q u_r = 0\}$.
\end{enumerate}
Then both the random variables
$\sqrt{N}(\bar{\bm X} - \bm m)$  and  $\sqrt{N}(\bm Z - \bm m)$ are tight.
\end{lem}

\begin{proof}
Set $f_r:P([q])\mapsto \R$  as in \eqref{frdef56}, setting $f=(f_1,\cdots,f_q)$, and
 $\xi(\bm t) = \bm t-f(\bm t)$ as in the statement of Lemma \ref{inversefunc}. We now claim that there exists open sets $U\subseteq\R^q$ containing ${\bm m}$, and $V\subseteq \R^q$ containing $\xi({\bm m})={\bm 0}$, such that the function $\xi:U\cap P([q])\mapsto V\cap T^*$ is invertible, with $$\varepsilon:=\inf_{{\bm u}\in U\cap P([q])}\inf_{{\bm v}\in T^*:\|{\bm v}\|_2=1} \|D\xi({\bm u}){\bm v}\|_2>0, $$
where $D\xi({\bm u})$ is the Jacobian of the function $\xi:\R^q\mapsto \R^q$ at ${\bm u}\in \R^q$.
\\

We now complete the proof of the lemma, deferring the proof of the claim. By Lemma \ref{concxbarcwpotts78} we have $\p(\bar{\bm X}\in U)\to 1$, as ${\bm m}\in U$, and on the set $\bar{\bm X}\in U$ we have
\begin{eqnarray*}\label{tightrnxbmmbh}
   \|\sqrt{N}(\bar{\bm X}-\bm m)\|_2 &= &   \left\|\sqrt{N}\left(\xi^{-1}(\xi(\bar{\bm X}))-\xi^{-1}(\boldsymbol{0})\right)\right\|_2\nonumber\\
&\le& \sqrt{N} \|\xi(\bar{\bm X})\|_2  ~\sup_{\bm v \in V}\|D\xi^{-1}(\bm v)\|_2\nonumber\\
&=& \sqrt{N} \|\xi(\bar{\bm X})\|_2~\sup_{\bm u \in U} \|(D\xi(\bm u)|_{T^*})^{-1}\|_2 \le \varepsilon^{-1}\sqrt{N} \|\xi(\bar{\bm X})\|_2.
\end{eqnarray*}
Since $\sqrt{N}\|\xi(\bar{\bm X})\|_2=O_p(1)$ by Lemma \ref{lem:concentration_potts_residual}, tightness of $\sqrt{N}(\bar{\bm X}-{\bar m})$ follows. Tightness of $\sqrt{N}({\bm Z}-{\bm m})$ follows on noting that
$$\Big(\sqrt{N}({\bm Z}-{\bm m})|{\bm X}\Big)\sim N\Big(\sqrt{N}(\bar{\bm X}-{\bm m}), \beta^{-1}\Big).$$

It thus remains to verify the claim, for which we will invoke the Inverse function theorem. We break the proof into the following steps:

\begin{itemize}
    \item {\it $ {\bm m}$ satisfies $\xi({\bm m})={\bm 0}$}.

   A Lagrangian argument gives that the global maximizer $\bm m$ of the function $H_{\beta,{\bm B}}(\cdot)$ satisfies the fixed point equation
\begin{equation}\label{fixedpoint6742}
    m_r = \frac{e^{\beta m_r + B_r}}{\sum_{s=1}^q e^{\beta m_s + B_s}} = f_r(\bm m)~,
\end{equation}
and so we have ${\bm m}=f({\bm m})$, i.e. $\xi({\bm m})={\bm 0}$.
\\

\item {\it $\xi$ maps $P([q])$ into $T^*$, where $T^*$ is as in \eqref{tstardef56}}.

Since $f({\bm t})\in \mathcal{P}([q])$ for all $\bm t \in \mathcal{P}([q])$, we have $\sum_{r=1}^q \xi_r(\bm t)=0$, and hence
$$\xi(\bm t) \in T^* = \{\bm u \in \R^q : \sum_{r=1}^q u_r = 0\}.$$

\item 
{\it  For $\bm u\in \mathcal P([q])$, the  Jacobian operator $D\xi({\bm u})$ induces a linear map}
$$D\xi(\bm u)|_{T^*}:T^*\to T^*.$$

To see this, use \eqref{Jacobian68} to note that for any ${\bm u}\in \R^q$ we have $$D\xi(\bm u) = \bm I -\beta (\mathrm{diag}(\bm f) - \bm f \bm f^\top)$$ where $\bm f := f(\bm u) \in \mathcal{P}([q])$, and hence,
$$\boldsymbol{1}^\top D\xi(\bm u) = \boldsymbol{1}^\top - \beta (\bm f^\top - \boldsymbol{1}^\top \bm f \bm f^\top) = \boldsymbol{1}^\top,$$
which implies that for $\bm v \in T^*$, one has $\boldsymbol{1}^\top D\xi(\bm u) \bm v = \boldsymbol{1}^\top \bm v = 0$. So, $D\xi(\bm u)|_{T^*}$ indeed maps $T^*$ to $T^*$.
\\

\item{\it Completing the proof}

By Lemma \ref{inversefunc}, the linear map $D\xi(\bm m)|_{T^*}:T^*\to T^*$ is injective, which along with the rank-nullity theorem and the previous step shows $D\xi(\bm m)|_{T^*}:T^*\to T^*$ is invertible. By continuity of $D\xi(\bm u)|_{T^*}$ in $\bm u$, there exists a non-empty neighborhood $U'\subseteq \R^q$ of $\bm m$ such that
$$\eps':=\inf_{\bm u\in U'\cap P([q])}\inf_{{\bm v}\in T^*:\|{\bm v}\|_2=1}\|D \xi(\bm u)\bm v\| >0,$$
 By the inverse function theorem applied to the map $\xi:\mathcal P([q])\to T^*$, there exist non-empty neighborhoods $U\subseteq U'$ of $\bm m$ and $V\subseteq T^*$ of $\xi(\bm m) = \boldsymbol{0}$ such that the restriction $\xi : U\cap P([q])\to V\cap T^*$ is a bijection (hence, invertible). This verifies the claim, and hence completes the proof of the Lemma.
\end{itemize}
\end{proof}

\begin{lem}\label{contg7}
   Let $(\beta,\bm B) \in (0,\infty)\times \R^{q-1}$ be such that:
   
   \begin{enumerate}
       \item[(i)] The function $H_{\beta,\bm B}(\bm t) = \frac{\beta}{2} \sum_{r=1}^q t_r^2 + \sum_{r=1}^{q} B_r t_r - \sum_{r=1}^q t_r \log t_r$ (as defined in \eqref{Hdefn6882}) has the unique global maximizer $\bm m$ on the set $\mathcal{P}([q])$,

       \item[(ii)] $\bm u^\top \nabla^2 H(\bm m) \bm u<0\text{  for all }\bm u \in T = \{\bm u \in \R^q\setminus\{{\bf 0}\}: \sum_{r=1}^q u_r = 0\}.$
       \end{enumerate}
   Then, the product measure $\bm m^N := \otimes_{i=1}^N \bm m$ is contiguous to the Curie-Weiss Potts measure $\p_{\beta,\bm B}^{\mathrm{CW}}$.
\end{lem}
\begin{proof}
Suppose that $\bm X$ is sampled from the Curie-Weiss Potts model \eqref{cwpotts1} and conditional on $\bm X$, let $Z_1,\ldots,Z_q$ be independent random variables with $Z_r \sim N(\bar{X}_r, (\beta N)^{-1})$. Let $\bm Z := (Z_1,\ldots,Z_q)$ and $\bar{\bm X} := (\bar{X}_1,\ldots,\bar{X}_q)$. Also, let $\p$ denote the joint distribution of $\bm X$ and $\bm Z$. By Lemma \ref{condindp} and \eqref{frdef56}, we have $\p(X_i=r|\bm Z) =  f_r(\bm Z)$, where
$$f_r(\bm t) :=\frac{e^{\beta t_r + B_r}}{\sum_{s=1}^q e^{\beta t_s + B_s}}. $$

Next, define 
\begin{equation}\label{pxsp76}
W_r := \sqrt{N}(f_r(\bm Z) - f_r(\bm m))~.
\end{equation}
Note that $\|\nabla f_r(\cdot)\|_\infty<\infty$ and $\sqrt{N}(\bm Z-\bm m)$ is tight (by Lemma \ref{c222}), and hence  $W_r = O_\p(1)$. Since $m_r = f_r(\bm m)$ (see \eqref{fixedpoint6742}), we get the following from \eqref{pxsp76}:

$$\p(X_i=r|\bm Z) = m_r + \frac{W_r}{\sqrt{N}}~.$$

Next, define a product measure $\mathbb{Q}$ on $[q]^N\times \mathbb{R}^q$ as
$\mathbb{Q} := \bm m^N \otimes \mu$, where $\mu$ denotes the marginal distribution of $\bm Z$. Setting $T_{N,r} := \left|\{i\in [N]:~X_i = r\}\right|$, the standard central limit theorem gives the following under $\mathbb{Q}$:
$$\frac{T_{N,r} - Nm_r}{\sqrt{N}} = O_{\mathbb{Q}}(1)~.$$ Now, for every $K > 0$, on the intersection of the events $|W_r|\le K$ and $|T_{N,r}-Nm_r|\le K\sqrt{N}$ for all $r\in [q]$, we have:
\begin{eqnarray*}
    \log\frac{\mathbb{Q}(\bm X|\bm Z)}{\p(\bm X|\bm Z)}&=& -\sum_{r=1}^q T_{N,r} \log \left(1+ \frac{W_r}{m_r\sqrt{N}}\right)\\&=& -\sum_{r=1}^q T_{N,r}\left(\frac{W_r}{m_r\sqrt{N}} + O\left(\frac{K^2}{N}\right)\right)\\&=& -\sum_{r=1}^q \frac{T_{N,r} W_r}{m_r\sqrt{N}} + O(K^2)\\&=& \sum_{r=1}^{q-1} W_r\left(\frac{T_{N,q}}{m_q\sqrt{N}} - \frac{T_{N,r}}{m_r\sqrt{N}}\right) + O(K^2)\quad(\text{since}~\sum_{r=1}^q W_r = 0)\\&\le& \sum_{r=1}^{q-1} |W_r| \left(\frac{K}{m_r}+\frac{K}{m_q}\right) + O(K^2)\\&\le& K^2\sum_{r=1}^q \left(\frac{1}{m_r}+\frac{1}{m_q}\right) + O(K^2)=: \phi_K.
\end{eqnarray*}
Thus, if $A_N\subseteq [q]^N$ is a sequence of sets such that $\p(\bm X \in A_N) \rightarrow 0$ as $N \rightarrow \infty$, then:
\begin{eqnarray*}
    && \mathbb{Q}\left(\bm X \in A_N, |W_r|\le K ~\forall r\in [q], |T_{N,r} - N m_r| \le K\sqrt{N} ~\forall r\in [q]\right)\\&=& \e \mathbb{Q}\left(\bm X \in A_N, |W_r|\le K ~\forall r\in [q], |T_{N,r} - N m_r| \le K\sqrt{N} ~\forall r\in [q] \Big|\bm Z\right)\\&\le& e^{\phi_K}\e \mathbb{\p}\left(\bm X \in A_N, |W_r|\le K ~\forall r\in [q], |T_{N,r} - N m_r| \le K\sqrt{N} ~\forall r\in [q] \Big|\bm Z\right)\\&=& e^{\phi_K}\mathbb{\p}\left(\bm X \in A_N, |W_r|\le K ~\forall r\in [q], |T_{N,r} - N m_r| \le K\sqrt{N} ~\forall r\in [q]\right)\\&\le& e^{\phi_K} \p(\bm X\in A_N).
\end{eqnarray*}
Hence, we have:
$$\mathbb{Q}(\bm X\in A_N) \le e^{\phi_K} \p(\bm X\in A_N) + \sum_{r=1}^q \mathbb{Q}(|W_r|>K) + \sum_{r=1}^q \mathbb{Q}(|T_{N,r}-Nm_r| > K\sqrt{N})$$
which on letting $N\rightarrow \infty$ followed by $K \rightarrow \infty$ gives $\mathbb{Q}(\bm X \in A_N) \rightarrow 0$. This completes the proof of Lemma \ref{contg7}.
\end{proof}

\section{Other Technical Lemmas}\label{sec:othertechnle}

In this section, we state additional technical lemmas necessary for proving some of the main results of the paper. We start with the following lemma, which is crucial in establishing existence of the partial MPL estimators $\hat{\beta}_N$ and $\hat{\bm B}_N$.
\begin{lem}\label{prevunp}
Define the sets $A_{2,N}, A_{3,N}, A_{4,N}$
     \begin{align*}A_{2,N} = &\{\bm x \in [q]^N:m_{i,x_i}(\bm x) = \min_{r\in [q]}m_{i,r}(\bm x)~\text{for all}~i\in [N]\},\\
     A_{3,N} = &\{\bm x \in [q]^N: m_{i,x_i}(\bm x) = \max_{r\in [q]}m_{i,r}(\bm x)~\text{for all}~i\in [N]\},\\
   A_{4,N} =& \{\bm x \in [q]^N: ~\text{There exists}~r\in [q]~ \text{such that for all}~i \in [N], ~x_i \ne r\},
   \end{align*}
   as in \eqref{eq:a2n}, \eqref{eq:a3n} and \eqref{eq:a4n} respectively. Then, the following conclusions hold:
\begin{enumerate}
    \item [(a)] If $\bm X \in A_{2,N}^c \bigcap A_{3,N}^c$, then for every $\bm B \in \R^{q-1}$, $\ell_N(\beta,\bm B) \rightarrow -\infty$ as $|\beta| \rightarrow \infty$. 

    \item [(b)] If $\bm X \in A_{4,N}^c$, then for every $\beta\in \R$, $\ell_N(\beta,\bm B) \rightarrow -\infty$ as $\|\bm B\|_\infty \rightarrow \infty$.
\end{enumerate}
   
\end{lem}

\begin{proof}
   To begin with, use \eqref{eq:ln} to note that:
$$\ell_N(\beta,\bm B) =  \sum_{i=1}^N \log(\theta_{i,X_i}),$$ where $\theta_{i,r}$ is as in \eqref{defcondprob881}, and satisfies the inequality
\begin{eqnarray}\label{eq:betair}
\theta_{i,X_i}  = \frac{\exp\{\beta m_{i,X_i}(\bm X) + B_{X_i}\}}{\sum_{t=1}^q \exp\{\beta m_{i,t}(\bm X) + B_t\}}\le \frac{1}{1+\exp\{\beta (m_{i,t}({\bm X})-m_{i,X_i}({\bm X}))+B_t-B_{X_i}\}}
\end{eqnarray}
for all $t\ne X_i$. Since $\theta_{i,r} \le 1$ for all $i,r$, showing $\ell_N(\beta,\bm B) \rightarrow -\infty$ is equivalent to proving that there exists $i \in [N]$ such that $\theta_{i,X_i} \rightarrow 0$. 
\begin{enumerate}
\item[(a)] $|\beta|\to \infty.$
\begin{case}[$\beta\rightarrow -\infty$]
If $\bm X \in A_{2,N}^c$, there exists $i\in [N]$ and $r\in [q]$ such that $m_{i,r}(\bm X) < m_{i,X_i}(\bm X)$ (and so $X_i\ne r$). Then, \eqref{eq:betair} with $t=r$ gives:
$$\theta_{i,X_i} \le \frac{1}{1+ \exp\{\beta(m_{i,r}(\bm X) -m_{i,X_i}(\bm X)) + B_r - B_{X_i}\}}$$ which implies that $\theta_{i,X_i} \rightarrow 0$ as $\beta \rightarrow -\infty$.
\end{case}

\begin{case}[$\beta\rightarrow \infty$]
If $\bm X \in A_{3,N}^c$, there exists $i\in [N]$ and $r\in [q]$ such that $m_{i,r}(\bm X) > m_{i,X_i}(\bm X)$ (and so $X_i\ne r$). Then, \eqref{eq:betair} with $t=r$ gives:
$$\theta_{i,X_i} \le \frac{1}{1+ \exp\{\beta(m_{i,r}(\bm X) -m_{i,X_i}(\bm X)) + B_r - B_{X_i} \}}$$ which implies that $\theta_{i,X_i} \rightarrow 0$ as $\beta \rightarrow \infty$.
\end{case}

\item[(b)] $\|\bm B\|_\infty \rightarrow \infty$

This ensures the existence of an $r\in [q]$ such that $|B_r|\rightarrow \infty$. Also $r\ne q$, as $B_q=0$ by convention.
\begin{case}[$B_r\rightarrow -\infty$]
If $\bm X \in A_{4,N}^c$, there exists $i\in [N]$ such that $X_i = r$, in which case \eqref{eq:betair} with $t=q$ gives:
$$\theta_{i,X_i} \le \frac{1}{1+ \exp\{\beta(m_{i,q}(\bm X) -m_{i,r}(\bm X)) -B_r\}}$$ which implies that $\theta_{i,X_i} \rightarrow 0$ as $B_r \rightarrow -\infty$.
\end{case}

\begin{case}[$B_r\rightarrow \infty$]
If $\bm X \in A_{4,N}^c$, there exists $i\in [N]$ such that $X_i \ne r$ (otherwise the whole vector $\{X_i\}_{1\le i\le N}$ 
 have the same color, which contradicts $A_{4,N}$). Using \eqref{eq:betair} with $t=r$ we have:
$$\theta_{i,X_i} \le \frac{1}{1+ \exp\{\beta(m_{i,r}(\bm X) -m_{i,X_i}(\bm X)) + B_r - B_{X_i} \}}$$ which implies that $\theta_{i,X_i} \rightarrow 0$ as $B_r \rightarrow \infty$.
\end{case}
This completes the proof of Lemma \ref{prevunp}.
\end{enumerate}
\end{proof}

The next result gives a concentration for the vector of conditional probabilities, and will be used to prove Theorem \ref{irrgr}.

\begin{lem}\label{epnt376}
 Suppose $\bm X$ is an observation from the Potts model \eqref{eq:pmf}, where the coupling matrix $\bm A_N$ satisfies the assumptions \eqref{as1}, \eqref{as2} and \eqref{unbddegr}. Let $\mathcal{S}_{N,q}$ denote the set of all $\bm y := ((y_{i,r}))_{i\in [N], r\in [q]} \in [0,1]^{Nq}$, such that $\sum_{r=1}^q y_{i,r} =1$ for all $i\in [N]$. Set the functions $h_N: \mathcal{S}_{N,q}\to \mathbb{R}$ and $I_N: \mathcal{S}_{N,q}\to \mathbb{R}$ as:
$$h_N(\bm y) = \frac{\beta}{2} \sum_{1\le i,j\le N} \sum_{r=1}^q a_{ij} y_{i,r}y_{j,r} + \sum_{i=1}^N \sum_{r=1}^{q} B_r y_{i,r}\quad\text{and}\quad I_N(\bm y) = \sum_{i=1}^N \sum_{r=1}^q y_{i,r}\log y_{i,r}$$ (as in \eqref{hNdef96}), and let $\psi_N(\bm y) = h_N(\bm y) - I_N(\bm y)$ (as in \eqref{psiNdef96}), $\boldsymbol{\theta}(\bm X) = ((\theta_{i,r}(\bm X)))_{i\in [N], r\in [q]}$ and 
    $$\overline{\nabla}(\bm y) = \frac{1}{q}\sum_{r=1}^q \nabla_{\cdot r}(\bm y)$$
(as in \eqref{deftildedelt}) where $\nabla_{\cdot r} (\bm y) := ((\partial \psi_N/\partial y_{i,r}))_{i\in [N]}$ (as in \eqref{2ndstep556}).
Then, we have the following.
    \begin{itemize}
        \item[(a)]~As $N\rightarrow \infty$,
        $$\psi_N({\boldsymbol{\theta}}(\bm X)) = \sup_{\bm y \in \mathcal{S}_{N,q}} \psi_N(\bm y) +o_\p(N). \footnote{By a slight abuse of notation, for a function $\ell:[0,1]^{Nq}\to \mathbb{R}$ and a vector $\bm x \in [q]^N$, we will often refer to $\ell((x_{i,r})_{i\in [N],r\in [q]})$ as $\ell(\bm x)$, where $x_{i,r} :=\mathbbm{1}_{x_i=r}$}$$ 

        \item[(b)]~For all $r\in [q]$, as $N\rightarrow \infty$,
        $$\|\nabla_{\cdot r} (\boldsymbol{\theta}(\bm X))-\onb (\boldsymbol{\theta}(\bm X))\| = o_\p(\sqrt{N}).$$
    \end{itemize}
\end{lem}
\begin{proof}
   (a)~ Choosing $b_{itrs}=\mathbbm{1}_{t=r}$ and $g\equiv 1$ in Lemma \ref{bp12} gives $L_i=1$, and so
   $$\p\left(\Big|\sum_{i=1}^N (X_{ir}-\theta_{i,r}({\bm X}))\Big|\ge \sqrt{Nt}\right)\le 2\exp(-Ct).$$
   Similarly, choosing $b_{itrs}=\mathbbm{1}_{t=r},g(x)=x$ and $\lambda=0$ in Lemma \ref{bp12} gives $L_i=1$, and so
   $$\p\left(\Big|\sum_{i=1}^N (X_{ir}-\theta_{i,r}({\bm X}))m_{i,r}({\bm X})\Big|\ge \sqrt{Nt}\right)\le 2\exp(-Ct).$$
   Fixing $\varepsilon>0$, a union bound over all the colors $r\in [q]$ gives $\p(\bm X \in A_{N,\varepsilon}\bigcap B_{N,\varepsilon})\to 1$, where:
    \begin{align*}A_{N,\varepsilon} :=& \left\{\bm x\in [q]^N: \left|\sum_{i=1}^N (x_{i,r} - \theta_{i,r}(\bm x))\right|\le N\varepsilon~\text{for all}~r\in [q]\right\},\\
    B_{N,\varepsilon} :=&\left\{\bm x\in [q]^N: \left|\sum_{i=1}^N (x_{i,r} - \theta_{i,r}(\bm x)) m_{i,r}(\bm x)\right|\le N\varepsilon~\text{for all}~r\in [q]\right\}.
    \end{align*}
    Also, it follows from the proof of Theorem 1.1 in \cite{mukherjeebasak} (see \cite[Lem 3.2]{mukherjeebasak}), that under Conditions \eqref{as1} and \eqref{unbddegr} we have:
    $$\e\left[\left(h_N(\bm X) - h_N(\theta(\bm X)\right)^2\right] = o(N^2)$$ and hence $\p(\bm X \in C_{N,\varepsilon})\to 1$, where
    \begin{align*}
    C_{N,\varepsilon} :=\left\{\bm x\in [q]^N: \left|h_N(\bm x) - h_N({\theta}(\bm x))\right|\le N\varepsilon\right\}.
    \end{align*}
    This shows that $\p(\bm X \in D_{N,\varepsilon})\to 1,$ where $D_{N,\varepsilon} = A_{N,\varepsilon}\bigcap B_{N,\varepsilon}\bigcap C_{N,\varepsilon}$.  Consequently, setting $$W_{N,\delta} := \left\{\bm x\in [q]^N: \psi_N({\theta}(\bm x)) \le r_N - N\delta\right\}$$
    where $r_N := \sup_{\bm y \in \mathcal{S}_{N,q}} \psi_N(\bm y)$, we have:
    \begin{equation}\label{maintermiq}
        \p(\bm X \in W_{N,\delta}) \le \p(\bm X\in W_{N,\delta}\cap D_{N,\varepsilon})+\p(D_{N,\varepsilon}^c)=
         \frac{\sum_{\bm x \in W_{N,\delta}\cap D_{N,\varepsilon}} e^{h_N(\bm x)}}{\sum_{\bm x \in [q]^N} e^{h_N(\bm x)}} + o(1).
    \end{equation}
    Also, the Gibbs variational principle gives a variational mean field lower bound to the denominator of \eqref{maintermiq} as follows (see, for example,  \cite[Eqn 1.8]{mukherjeebasak}):
    \begin{equation}\label{denom73}
        \sum_{\bm x\in [q]^N} e^{h_N(\bm x)} \ge \sup_{\bm x \in \mathcal{P}([q])^N} e^{\psi_N(\bm x)} = e^{r_N}.
    \end{equation}

    The task now is to bound the numerator of the ratio in the right-hand side of \eqref{maintermiq}. Towards this, define $g_N: [0,1]^{2Nq}\to \mathbb{R}$ as
$$g_N(\bm z,\bm w) := \sum_{i=1}^N \sum_{r=1}^{q} z_{i,r}\log w_{i,r}~.$$
Note that $I_N(\bm y) = g_N(\bm y,\bm y)$. It follows from the proof of \cite[Thm 1.1]{mukherjeebasak} (see \cite[Page 575 last display]{mukherjeebasak}) that:
\begin{align}\label{mbbdddf}
   \notag \left|g_N(\widetilde{\bm X}, \boldsymbol{\theta}(\bm X)) - I_N(\boldsymbol{\theta}(\bm X))\right| = &\left|\sum_{i\in [N],r\in [q]} (X_{i,r}-\theta_{i,r}(\bm X))(\beta m_{i,r}(\bm X) + B_r)\right|\\
    \stackrel{D_{N,\varepsilon}}{\le}& (\beta + B)qN\varepsilon
\end{align}
where $B := \|\bm B\|_\infty$, and for any ${\bm x}\in [q]^N$, we denote $\widetilde{\bm x}:=(x_{i,r})_{i\in [N],r\in [q]}\in \mathcal{S}_{N,q}$. At this point, we need the following definition:
    \begin{definition}
        For $S \subseteq \mathbb{R}^N$ and $\varepsilon >0$, a set $D$ is called an $\varepsilon$-net of $S$, if given any $s\in S$ there exists $d\in D$ such that $\|s-d\|_2\le \varepsilon$.
    \end{definition}

The following result (\cite[Lem 3.4]{mukherjeebasak}) guarantees that under Condition \eqref{unbddegr}, the set $\{\bm A_N \bm v,\bm v\in [0,1]^N\}$ has an $\varepsilon\sqrt{N}$-net of size $e^{o(N)}$.

\begin{proposition}
If $\bm A_N$ satisfies Condition \eqref{unbddegr}, then for every $\varepsilon > 0$, the set $$\{\bm A_N \bm v: \bm v \in [0,1]^N\}$$ has an $\varepsilon\sqrt{N}$-net $J_{N,\varepsilon}$ of cardinality $e^{o(N)}$. 
\end{proposition}

For every $\bm p \in J_{N,\varepsilon}$, define:
$$L_r(\bm p) := \left\{\bm x \in [q]^N: \|\bm p - m_{\cdot r}(\bm x)\|\le \varepsilon\sqrt{N}\right\}$$
where $m_{\cdot r}(\bm x) := (m_{i,r}(\bm x))_{i\in [N]}$, and $m_{ir}({\bm x})=\sum_{j=1}^Na_{ij}\mathbbm{1}_{x_j=r}$. Then, we have:
\begin{eqnarray*}
    &&\sum_{\bm x \in W_{N,\delta} \bigcap D_{N,\varepsilon}}  e^{h_N(\bm x)}\\ &\le_{C_{N,\varepsilon}} & e^{N\varepsilon} \sum_{\bm x \in W_{N,\delta} \bigcap D_{N,\varepsilon}}  e^{h_N({\theta}(\bm x))}\\&\le & e^{N\varepsilon(1+q(\beta+B))} \sum_{\bm x \in W_{N,\delta} \bigcap D_{N,\varepsilon}} e^{h_N(\theta(\bm x)) + g_N(\widetilde{\bm x},\theta(\bm x)) - I_N(\theta(\bm x))}\quad \text{(by \eqref{mbbdddf})}\\&\le& e^{N\varepsilon(1+q(\beta+B))} \sup_{\bm x \in W_{N,\delta}} e^{\psi_N(\theta(\bm x))} \sum_{\bm x \in [q]^N} e^{g_N(\widetilde{\bm x},\theta(\bm x))}\\&\le_{W_{N,\delta}}& e^{N\varepsilon(1+q(\beta+B)) + r_N-N\delta} \sum_{\bm x \in [q]^N} e^{g_N(\widetilde{\bm x},\theta(\bm x))}
    \end{eqnarray*}
    Since  $J_{N,\varepsilon}$ is a $\varepsilon\sqrt{N}$ net of the set $\{\bm A_N {\bm v}, {\bm v}\in [0,1]^N\}$, and $\|\bm A_N{\bm v}\|_\infty\le \gamma$ for all ${\bm v}\in [0,1]^N$ (see \eqref{as1}), by replacing $\varepsilon$ by a factor of $2$ if necessary, without loss of generality we can assume $\|{\bm p}\|_\infty\le \gamma$ for all ${\bm p}\in J_{n,\varepsilon}$.
    Thus, noting that
    $m_{\cdot r}(\bm x) = \bm A_N \bm x_{\cdot r}$ where $x_{ir}=\mathbbm{1}_{x_i=r}$, for every $\bm x\in [q]^N$ and every $r\in [q]$, there exists $\bm p \in J_{N,\varepsilon}$ such that $\|\bm p - m_{\cdot r}(\bm x)\| \le \varepsilon\sqrt{N}$, i.e. $\bm x \in L_r(\bm p)$. 
  For $\bm P := (\bm p_1,\ldots,\bm p_q)\in J_{N,\varepsilon}^q$ we can write:
    \begin{equation}\label{netapx}
       \sum_{\bm x \in W_{N,\delta} \bigcap D_{N,\varepsilon}}  e^{h_N(\bm x)} \le  e^{N\varepsilon(1+q(\beta+B)) + r_N-N\delta} \sum_{\bm P \in J_{N,\varepsilon}^q}\sum_{\bm x\in L(\bm P)} e^{g_N(\widetilde{\bm x},\theta(\bm x))}
    \end{equation}
    where $L(\bm P) := \cap_{r\in [q]} L_r(\bm p_r)$. Next, define the matrix $\bm u(\bm P) := (u_{i,r}(\bm P))_{i\in [n],r\in [q]}$, where:
    $$u_{i,r}(\bm P) := \frac{\exp\{\beta p_{i,r} + B_r\}}{\sum_{t=1}^q \exp\{\beta p_{i,t} + B_t\}}\ge q^{-1}\exp(-\beta \gamma-2B)=\alpha,$$
   where $\alpha$ is as in \eqref{alphadef4}. Since $\theta({\bm x})\ge \alpha$ as well, mean-value theorem gives
    \begin{align*}
    |g_N(\widetilde{\bm x},\theta(\bm x)) - g_N(\widetilde{\bm x},\bm u(\bm P))|\le& C_1\sum_{i=1}^N\sum_{r=1}^q|\theta_{i,r}({\bm x})-u_{i,r}({\bm P})|.
    \end{align*}
    Also observe that $$\theta_{i,r}({\bm x})=f_r(m_{i,1}({\bm x}),\cdots,m_{i,q}({\bm x})),\quad u_{i,r}({\bm P})=f_r(p_{i,1},\cdots,p_{i,q}),$$
    where $$f_r(t_1,\cdots,t_q)=\frac{\exp(\beta t_r+B_r)}{\sum_{s=1}^q \exp(\beta t_s+B_s)}$$
    as in \eqref{frdef56}. Since $\|\nabla f_r\|_\infty<\infty$, another mean value theorem gives
    \begin{align*}
|\theta_{i,r}({\bm x})-u_{i,r}({\bm P})|\le &    C_2\sum_{s=1}^q |m_{i,s}(\bm x)-p_{i,s}|
\end{align*}
Combining the last two bounds we get

\begin{align*}
|g_N(\widetilde{\bm x},\theta(\bm x)) - g_N(\widetilde{\bm x},\bm u(\bm P))|\le&C_1 C_2 q \sum_{i=1}^N\sum_{s=1}^q |m_{i,s}(\bm x)-p_{i,s}|\\
\le &C_1C_2q\sqrt{N}\sum_{s=1}^q \| m_{\cdot s}(\bm x)-\bm p_s \|\le Cq \varepsilon N,
    \end{align*}
   where the last equality uses the fact that $\bm x \in L(\bm P)$, and $C=C_1C_2q$.
   
   Hence, from \eqref{netapx},
   \begin{eqnarray*}
   &&\sum_{\bm x \in W_{N,\delta} \bigcap D_{N,\varepsilon}}  e^{h_N(\bm x)}\\ &\le&  e^{N\varepsilon(1+q(\beta+B+C)) + r_N-N\delta} \sum_{\bm P \in J_{N,\varepsilon}^q}\sum_{\bm x\in L(\bm P)} e^{g_N(\widetilde{\bm x},\bm u(\bm P))}\\&\le& e^{N\varepsilon(1+q(\beta+B+C)) + r_N-N\delta} \sum_{\bm P \in J_{N,\varepsilon}^q}\sum_{\bm x\in [q]^N} e^{g_N(\widetilde{\bm x},\bm u(\bm P))}
   \end{eqnarray*}
    Finally, since $u_{i,r}({\bm P})=v_{i,r}$ satisfies $\sum_{r=1}^q v_{i,r}=1$ for any ${\bm P}\in J_{N,\varepsilon}^q$, we have
    \begin{eqnarray*}
    \sum_{\bm x\in [q]^N} e^{g_N(\widetilde{\bm x},\bm u(\bm P))} &=& \sum_{\bm x\in [q]^N}\exp\left\{\sum_{i=1}^N \sum_{r=1}^q x_{i,r}\log v_{i,r}\right\}\\&=& \sum_{\bm x\in [q]^N}\prod_{i=1}^N \prod_{r=1}^q v_{i,r}^{x_{i,r}}\\&=& \prod_{i=1}^N \left(\sum_{\bm x \in [q]}\prod_{r=1}^q v_{i,r}^{x_r} \right)\\&=&\prod_{i=1}^N \sum_{r=1}^q v_{i,r} = 1~.
    \end{eqnarray*}
    Hence, using the fact that $|J_{N,\varepsilon}|=e^{o(N)}$ we have:
    \begin{eqnarray}\label{num73}
         \sum_{\bm x \in W_{N,\delta} \bigcap D_{N,\varepsilon}}  e^{h_N(\bm x)} &\le& e^{N\varepsilon(1+q(\beta+B+C)) + r_N-N\delta} |J_{n,\varepsilon}|^q\nonumber\\ &=& e^{N\varepsilon(1+q(\beta+B+C)) + r_N-N\delta + qo(N)}
    \end{eqnarray}
   Combining \eqref{denom73} and \eqref{num73}, the above display gives
   \begin{equation*}
       \p(\bm X \in W_{N,\delta}) \le e^{N\varepsilon(1+q(\beta+B+C)) -N\delta + qo(N)} + o(1)
   \end{equation*}
    Since $\varepsilon > 0$ is arbitrary, we conclude  that $\p(\bm X \in W_{N,\delta})=o(1)$, i.e. 
    $$\p\left(\sup_{\bm y\in \mathcal{S}_{N,q}} \psi_N(\bm y) - \psi_N(\theta(\bm X)) \ge N\delta\right) = o(1)$$
    for all $\delta > 0$, which implies that $$\sup_{\bm y\in \mathcal{S}_{N,q}} \psi_N(\bm y) - \psi_N(\theta(\bm X)) = o_\p(N)$$ and completes the proof of part (a) of Lemma \ref{epnt376}.
    \vspace{0.1in}
   
    \noindent(b)~Fixing $\delta >0$, note that it suffices to show the following for all $r\in [q]$:
    $$\p\left(\|\nabla_{\cdot r} (\boldsymbol{\theta}(\bm X))-\overline{\nabla}(\boldsymbol{\theta}(\bm X))\|^2 > N\delta\right) = o(1).$$
    To begin with, use the expression of $\theta_{i,r}({\bm X})$ in \eqref{defcondprob881} to note that for all $i\in [N]$ and $r\in [q]$, $$0<p_1 := \frac{1}{1+(q-1)e^{\beta \gamma+2B}} \le \theta_{i,r}(\bm X) \le \frac{1}{1+(q-1)e^{-\beta \gamma-2B}} =: p_2<1,$$
   where $B := \|\bm B\|_\infty$. Now, suppose that $\bm y \in [p_1,p_2]^{Nq}\cap \mathcal{S}_{N,q}$ is such that $\|\textcolor{black}{\nabla_{\cdot r} }(\bm y)-\onb(\bm y)\| ^2 > N\delta$ for some $r\in [q]$, where $\nabla_{.r}$ and $\bar{\nabla}$ are as in the statement of the lemma. 
       Also set $\widetilde{\nabla} (\bm y) := (\onb(\bm y),\ldots,\onb(\bm y))$ as in \eqref{deftildedelt}. Then, setting $\bm y^{(t)} := \bm y + t \left(\nabla\psi_N(\bm y)- \widetilde{\nabla} (\bm y)\right)$, we claim that $\bm y^{(t)} \in [0,1]^{Nq}$ for $t\in [0,\varepsilon]$ for some fixed $\varepsilon > 0$ not depending on $N$. Towards showing this, note from \eqref{partder4460} that:
    $$\|\nabla \psi_N(\bm y)\|_\infty \le \beta\gamma +B+1-\log p_1$$
    and hence, we can take $\varepsilon := \frac{1}{2}(\beta\gamma +B+1-\log p_1)^{-1}\min\{p_1,1-p_2\}$. Since $\bm y \in \mathcal{S}_{N,q}$, and 
    $$\sum_{r=1}^q \left(\nabla\psi_N(\bm y)- \widetilde{\nabla} (\bm y)\right)_{i,r} = 0\text{ for all }i\in [N]\text{ using \eqref{deftildedelt}},$$
    we also have $\bm y^{(t)} \in \mathcal{S}_{N,q}$.
    Now, a two-term Taylor expansion of the function $t\mapsto \psi_N(\bm y^{(t)})$ for $t\in [0,\varepsilon]$ gives some $\xi\in (0,t)$ satisfying:
   
    \begin{eqnarray*}
   && \psi_N(\bm y^{(t)}) - \psi_N(\bm y)\\ &=& t\|\nabla \psi_N(\bm y)-\widetilde{\nabla}(\bm y)\|^2 +\frac{t^2}{2} \left(\nabla \psi_N(\bm y)-\widetilde{\nabla}(\bm y)\right)^\top \nabla^2\psi_N(\bm y^{(\xi)})\left(\nabla \psi_N(\bm y)-\widetilde{\nabla}(\bm y)\right) \\&\ge& t\|\nabla \psi_N(\bm y)-\widetilde{\nabla}(\bm y)\|^2 -\frac{t^2}{2} \|\nabla \psi_N(\bm y)-\widetilde{\nabla}(\bm y)\|^2 \lambda_{\max}(-\nabla^2\psi_N(\bm  y^{(\xi)}))\\&\ge& t\|\nabla \psi_N(\bm y)-\widetilde{\nabla}(\bm y)\|^2 -\frac{t^2}{2} \|\nabla \psi_N(\bm y)-\widetilde{\nabla}(\bm y)\|^2 \|\nabla^2\psi_N(\bm y^{(\xi)})\|_1 \\
   &\ge& t\|\nabla \psi_N(\bm y)-\widetilde{\nabla}(\bm y)\|^2 -\frac{t^2}{2} \|\nabla \psi_N(\bm y)-\widetilde{\nabla}(\bm y)\|^2 \left(\beta\gamma + \frac{1}{p_1}\right)\\&=& t\|\nabla \psi_N(\bm y)-\widetilde{\nabla}(\bm y)\|^2\left[1-\frac{t}{2}\left(\beta\gamma + \frac{1}{p_1}\right)\right].
    \end{eqnarray*}
    In the above display, the last inequality uses  the fact 
    \begin{eqnarray*}
        \nabla^2\psi_N({\bm z})_{ir,js}=&-\frac{1}{z_{ir}}&\text{ if }i=j\text{ and }r=s,\\
        =&\beta a_{ij}&\text{ if }i\ne j\text{ and }r=s,\\
        =&0&\text{ if }r\ne s
    \end{eqnarray*}
    to get the 
    bound $\|\nabla^2\psi_N(\bm z)\|_1\le \beta \gamma+\frac{1}{p_1}$ uniformly in ${\bm z}\in \mathcal{S}_{N,q}$. Choosing $t\in [0,\varepsilon]$ sufficiently small such that $1-\frac{t}{2}(\beta \gamma + \frac{1}{p_1}) \ge  \frac{1}{2}$, we can thus conclude that:
    $$\psi_N(\bm y^{(t)}) - \psi_N(\bm y) \ge \frac{Nt\delta}{2},\quad\text{i.e.}\sup_{\bm y \in [p_1,p_2]^{Nq}: \|\nabla_{\cdot r} (\bm y)-\onb(\bm y)\|^2 > N\delta} \psi_N(\bm y) \le \sup_{\bm y \in \mathcal{S}_{N,q}} \psi_N(\bm y) - \frac{Nt\delta}{2}.$$
    Therefore, on the event $F_r:= \{\|\nabla_{\cdot r} (\boldsymbol{\theta}(\bm X))-\onb (\boldsymbol{\theta}(\bm X))\|^2 > N\delta\}$, we have:
    $$\psi_N(\boldsymbol{\theta}(\bm X)) \le \sup_{\bm y \in \mathcal{S}_{N,q}} \psi_N(\bm y) - \frac{Nt\delta}{2}.$$
    Therefore, we have by part (a),
    $$\p(F_r) \le \p\left(\sup_{\bm y \in \mathcal{S}_{N,q}} \psi_N(\boldsymbol{\theta}(\bm X)) - \psi_N(\boldsymbol{\theta}(\bm X)) \ge \frac{Nt\delta}{2}\right) = o(1)$$ which completes the proof of part (b).
    
\end{proof}

The following is a technical lemma needed in the proof of Theorem \ref{partialestm}.
   
    \begin{lem}\label{c7y}
        Suppose that $w_1,\ldots,w_q \in (\alpha,1]$ with $\alpha >0$, and $\sum_{r=1}^q w_r =1$. Then, for positive real numbers $t_1,\ldots,t_q$ bounded above by $\gamma$, we have:
        $$\max_r t_r - \sum_r w_r t_r \ge \frac{\alpha}{q(q-1)\gamma} \sum_{r<s} (t_r-t_s)^2~.$$
    \end{lem}
    \begin{proof}
        Without loss of generality, suppose that $t_1=\max_r t_r$ and $t_2 = \min_r t_r$. Then,
        \begin{eqnarray*}
        \max_r t_r - \sum_r w_r t_r &=& t_1(1-w_1) - t_2w_2 - \sum_{r\ge 3} t_rw_r\\&\ge& t_1\left (1-w_1-\sum_{r\ge 3} w_r\right) - t_2 w_2\\&=& t_1 w_2 - t_2 w_2\\ &\ge& \alpha |t_r-t_s|
        \end{eqnarray*}
        for every $r,s$. Hence, using the fact $\max_{r<s} |t_r-t_s| \le 2\gamma$, we have 
        \begin{eqnarray*}
            \max_r t_r - \sum_r w_r t_r &\ge& \frac{2\alpha}{q(q-1)} \sum_{r<s} |t_r-t_s|\\&\ge& \frac{\alpha}{q(q-1)\gamma} \sum_{r<s} (t_r-t_s)^2
        \end{eqnarray*}
        This completes the proof of Lemma \ref{c7y}.
    \end{proof}

\end{appendix}

\end{document}